\documentclass[reqno,10pt]{amsart}
\usepackage{amssymb,amsmath,amsthm,}
\usepackage{amscd}
\usepackage{amsfonts}
\usepackage{amssymb}
\usepackage{latexsym}
\usepackage{color}
\usepackage{esint}
\usepackage{graphicx}
\usepackage{float}
\graphicspath{{Figures/}}

\usepackage{graphicx}

\usepackage[makeroom]{cancel}

\newtheorem{theorem}{Theorem}

\newtheorem{corollary}[theorem]{Corollary}
\newtheorem{definition}{Definition}
\newtheorem{lemma}{Lemma}
\newtheorem{proposition}[theorem]{Proposition}
\newtheorem{remark}{Remark}

\let\e=\varepsilon

\let\O=\Omega

\numberwithin{equation}{section}

\let\hide\iffalse

\newcommand{\R}{\mathbb{R}}

\newcommand{\be}{\begin{equation}}
\newcommand{\bm}{\begin{multline}}
\newcommand{\ee}{\end{equation}}

\newcommand{\Bes}{\begin{eqnarray*}}
\newcommand{\Ees}{\end{eqnarray*}}
\newcommand{\Be}{\begin{equation} }
\newcommand{\Ee}{\end{equation}}

\def\O{\Omega}
\def\R{\mathbb{R}}

\def\B{\begin{equation}}
\def\E{\end{equation}}
\def\BN{\begin{eqnarray*}}
\def\EN{\end{eqnarray*}}

\begin{document}
\title{Cercignani-Lampis boundary in the Boltzmann theory}

\author{Hongxu Chen}
\address{Mathematics Department, University of Wisconsin-Madison, 480 Lincoln Dr., Madison, WI 53705 USA.}
\email{hchen463@wisc.edu}

 \begin{abstract}
The Boltzmann equation is a fundamental kinetic equation that describes the dynamics of dilute gas. In this paper we study the local well-posedness of the Boltzmann equation in bounded domain with the Cercignani-Lampis boundary condition, which describes the intermediate reflection law between diffuse reflection and specular reflection via two accommodation coefficients. We prove the local-in-time well-posedness of the equation by establishing an $L^\infty$ estimate. In particular, for the $L^\infty$ bound we develop a new decomposition on the boundary term combining with repeated interaction through the characteristic. Via this method, we construct a unique steady solution of the Boltzmann equation with constraints on the wall temperature and the accommodation coefficient.
 \end{abstract}

\maketitle

\section{Introduction}
%\subsection{Problem setting}
In this paper we consider the classical Boltzmann equation, which describes the dynamics of dilute particles. Denoting $F(t,x,v)$ the phase-space-distribution function of particles at time $t$, location $x\in\Omega$ moving with velocity $v\in\mathbb{R}^3$, the equation writes:
\begin{equation}\label{eqn: VPB equation}
\partial_t F + v\cdot \nabla_x F = Q(F,F)\,.
\end{equation}
The collision operator $Q$ describes the binary collisions between particles:
\begin{equation}\label{eqn: Q}
\begin{split}
   & Q(F_1,F_2)(v)=Q_{\text{gain}}-Q_{\text{loss}} = Q_{\text{gain}}(F_1,F_2)-\nu(F_1)F_2 \\
    & :=\iint_{\mathbb{R}^3\times \mathbb{S}^2} B(v-u,\omega)F_1(u')F_2(v') d\omega du-F_2(v)\left(\iint_{\mathbb{R}^3\times \mathbb{S}^2} B(v-u,\omega)F_1(u) d\omega du\right)\,.
\end{split}
\end{equation}
In the collision process, we assume the energy and momentum are conserved. We denote the post-velocities:
\begin{equation}\label{eqn: u' v'}
u'=u-[(u-v)\cdot \omega]\omega ,\quad \quad v'=v+[(u-v)\cdot \omega]\omega\,,
\end{equation}
then they satisfy:
\begin{equation}\label{eqn: conservation}
  u'+v'=u+v\,,\quad |u'|^2+|v'|^2=|u|^2+|v|^2\,.
\end{equation}
In equation~\eqref{eqn: Q}, $B$ is called the collision kernel which is given by
\[
B(v-u,\omega)=|v-u|^{\mathcal{K}}q_0(\frac{v-u}{|v-u|}\cdot \omega)\,,\quad\text{with}\quad -3< \mathcal{K}\leq 1\,,\quad 0\leq q_0(\frac{v-u}{|v-u|}\cdot \omega)\leq C\Big|\frac{v-u}{|v-u|}\cdot \omega\Big|\,.
\]

To describe the boundary condition for $F$, we denote the collection of coordinates on phase space at the boundary:
\[\gamma:=\{(x,v)\in \partial \Omega\times \mathbb{R}^3\}.\]
And we denote $n=n(x)$ as the outward normal vector at $x\in \Omega$. We split the boundary coordinates $\gamma$ into the incoming ($\gamma_-$) and the outgoing ($\gamma_+$) set:
\[\gamma_\mp:=\{(x,v)\in \partial \Omega\times \mathbb{R}^3 :n(x)\cdot v\lessgtr 0\}\,.\]
The boundary condition determines the distribution on $\gamma_-$, and shows how particles back-scattered into the domain. In our model, we use the scattering kernel $R(u \rightarrow v;x,t)$:
\begin{equation}\begin{split}\label{eqn:BC}
&F(t,x,v) |n(x) \cdot v|= \int_{n(x) \cdot u>0}
R(u \rightarrow v;x,t) F(t,x,u)
\{n(x) \cdot u\} d u, \quad \text{ on }\gamma_-\,.
\end{split}
\end{equation}
Physically, $R(u\to v;x,t)$ represents the probability of a molecule striking in the boundary at $x\in\partial\Omega$ with velocity $u$, and to be sent back to the domain with velocity $v$ at the same location $x$ and time $t$. There are many models for it. In~\cite{CIP,CL} Cercignani and Lampis proposed a generalized scattering kernel that encompasses pure diffusion and pure reflection molecules via two accommodation coefficients $r_\perp$ and $r_\parallel$. Their model writes:
\begin{equation}\label{eqn: Formula for R}\begin{split}
&R(u \rightarrow v;x,t)\\
:=& \frac{1}{r_\perp r_\parallel (2- r_\parallel)\pi/2} \frac{|n(x) \cdot v|}{(2T_w(x))^2}
\exp\left(- \frac{1}{2T_w(x)}\left[
\frac{|v_\perp|^2 + (1- r_\perp) |u_\perp|^2}{r_\perp}
+ \frac{|v_\parallel - (1- r_\parallel ) u_\parallel|^2}{r_\parallel (2- r_\parallel)}
\right]\right)\\
& \times  I_0 \left(
 \frac{1}{2T_w(x)}\frac{2 (1-r_\perp)^{1/2} v_\perp u_\perp}{r_\perp}
\right),
\end{split}
\end{equation}
where $T_w(x)$ is the wall temperature for $x\in \partial \Omega$ and
\begin{equation*}
I_0 (y) := \pi^{-1} \int^{\pi}_0e^{y \cos \phi } d \phi\,.
\end{equation*}
In the formula, $v_\perp$ and $v_\parallel$ denote the normal and tangential components of the velocity respectively:
   \begin{equation}\label{eqn: def of vperppara}
   v_\perp= v\cdot n(x)\,,\quad v_\parallel = v- v_\perp n(x)\,.
\end{equation}
Similarly $u_\perp= u\cdot n(x)$ and $u_\parallel = u- u_\perp n(x)$.

There are a few properties the Cercignani-Lampis(C-L) model satisfies, including:
\begin{itemize}
\item the reciprocity property:
  \begin{equation}\label{eqn: reciprocity}
    R(u\to v;x,t)=R(-v\to -u;x,t) \frac{e^{-|v|^2/(2T_w(x))}}{e^{-|u|^2/(2T_w(x))}}\frac{|n(x)\cdot v|}{|n(x)\cdot u|}\,,
  \end{equation}
\item the normalization property(see the proof in appendix)
\begin{equation}\label{eqn: normalization}
\int_{n(x)\cdot v<0} R(u\to v;x,t) dv=1\,.
\end{equation}
\end{itemize}

The normalization~\eqref{eqn: normalization} property immediately leads to null-flux condition for $F$:
\begin{equation}\label{eqn: Null flux condition}
  \int_{\mathbb{R}^3}F(t,x,v)\{n(x)\cdot v\}dv=0\,,\quad \text{for }x\in \partial\Omega.
\end{equation}
This condition guarantees the conservation of total mass:
\begin{equation}\label{eqn: Mass conservation}
  \int_{\Omega\times \mathbb{R}^3}F(t,x,v)dvdx=\int_{\Omega\times \mathbb{R}^3}F(0,x,v)dvdx \text{ for all }t\geq 0\,.
\end{equation}

\begin{remark}The C-L model is an extension of the following classical diffuse boundary condition. The distribution function and scattering kernel are given by:

  %Diffuse boundary condition:
\begin{equation}\label{eqn: diffuse}
 F(t,x,v)= \frac{2}{\pi (2T_w(x))^2}e^{-\frac{|v|^2}{2T_w(x)}}\int_{n(x)\cdot u>0} F(t,x,u)\{n(x)\cdot u\}du \text{ on }(x,v)\in\gamma_-,
\end{equation}
\[R(u\to v;x,t)=\frac{2}{\pi (2T_w(x))^2}e^{-\frac{|v|^2}{2T_w(x)}} |n(x)\cdot v|.\]
It corresponds to the scattering kernel in~\eqref{eqn: Formula for R} with $r_\perp=1,r_\parallel=1$.

Other basic boundary conditions can be considered as a special case with singular $R$: specular reflection boundary condition:
\[F(t,x,v)=F(t,x,\mathfrak{R}_xv) \text{ on }(x,v)\in \gamma_-,  \quad \mathfrak{R}_xv=v-2n(x)(n(x)\cdot v),\]
\[R(u\to v;x,t)=\delta(u-\mathfrak{R}_xv),\]
where $r_\perp=0,r_\parallel=0$.

Bounce-back reflection boundary condition:
\[ F(t,x,v)=F(t,x,-v) \text{ on } (x,v)\in \gamma_-,\]
\[R(u\to v;x,t)=\delta(u+v),\]
where $r_\perp=0,r_\parallel=2$.

\end{remark}

Here we mention the Maxwell boundary condition, which is another classical model describes the intermediate reflection law. The scattering kernel is given by the convex combination of the diffuse and specular scattering kernel:
\[R(u\to v)= c\frac{2}{\pi (2T_w(x))^2}e^{-\frac{|v|^2}{2T_w(x)}} |n(x)\cdot v|+ (1-c)\delta(u-\mathfrak{R}_xv),\quad      0\leq c\leq 1.\]
Compared with the C-L boundary condition, the Maxwell boundary condition does not cover the combination with the bounce back boundary condition. Such combination is covered in the C-L boundary condition with $r_\parallel>1$. Moreover, the C-L boundary condition represents a smooth transition from the diffuse to the specular. The Maxwell boundary condition represents the convex combination of the Maxwellian and the dirac $\delta$ function. Here we show the graphs for both boundary condition in the two dimension for comparison. We assume the particles are moving towards the boundary with velocity $u=(u_\parallel,u_\perp)=(2,-2)$, thus
the boundary condition is given by
\[\big[F(t,x,v)|n(x)\cdot v|\big]\Big|_{\gamma_-}=\int_{n(x)\cdot u>0}R(u\to v) \delta\Big(u-(2,-2)\Big) |n(x)\cdot u|du.\]
Then the distribution function $F(t,x,v)|_{\gamma_-}$ for both boundary condition can be viewed as the following graphs:
\begin{figure}[H]
\begin{minipage}[t]{0.45\linewidth}
\centering
\includegraphics[width=\textwidth]{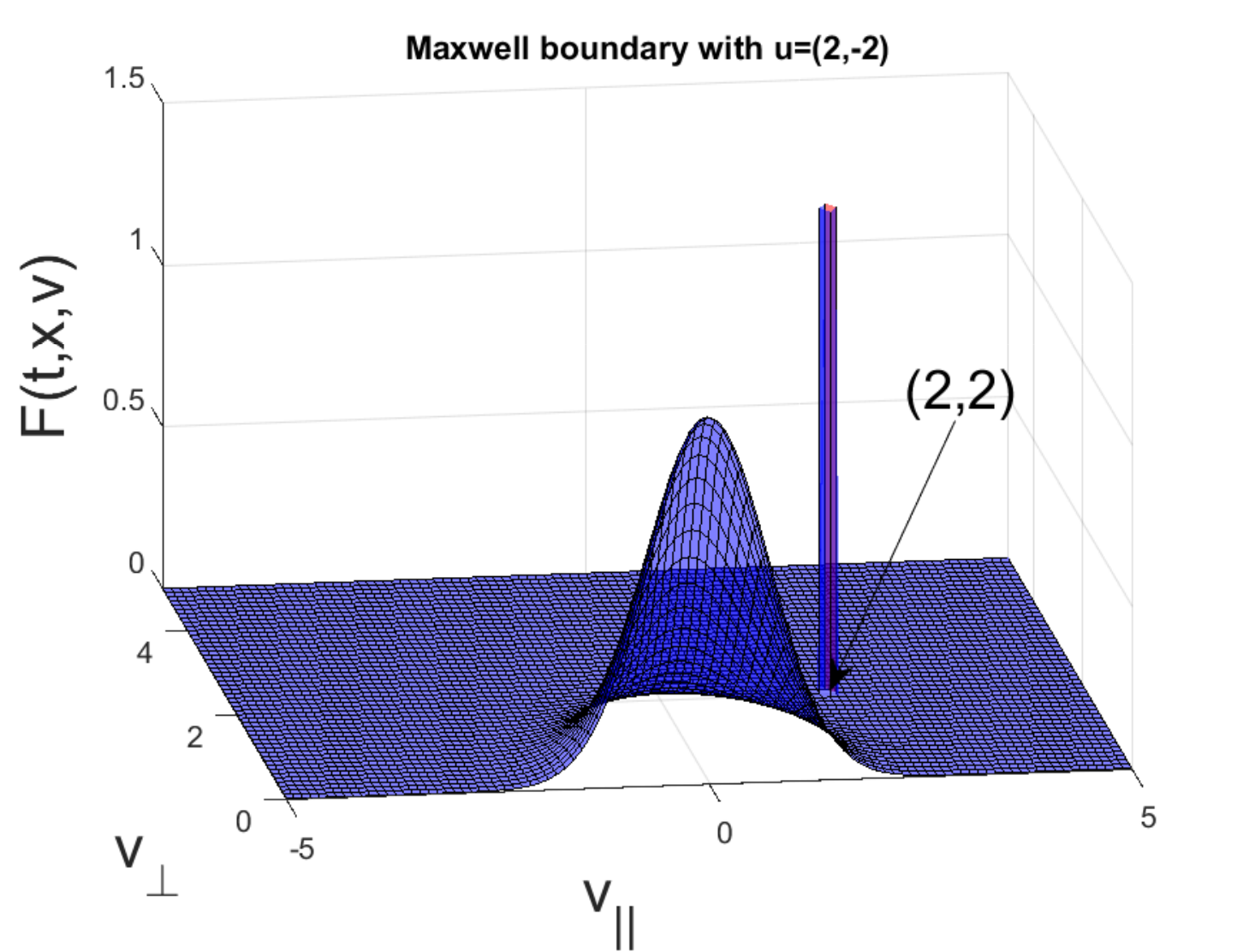}
\caption{Maxwell boundary condition with $c=1/2$.}
\label{fig:Maxwell}
\end{minipage}
\begin{minipage}[t]{0.45\linewidth}
\centering
\includegraphics[width=\textwidth]{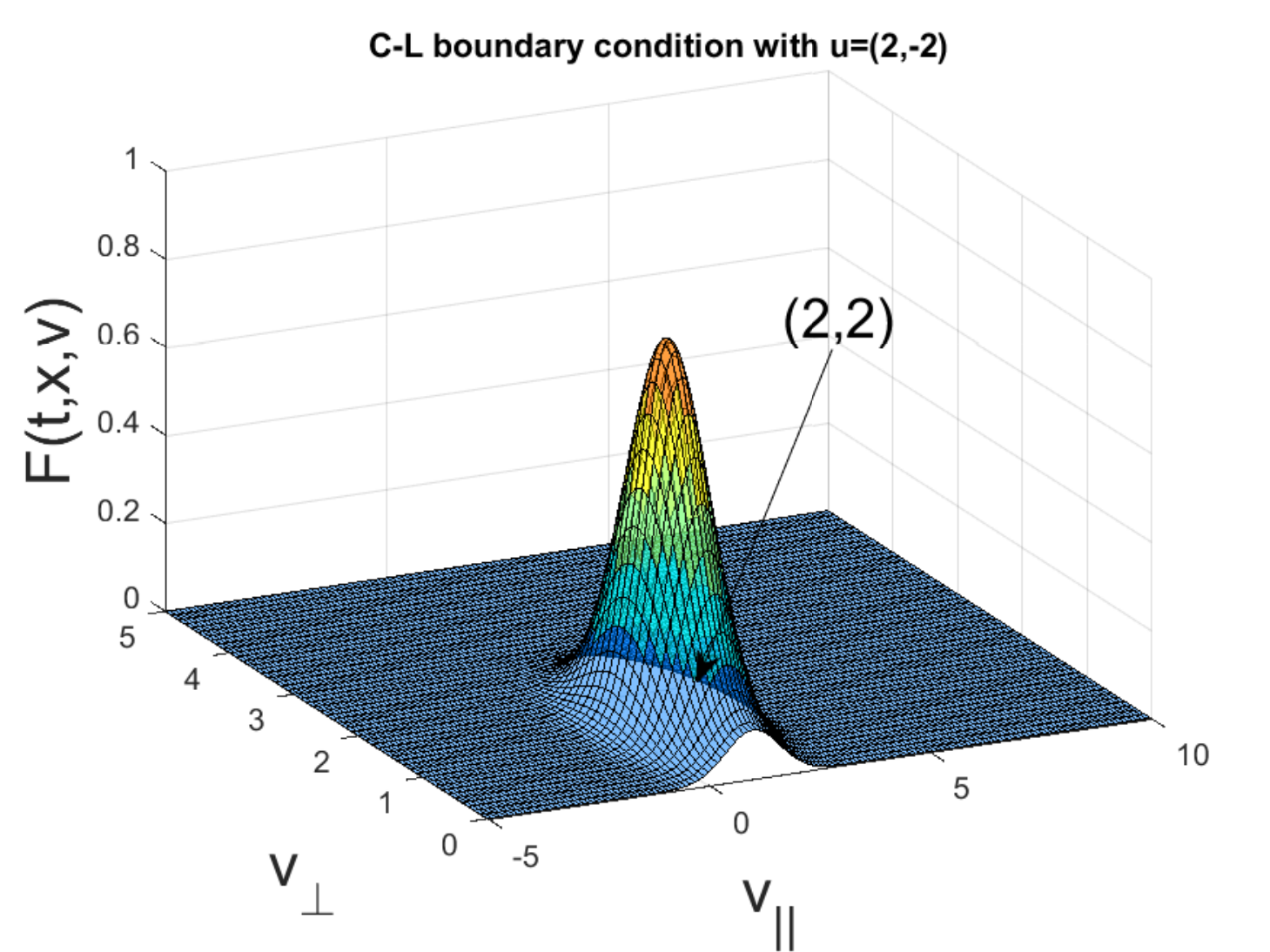}
\caption{C-L boundary condition with $r_\perp=r_\parallel=1/2$.}
\label{fig:C-L0.5}
\end{minipage}
\end{figure}
Moreover, we show the graphs for the distribution function $F|_{\gamma_-}$ with C-L boundary condition with smaller accommodation coefficients.
\begin{figure}[H]
\begin{minipage}[t]{0.45\linewidth}
\centering
\includegraphics[width=\textwidth]{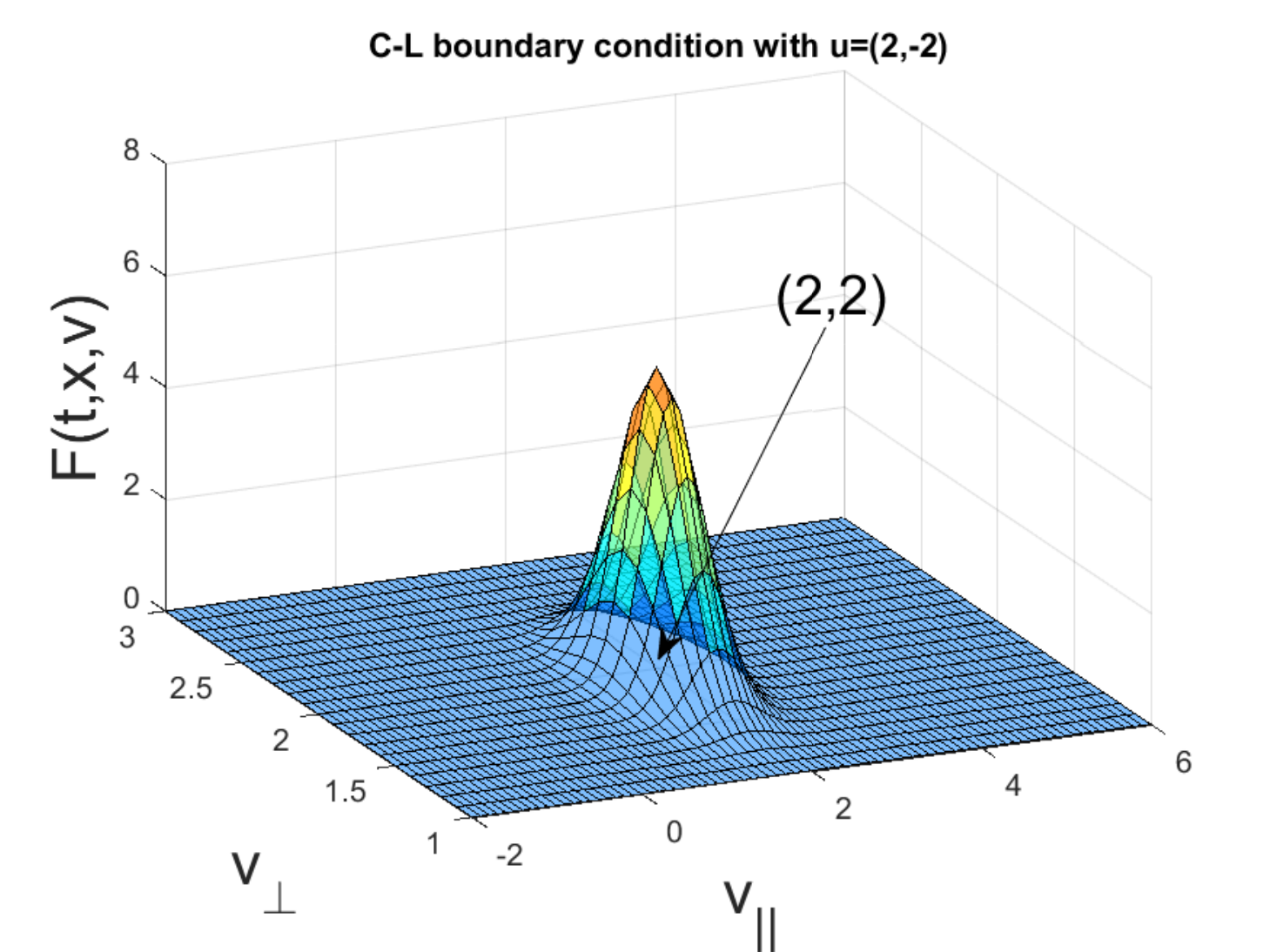}
\caption{C-L boundary condition with $r_\perp=r_\parallel=1/10$.}
\label{fig:C-L0.1}
\end{minipage}
\begin{minipage}[t]{0.45\linewidth}
\centering
\includegraphics[width=\textwidth]{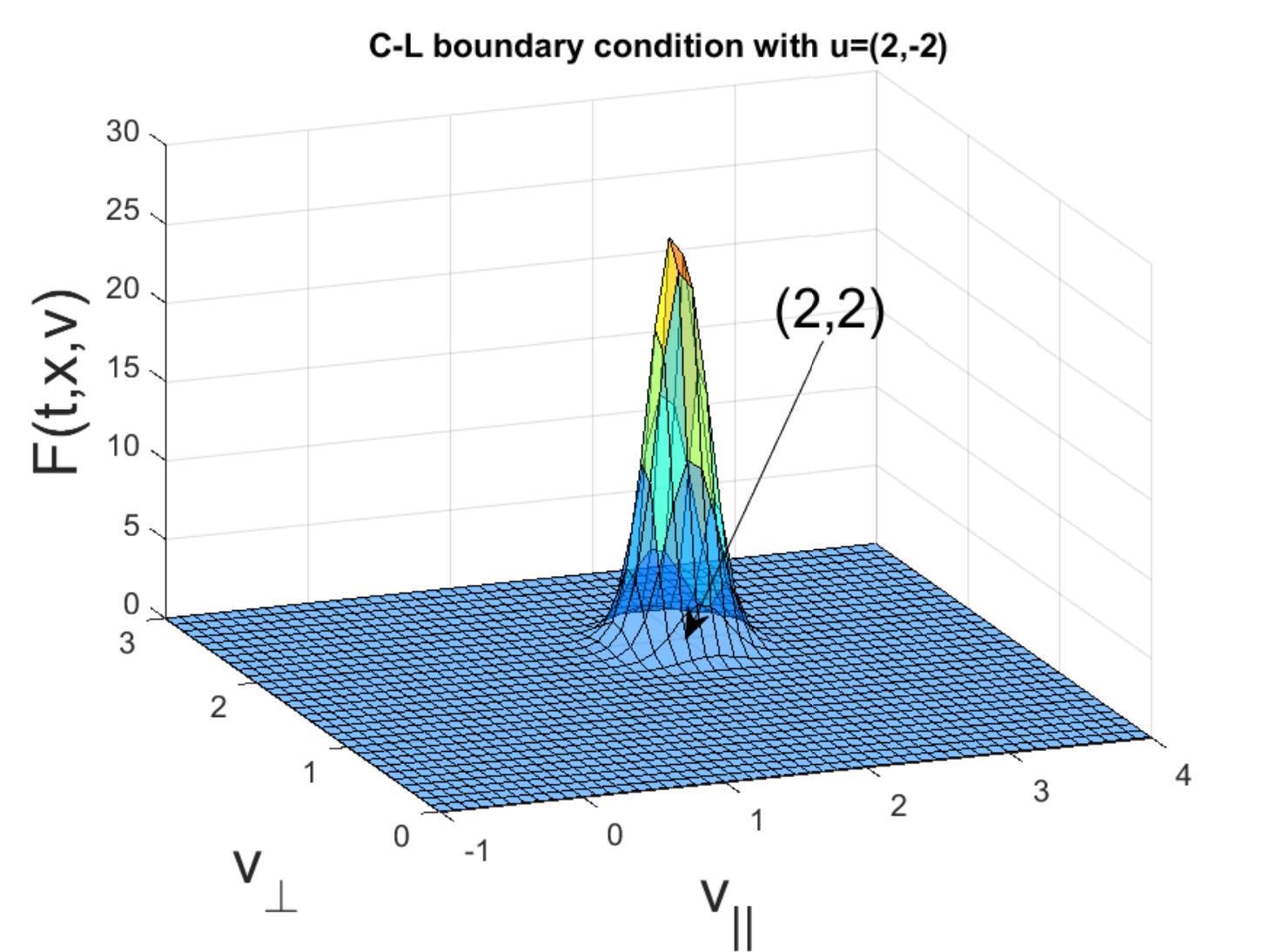}
\caption{C-L boundary condition with $r_\perp=r_\parallel=1/30$.}
\label{fig:C-L0.03}
\end{minipage}
\end{figure}
Figure \ref{fig:C-L0.5} shows a smoother transition since the particles begin to concentrate toward to the point $(2,2)$. Meanwhile Figure \ref{fig:Maxwell} represents the phenomena that half particles are specular reflected and half particles are diffusive. When we take smaller accommodation coefficient, Figure \ref{fig:C-L0.1} and Figure \ref{fig:C-L0.03} demonstrate that the distribution function $F(t,x,v)|_{\gamma_-}$ gradually concentrate on $(2,2)$. Moreover, the $z$-coordinate shows that the C-L scattering kernel indeed tends to a dirac $\delta$ function as the accommodation coefficients become smaller.

Due to the generality of the C-L model, it has been vastly used in many applications. There are other derivations of C-L model besides the original one, and we refer interested readers to~\cite{C,CIP,CC}. Also there have been many application of this model in recent years, on the rarefied gas flow in~\cite{KB,KB2,SF1,SF2,SF3}; extension to the gas surface interaction model in fluid dynamics~\cite{L,L2,WR}; on the linearized Boltzmann equation in~\cite{Gar,SI,LS,CS}; on S-model kinetic equation in~\cite{CES} etc.

%Extended model to cover more properties by~\cite{} and special model for diatomic rarefied gas flows based on the C-L kernel by~\cite{}.

%\ql{what is this sentence for...?}There has been a lot of study in the Boltzmann equation,
%

\subsection{Main result}

We assume that the domain is $C^2$. Denote the maximum wall temperature:
\begin{equation}\label{eqn: T}
T_M:=\max\{T_w(x)\}<\infty\,.
\end{equation}
Define the global Maxwellian using the maximum wall temperature:
\begin{equation}\label{eqn: def for weight}
\mu:=e^{-\frac{|v|^2}{2T_M}}\,,
\end{equation}
and weight $F$ with it: $F=\sqrt{\mu}f$, then $f$ satisfies
\begin{equation}\label{equation for f}
  \partial_t f+v\cdot \nabla_x f=\Gamma(f,f)\,,
\end{equation}
where the collision operator becomes:
\begin{equation}\label{Def: Gamma}
\Gamma(f_1,f_2)=\Gamma_{\text{gain}}(f_1,f_2)-\nu(F_1)F_2/\mu=\frac{1}{\sqrt{\mu}}Q_{\text{gain}}(\sqrt{\mu}f_1,\sqrt{\mu}f_2)-\nu(F_1)f_2\,.
\end{equation}

By the reciprocity property~\eqref{eqn: reciprocity}, the boundary condition for $f$ becomes, for $(x,v)\in \gamma_-$,
\[f(t,x,v)|n(x)\cdot v|=\frac{1}{\sqrt{\mu}}\int_{n(x)\cdot u>0}   R(u\to v;x,t)  f(t,x,u)\sqrt{\mu(u)}\{n(x)\cdot u\}du\]
\[=\frac{1}{\sqrt{\mu}}\int_{n(x)\cdot u>0}   R(-v\to -u;x,t) \frac{e^{-|v|^2/(2T_w(x))}}{e^{-|u|^2/(2T_w(x))}}  f(t,x,u)\sqrt{\mu(u)}\frac{|n(x)\cdot v|}{|n(x)\cdot u|}\{n(x)\cdot u\}du.\]
Thus
\begin{equation}\label{eqn:C-L boundary condition in pro measure}
f(t,x,v)|_{\gamma_-}=e^{[\frac{1}{4T_M}-\frac{1}{2T_w(x)}]|v|^2}\int_{n(x)\cdot u>0} f(t,x,u)e^{-[\frac{1}{4T_M}-\frac{1}{2T_w(x)}]|u|^2}d\sigma(u,v).
\end{equation}
Here we denote
\begin{equation}\label{eqn:probability measure}
d\sigma(u,v):=R(-v\to -u;x,t)du,
\end{equation}
the probability measure in the space $\{(x,u),n(x)\cdot u>0\}$ (well-defined due to~\eqref{eqn: normalization}).

Denote
\begin{equation}\label{Def: w_theta}
  w_{\theta}:=e^{\theta |v|^2},
\end{equation}
\begin{equation}\label{eqn: langle v rangle}
  \langle v\rangle:=\sqrt{|v|^2+1}.
\end{equation}

	\begin{theorem}\label{local_existence}

Assume $\Omega \subset \mathbb{R}^3$ is bounded and $C^2$. Let $0< \theta <\frac{1}{4T_M}$. Assume
\begin{equation}\label{eqn: r condition}
  0<r_\perp\leq 1,\quad 0<r_\parallel<2\, ,
\end{equation}
\begin{equation}\label{eqn: Constrain on T}
\frac{\min(T_w(x))}{T_M}>\max\Big(\frac{1-r_\parallel}{2-r_\parallel},\frac{\sqrt{1-r_\perp}-(1-r_\perp)}{r_\perp}\Big),
\end{equation}
where the $T_M$ is defined in~\eqref{eqn: T}.

If $F_0= \sqrt{\mu}f_0\geq 0$ and $f_0$ satisfies the following estimate:
\begin{equation}\label{eqn: w f_0}
\| w_\theta f_0 \|_\infty		< \infty,
\end{equation}
then there exists a unique solution $F(t,x,v) =  \sqrt{\mu} f(t,x,v)\geq 0$ to~\eqref{eqn: VPB equation} and~\eqref{eqn:BC} in $[0, t_\infty] \times \O \times \R^3$ with
\[t_\infty=t_\infty(\Vert w_\theta f_0\Vert_\infty,r_\perp,r_\parallel,\theta,T_M,\min\{T_w(x)\},\Omega).\]
Moreover, the solution $F=\sqrt{\mu}f$ satisfies
		\Be\begin{split}\label{infty_local_bound} \sup_{0 \leq t \leq t_\infty}
			\| w_{\theta}e^{-|v|^2 t} f  (t) \|_{\infty}
			\lesssim \| w_\theta f_0 \|_\infty .
		\end{split}\Ee

		\end{theorem}
		
\begin{remark}
In Theorem \ref{local_existence} the accommodation coefficient can be any number that does not correspond to the dirac $\delta$ case. Also we cover all the range for $\mathcal{K}$ in the collision kernel $B$ in~\eqref{eqn: Q}. We derive~\eqref{infty_local_bound} and existence using the sequential argument. Assumption~\eqref{eqn: w f_0} is used to obtain the estimate~\eqref{infty_local_bound} for the sequence solution, which is the key factor to prove the theorem.

\end{remark}

\begin{remark}
There has been a lot of studies for Boltzmann equation in many aspects, the global solution~\cite{G,G2,CKL}; regularity estimate~\cite{GKTT,GKTT2}; the steady solution~\cite{EGKM,EGKM2,Duan}.

So far we are only able to prove the local well-posedness with the C-L boundary condition. There are several obstacles to construct the global solution with the C-L boundary condition for arbitrary accommodation coefficient.

To obtain the global solution of the Boltzmann equation~\cite{G} developed the $L^2-L^\infty$ bootstrap and derive the time decay and continuous solution of the linearized Boltzmann equation with various boundary condition. In particular, for the diffuse boundary condition with constant wall temperature,~\cite{G} used the $L^2$ estimate on the boundary
\[\int_{n(x)\cdot u<0}f^2(t,x,u)|_{\gamma_-} |n(x)\cdot u|du\leq \int_{n(x)\cdot u>0}f^2(t,x,u)|n(x)\cdot u|du, \]
\begin{equation}\label{eqn: diffuse f}
\text{ with }f|_{\gamma_-}=c_\mu\sqrt{\mu}\int_{n(x)\cdot u>0}f(t,x,u)\sqrt{\mu}|n(x)\cdot u|du.
\end{equation}
Here $c_\mu$ is the normalization constant such that $c_\mu \sqrt{\mu}|n\cdot u|du$ is a probability measure. To be more specific, the diffuse boundary condition can be regarded as a projection $P_\gamma f=f|_{\gamma_-}$. Then
\[\int_{n(x)\cdot u>0}(f-P_\gamma f)^2|n(x)\cdot u|du=\int_{n(x)\cdot u>0}f^2 |n(x)\cdot u|du-\int_{n(x)\cdot u>0}P_\gamma f^2 |n(x)\cdot u|du\geq 0.\]
However, for the C-L boundary condition, such $L^2$ inequality does not work. We can not regard the boundary condition~\eqref{eqn:C-L boundary condition in pro measure} as a projection because of the new probability measure $d\sigma(u,v)$ in~\eqref{eqn:probability measure}.

Another method to obtain the global solution is to use the entropy inequality. \cite{G2} used the entropy inequality and the $L^1-L^\infty$ bootstrap to derive the bounded solution of the linearized Boltzmann equation with periodic boundary condition. To adapt the entropy method in bounded domain,~\cite{M} used the Jensen inequality for the \textit{Darroz\`{e}s-Guiraud information} with Maxwell boundary condition. To be more specific, we define $\mathcal{E}$ as the \textit{Darroz\`{e}s-Guiraud information}:
\[\mathcal{E}:=\int_{\gamma_+}h\Big(\frac{F}{c_\mu\mu}\Big)c_\mu\mu(u)|n(x)\cdot u|du-h\Big(\int_{\gamma_+}\frac{F}{c_\mu\mu}c_\mu\mu(u) |n(x)\cdot u|du \Big),\quad h(s)=s\log s.\]
Since $c_\mu\mu(u)|n(x)\cdot u|du$ is a probability measure then $\mathcal{E}\geq 0$ by the Jensen inequality and thus the entropy inequality follows. For the C-L boundary condition, such inequality does not work since the probability measure is given by $d\sigma(u,v)$~\eqref{eqn:probability measure}, which is different from $c_\mu\mu(u)|n(x)\cdot u|du$.
f
Even though the global solution is not available for arbitrary accommodation coefficient, we are able to construct the steady and global solution when the coefficients are closed to $1$. This means the we require the boundary condition to be closed to the diffuse boundary condition. We will discuss the steady solution in the following section.

\end{remark}

%\begin{remark}
%As far as the authors know, Theorem 1 provides the \textit{first} local in time result solution to the Vlasov-Poission-Boltzmann system in bounded domains with the generalized diffuse boundary condition. Moreover the result can be easily applied to the Boltzmann equation without field, which also serves as the first result for the Boltzmann equation.
%\end{remark}		

\subsection{}
Beside the local-in-time well-posedness, we can establish the stationary solution under some constraints. The steady problem is given as
\begin{equation}\label{eqn: Steady Boltzmann}
  v\cdot \nabla_x F=Q(F,F),\quad (x,v)\in \Omega \times \mathbb{R}^3
\end{equation}
with $F$ satisfying the C-L boundary condition.

We use the short notation $\mu_0$ to denote the global Maxwellian with temperature $T_0$,
\[\mu_0:=\frac{1}{2\pi (T_0)^2}\exp\Big(-\frac{|v|^2}{2T_0} \Big).\]
Denote $L$ as the standard linearized Boltzmann operator
\begin{equation}\label{eqn: L operator}
\begin{split}
   & Lf:=-\frac{1}{\sqrt{\mu_0}}\big[ Q(\mu_0,\sqrt{\mu_0}f)+Q(\sqrt{\mu_0}f,\mu_0)\big]=\nu(v)f-Kf \\
    & =\nu(v)f-\int_{\mathbb{R}^3}k(v,v_*)f(v_*)dv_*,
\end{split}
\end{equation}
with the collision frequency $\nu(v)\equiv \iint_{\mathbb{R}^3\times \mathbb{S}^2}B(v-v_*,w)\mu_0(v_*)dwdv_*\sim \{1+|v|\}^{\mathcal{K}}$ for $-3< \mathcal{K}\leq 1$. Finally we define
\begin{equation}\label{eqn: diffuse projection}
  P_\gamma f(x,v):=c_\mu\sqrt{\mu_0(v)}\int_{n(x)\cdot u>0}f(x,u)\sqrt{\mu_0(u)}(n(x)\cdot u)du,
\end{equation}
where $c_\mu$ is the normalization constant.

\begin{corollary}\label{Thm: steady solution}
For given $T_0>0$, there exists $\delta_0>0$ such that if
\begin{equation}\label{eqn: small pert condition}
\sup_{x\in \partial \Omega}|T_w(x)-T_0|<\delta_0,\quad \max\{|1-r_\perp|,|1-r_\parallel|\}<\delta_0,
\end{equation}
then there exists a non-negative solution $F_s=\mu_0+\sqrt{\mu_0}f_s\geq 0$ with $\iint_{\Omega\times \mathbb{R}^3}f_s\sqrt{\mu_0}dxdv=0$ to the steady problem~\eqref{eqn: Steady Boltzmann}. And for all $0\leq \zeta<\frac{1}{4+2\delta_0}$, $\beta>4$,
\[\Vert \langle v\rangle^{\beta}e^{\zeta |v|^2}f_s\Vert_\infty +|\langle v\rangle^\beta e^{\zeta|v|^2}f_s|_\infty \lesssim \delta_0\ll 1.\]
If $\mu_0+\sqrt{\mu_0}g_s$ with $\iint_{\Omega\times \mathbb{R}^3}g_s\sqrt{\mu_0}dxdv=0$ is another solution such that $\Vert \langle v\rangle^\beta g_s\Vert_\infty+|\langle v\rangle^{\beta}g_s|_\infty\ll 1$ for $\beta>4$, then $f_s\equiv g_s$.

\end{corollary}

\begin{corollary}\label{Thm: dynamical stability} For $0< \zeta < \frac{1}{4+2\delta_0},$ set $\beta=0$, and for $%
\zeta=0,$ set $\beta>4$ where $\delta_0>0$ is in Corollary \ref{Thm: steady solution}. There exists $%
\lambda >0$ and $\varepsilon _{0}>0$, depending on $\delta _{0}$, such that
if $\iint_{\Omega\times\mathbf{R}^3} f_0\sqrt{\mu_0} = \iint_{\Omega\times%
\mathbf{R}^3} f_s\sqrt{\mu_0}=0,$ and if
\begin{equation}
\Vert \langle v\rangle ^{\beta }e^{\zeta |v|^{2}}[f(0)-f_{s}]\Vert _{\infty
}+|\langle v\rangle ^{\beta }e^{\zeta |v|^{2}}[f(0)-f_{s}]|_{\infty }\leq
\varepsilon _{0},  \label{epsilon0}
\end{equation}%
then there exists a unique non-negative solution {$F(t)=\mu_0 +\sqrt{\mu_0 }%
f(t)\geq 0$} to the dynamical problem~\eqref{eqn: VPB equation} with boundary condition~\eqref{eqn:BC},~\eqref{eqn: Formula for R}. And we have
\begin{eqnarray*}
&&\ \ \Vert \langle v\rangle ^{\beta }e^{\zeta |v|^{2}}[f(t)-f_{s}]\Vert
_{\infty }+|\langle v\rangle ^{\beta }e^{\zeta |v|^{2}}[f(t)-f_{s}]|_{\infty
} \\
&&\ \lesssim \ e^{-\lambda t}\big\{\Vert \langle v\rangle ^{\beta }e^{\zeta
|v|^{2}}[f(0)-f_{s}]\Vert _{\infty }+|\langle v\rangle ^{\beta }e^{\zeta
|v|^{2}}[f(0)-f_{s}]|_{\infty }\big\}.
\end{eqnarray*}%

\end{corollary}

\begin{remark}
Different to the accommodation coefficient with almost no constraint in Theorem \ref{local_existence}, in Corollary \ref{Thm: steady solution}, Corollary \ref{Thm: dynamical stability} we need to restrict these two coefficients to be close to $1$ in~\eqref{eqn: small pert condition}. To be more specific, we require the C-L boundary to be close to the diffuse boundary condition.

In this paper we show the proof for the hard sphere case where $0\leq \mathcal{K} \leq 1$. We can establish the same result for the soft potential case( $-3< \mathcal{K}<0$ ) using the argument provided in~\cite{Duan}.

\end{remark}

\subsection{Difficulty and proof strategy}

For proving the local well-posedness we focus on establishing $L^\infty$ estimate. In particular, for the $L^\infty$ estimate we trace back along the characteristic until it hits the boundary or the initial datum. Thus we derive a new trajectory formula with C-L boundary condition in~\eqref{eqn:C-L boundary condition in pro measure}.
Before tracing back to $t=0$ there will be repeated interaction with the boundary, which creates a multiple integral due to the boundary condition~\eqref{eqn:BC}. We present the formula in Lemma~\ref{lemma: the tracjectory formula for f^(m+1)}.

To understand this multiple integral we define $v_k,v_{k-1},\cdots,v_1$ in Definition \ref{Def:Back time cycle}. The $v_i$ represents the integral variable at $i$-th interaction with the boundary. For the diffuse reflection~\eqref{eqn: diffuse} with constant wall temperature, the boundary condition for $f=F/\sqrt{\mu}$ is given by~\eqref{eqn: diffuse f}. Thus at the $i$-th interaction the boundary condition is given by
\[f(v_{i-1})=c_\mu\sqrt{\mu(v_{i-1})}\int_{n\cdot v_{i}>0}f(v_{i})\sqrt{\mu(v_{i})}|n\cdot v_{i}|dv_i.\]
If we further trace back $f(v_i)$ in the integrand along the trajectory until the next interaction we have
\[f(v_i)=c_\mu\sqrt{\mu(v_{i})}\int_{n\cdot v_{i+1}>0}f(v_{i+1})\sqrt{\mu(v_{i+1})}|n\cdot v_{i+1}|dv_{i+1}.\]
Thus the integral over $v_i$ becomes
\[\int_{n\cdot v_i>0}c_\mu \mu(v_i)|n\cdot v_i|dv_i.\]
The integrand for $v_i$ is symmetric for all $1\leq i<k$ and not affected by the other variables. Moreover, $c_\mu \mu(v_i)|n\cdot v_i|dv_i$ is probability measure. Thus we can apply Fubini's theorem to compute this multiple integral. But for the C-L boundary condition~\eqref{eqn:BC}~\eqref{eqn: Formula for R}, the integrand is a function of both $v$ and $u$, as a result the probability measure is not symmetric for $v_i$. We are not free to apply the Fubini's theorem, which brings difficulty in bounding the trajectory formula. To be more specific, we need to compute the integral with the fixed order $v_k,v_{k-1},\cdots v_1$. We start from the integral of $v_k$. By~\eqref{eqn:C-L boundary condition in pro measure}, the integral of $v_k$ is
\begin{equation}\label{eqn: proof strat}
 \int_{n(x)\cdot v_k>0}   e^{-[\frac{1}{4T_M}-\frac{1}{2T_w(x)}]|v_k|^2}  d\sigma(v_k,v_{k-1}).
\end{equation}
When $r_\perp,r_\parallel\neq 0$, unlike the diffuse case, we can not decompose $d\sigma(v_k,v_{k-1})$ in~\eqref{eqn:probability measure}~\eqref{eqn: Formula for R} into a product of a function of $v_k$ and a function of $v_{k-1}$. Thus the integral ends up with a function of $v_{k-1}$, which will be included as a part of the integral over $v_{k-1}$. This justifies that the order of the integral can not be changed. Also the integral of $v_i$ is affected by the variables $v_{i+1},v_{i+2},\cdots v_k$. Thus we have to compute the multiple integral with fixed order from $v_k$ to $v_1$.

In fact,~\eqref{eqn: proof strat} can be computed explicitly as $e^{c|v_{k-1}|^2}$( Lemma \ref{Lemma: abc},Lemma \ref{Lemma: perp abc} ) and thus the integral for the variable $v_{k-1}$ has exactly the same form as~\eqref{eqn: proof strat}. This allows us to inductively derive an upper bound for this multiple integral. We present the induction result in Lemma~\ref{lemma: boundedness}.

With an upper bound for the trajectory formula another difficulty in the $L^\infty$ estimate is the measure $\mathbf{1}_{\{t_k>0\}}$. We need to show that this measure is small when $k$ is large so that the $L^\infty$ estimate follows by bounding a finite fold integral.

For this purpose~\cite{G,CKL} decompose $\gamma_+$ into the subspace
\[\gamma_+^{\delta}=\{u\in \gamma_+:|n\cdot u|>\delta,|u|\leq \delta^{-1}\}.\]
For diffuse case~\eqref{eqn: diffuse} the boundary condition for $f$ is given by~\eqref{eqn: diffuse f}. We can derive that there can be only finite number of $v_j$ belong to $\gamma_+^\delta$ under the constraint that $t<\infty$. Meanwhile, by~\eqref{eqn: diffuse f} the integral over $\gamma_+\backslash \gamma_+^\delta$ is a small magnitude number $O(\delta)$. When $k$( times of interaction with boundary ) is large enough one can obtain a large power of $O(\delta)$. The smallness of the measure $\mathbf{1}_{\{t_k>0\}}$ follows by this large power.

However, for our C-L boundary condition, the integrand is given by ~\eqref{eqn:C-L boundary condition in pro measure}~\eqref{eqn: Formula for R}, which contains the term $e^{-|v_\parallel-(1-r_{\parallel})u_\parallel|^2}$ in~\eqref{eqn:probability measure}. If we apply the standard decomposition the integral over $\gamma_+\backslash \gamma_+^\delta$ is no longer a small number $O(\delta)$. This is because even $|v_\parallel|\gg 1$, $|v_\parallel-(1-r_{\parallel})u_\parallel|$ still depends on $u_\parallel$.

A key observation is that when $|v_\parallel|$ is large enough, if $|v_\parallel-(1-r_{\parallel})u_\parallel|<\delta^{-1}$, we can obtain $|u_\parallel|\geq |v_\parallel|+\delta$ using $1-r_\parallel<1$. We take $1-r_\parallel=1/2$ as example. If $|v_\parallel-\frac{1}{2}u_\parallel|<\delta^{-1}$, we take $|v_\parallel|\geq 3\delta^{-1}$. Then we have
\[\frac{1}{2}|u_\parallel|>|v_\parallel|-\delta^{-1}>\frac{1}{2}|v_\parallel|+\frac{1}{2}\delta^{-1},\quad |u_\parallel|>|v_\parallel|+\delta^{-1}.\]
For $1-r_\parallel\neq 1/2$, we can choose a different number that depends on $1-r_\parallel$ to keep this property.

Now we suppose the "bad" case $|v_\parallel-(1-r_\parallel)u_\parallel|<\delta^{-1}$ happens for a large amount of times. By the discussion above, for the multiple integral with order $v_{k},\cdots,v_1$ we get an extremely huge velocity $|v_i|$ with some $i<k$. When we compute the integral with $d\sigma(v_i,v_{i-1})$, once $|v_{i-1}|$ is small the result is extremely small. This will provide the key factor to cancel all the other growth terms and prove the smallness of the measure $\mathbf{1}_{\{t_k>0\}}$. The other one is the "good" case $|v_\parallel-(1-r_{\parallel})u_\parallel|>\delta^{-1}$. From~\eqref{eqn: Formula for R} we can conclude the integral under this condition is a small magnitude number $O(\delta)$. Thus we can obtain some small factors to prove the smallness in both cases. Since the integrand in $d\sigma(u,v)$ in~\eqref{eqn:probability measure}~\eqref{eqn: Formula for R} still contains the variable $u_\perp,v_\perp$, we also need to apply the decomposition for these variables. The decomposition is similar and we skip the discussion here. But we point out that since the integrand for $u_\perp$ involves the first type Bessel function $I_0$, we need some basic estimate to verify that the integral for $u_\perp$ has the same property as $v_\parallel,u_\parallel$. We put these estimates in the appendix.

Thus our new ingredient here is that we decompose the boundary term $\gamma_+$ into the subspace
\[\gamma_+^{\eta}=\{u\in \gamma_+: |n\cdot u|>\eta\delta,|u|\leq (\eta\delta)^{-1}\}.\]
Here $\eta$ is small number depends on the coefficient $r_\parallel$ to ensure $|u_\parallel|\geq |v_\parallel|+\delta^{-1}$ when $|v_\parallel-(1-r_\parallel)u_\parallel|<\delta^{-1}$. During computing the trajectory formula the integral involves the variable $T_w(x)$( the wall temperature on $x\in\partial \Omega$ in~\eqref{eqn: Formula for R} ). It affects the real value of the coefficient for $u_\parallel$( different to $1-r_\parallel$ ). This is the reason that we need to impose some constraint on the wall temperature, which is the condition~\eqref{eqn: Constrain on T} in Theorem \ref{local_existence}. We present the decomposition and detail in Lemma~\ref{lemma: t^k} and its proof.

The way to construct the stationary solution and the dynamical stability( Corollary \ref{Thm: steady solution} and Corollary \ref{Thm: dynamical stability} ) comes from the ideas in~\cite{EGKM,EGKM2}. They consider the diffuse boundary condition with a small fluctuation on the wall temperature. Thus it can be regarded as a perturbation around the diffuse boundary condition with constant temperature. For our C-L boundary condition, when $r_\perp$ and $r_\parallel$ are close to $1$, it can be also regarded as a perturbation. Thus we need to restrict the accommodation coefficient to have a small fluctuation around $1$. Then we need to verify the boundary condition satisfies the property as stated in Proposition 4.1 in~\cite{EGKM}( the condition~\eqref{eqn: g,r condition} in this paper ). Then we can follow the standard procedure provided in~\cite{EGKM} to prove Corollary \ref{Thm: steady solution} and Corollary \ref{Thm: dynamical stability}.

\subsection{Outline}
In section 2 we conclude Theorem \ref{local_existence} by proving the $L^\infty$ bound for the sequence $f^m$ as well as the existence and $L^\infty$ stability. In section 3, we conclude Corollary \ref{Thm: steady solution} and Corollary \ref{Thm: dynamical stability} by using the key propositions provided in~\cite{EGKM}. In the appendix we prove some necessary estimates.

\section{Local well-posedness}
We start with the construction of the following iteration equation, which is positive preserving as in~\cite{G,K}. Then equation is given by
\begin{equation}\label{eqn: Fm+1}
    \partial_t F^{m+1}+v\cdot \nabla_x F^{m+1}=Q_{\text{gain}}(F^m,F^m)-\nu(F^m)F^{m+1},\quad F^{m+1}|_{t=0}=F_0,
\end{equation}
with boundary condition
\[F^{m+1}(t,x,v)|n(x)\cdot v|=\int_{n(x)\cdot u>0}R(u\to v;x,t)F^{m}(t,x,u)\{n(x)\cdot u\}du.\]
For $m\leq 0$ we set
\[F^m(t,x,v)=F_0(x,v).\]

We pose $F^{m+1}=\sqrt{\mu}f^{m+1}$ and
\begin{equation}\label{eqn: hm+1}
h^{m+1}(t,x,v)=e^{(\theta-t)|v|^2} f^{m+1}(t,x,v).
\end{equation}
The equation for $h^{m+1}$ reads
\begin{equation}\label{eqn:formula of f^(m+1)}
\partial_t h^{m+1}+v\cdot \nabla_x h^{m+1}+\nu^mh^{m+1} =e^{(\theta-t)|v|^2} \Gamma_{\text{gain}}\left(\frac{h^m}{e^{(\theta-t)|v|^2}},\frac{h^m}{e^{(\theta-t)|v|^2}}\right),
\end{equation}
with boundary condition
\begin{equation}\label{eqn: BC for fm+1}
     h^{m+1}(t,x,v)=e^{(\theta-t)|v|^2} e^{[\frac{1}{4T_M}-\frac{1}{2T_w(x)}]|v|^2}\int_{n(x)\cdot u>0} h^m(t,x,u)e^{-[\frac{1}{4T_M}-\frac{1}{2T_w(x)}]|u|^2} e^{-(\theta-t)|u|^2}  d\sigma(u,v).
\end{equation}
Here
\begin{equation}\label{eqn: num}
  \nu^m=|v|^2+\nu(F^m)\geq |v|^2.
\end{equation}

We use this section to establish the $L^\infty$ estimate of the sequence $h^{m+1}$ and derive the existence and uniqueness of the equation~\eqref{eqn: VPB equation}. The $L^\infty$ estimate is given by the following proposition.

\begin{proposition}\label{proposition: boundedness}
Assume $h^{m+1}$ satisfies~\eqref{eqn: hm+1} with Cercignani-Lampis boundary condition. Also assume $\theta<\frac{1}{4T_M}$, $\frac{\min(T_w(x))}{T_M}>\max\Big(\frac{1-r_\parallel}{2-r_\parallel},\frac{\sqrt{1-r_\perp}-(1-r_\perp)}{r_\perp} \Big)$ and
\begin{equation}\label{}
  \Vert h_0(x,v)\Vert_{L^\infty}<\infty,
\end{equation}
If
\begin{equation}\label{eqn: fm is bounded}
  \sup_{i\leq m}\Vert h^i(t,x,v)\Vert_{L^\infty}\leq C_\infty \Vert h_0(x,v)\Vert_{L^\infty}, \quad t\leq t_\infty,
\end{equation}
then we have
\begin{equation}\label{eqn: L_infty bound for f^m+1}
\sup_{0\leq t\leq t_\infty}\Vert h^{m+1}(t,x,v)\Vert_{L^\infty} \leq C_\infty\Vert h_0(x,v)\Vert_{L^\infty}.
\end{equation}
Here $C_\infty$ is a constant defined in~\eqref{eqn: Cinfty} and
\begin{equation}\label{eqn: t_1}
t\leq t_{\infty}=t_{\infty}( \Vert h_0(x,v)\Vert_{L^\infty},T_M,\min\{T_w(x)\},\theta,r_\perp,r_\parallel,\Omega)\ll 1.
\end{equation}

\end{proposition}

\begin{remark}
The condition~\eqref{eqn: t_1} is important. The smallness of the time will be used in the proof many times. And the parameters in~\eqref{eqn: t_1} guarantee that the time only depends on the temperature, accommodation and the initial condition.

The Proposition \ref{proposition: boundedness} implies the uniform-in-$m$ $L^\infty$ estimate for $h^{m}(t,x,v)$,
\begin{equation}\label{eqn: theta'}
\sup_m\Vert h^m\Vert_\infty<\infty
\end{equation}

The strategy to prove Proposition \ref{proposition: boundedness} is to express $h^{m+1}$ along the characteristic using the C-L boundary condition. We present the formula in Lemma~\ref{lemma: the tracjectory formula for f^(m+1)}. We will use Lemma~\ref{lemma: boundedness} and Lemma~\ref{lemma: t^k} to bound the formula.
\end{remark}

We represent $h^{m+1}$ with the stochastic cycles defined as follows.
\begin{definition}\label{Def:Back time cycle}
Let $\left(X^1(s;t,x,v),v\right)$ be the location and velocity along the trajectory before hitting the boundary for the first time,
\begin{equation}\label{eqn: trajectory for Xm}
  \frac{d}{ds}\left(
                \begin{array}{c}
                  X^1(s;t,x,v) \\
                  v \\
                \end{array}
              \right)=\left(
                        \begin{array}{c}
                          v \\
                          0\\
                        \end{array}
                      \right).
\end{equation}
Therefore, from~\eqref{eqn: trajectory for Xm}, we have
\[X^1(s;t,x,v)=x-v(t-s).\]
Define the back-time cycle as
\[t_{1}(t,x,v)=\sup\{s<t:X^1(s;t,x,v)\in \partial \Omega\},\]
\[x_{1}(t,x,v)=X^1\left(t_1(t,x,v);t,x,v\right),\]
%\[v_1(t,x,v)=V^1(t_1(t,x,v);t,x,v)\]
\[v_1\in \{v_1\in \mathbb{R}^3:n(x_1)\cdot v_1>0\}.\]
Also define
\[\mathcal{V}_1=\{v_1:n(x_1)\cdot v_1>0\},\quad x_1\in \partial \Omega.\]

Inductively, before hitting the boundary for the $k$-th time, define
\[t_k(t,x,v,v_1,\cdots,v_{k-1})=\sup\{s<t_{k-1}:X^k(s;t_{k-1},x_{k-1},v_{k-1})\in \partial \Omega\},\]
\[x_k(t,x,v,v_1,\cdots,v_{k-1})=X^k\left(t_k(t,x,v,v_{k-1});t_{k-1}(t,x,v),x_{k-1}(t,x,v),v_{k-1}\right),\]
%\[v_k(t,x,v,v_1,\cdot,v_{k-1})=V^k(t_k(t,x,v,v_{k-1});t_{k-1}(t,x,v),x_{k-1}(t,x,v),v_{k-1})\]
\[v_k\in \{v_k\in \mathbb{R}^3:n(x_k)\cdot v_k>0\},\]
\[\mathcal{V}_k=\{v_k:n(x_k)\cdot v_k>0\},\]
\[X^k(s;t_{k-1},x_{k-1},v_{k-1})=x_{k-1}-(t_{k-1}-s)v_{k-1}.\]
Here we set
\[(t_0,x_0,v_0)=(t,x,v).\]
For simplicity, we denote
\[X^k(s):=X^k(s;t_{k-1},x_{k-1},v_{k-1}).\]
in the following lemmas and propositions.

\end{definition}

\begin{lemma}\label{lemma: the tracjectory formula for f^(m+1)}
Assume $h^{m+1}$ satisfy~\eqref{eqn:formula of f^(m+1)} with the Cercignani-Lampis boundary condition~\eqref{eqn: BC for fm+1}, if $t_1\leq 0$, then
\begin{equation}\label{eqn: Duhamal principle for case1}
|h^{m+1}(t,x,v)|\leq |h_0(X^1(0),v)| +\int_0^t e^{-|v|^2(t-s)} e^{|v|^2(\theta-s)}      \Gamma_{\text{gain}}^m(s)ds.
\end{equation}
If $t_1>0$, for $k\geq 2$, then
\begin{equation}\label{eqn: Duhamel principle for case 2}
|h^{m+1}(t,x,v)|\leq \int_{t_1}^t e^{-|v|^2(t-s)} e^{|v|^2(\theta-s)}   \Gamma_{\text{gain}}^m(s)ds+e^{|v|^2(\theta-t_1)}e^{[\frac{1}{4T_M}-\frac{1}{2T_w(x_1)}]|v|^2}\int_{\prod_{j=1}^{k-1}\mathcal{V}_j}H,
\end{equation}
where $H$ is bounded by
\begin{equation}\label{eqn: formula for H}
\begin{split}
   &  \sum_{l=1}^{k-1}\mathbf{1}_{\{t_l>0,t_{l+1}\leq 0\}}|h_0\left(X^{l+1}(0),v_l\right)|d\Sigma_{l,m}^k(0)\\
    & +\sum_{l=1}^{k-1}\int_{\max\{0,t_{l+1}\}}^{t_l} e^{|v_l|^2(\theta-s)}|\Gamma_{\text{gain}}^{m-l}(s)|d\Sigma_{l,m}^k(s)ds\\
    &+\mathbf{1}_{\{t_k>0\}}|h^{m-k+2}\left(t_k,x_k,v_{k-1}\right)|d\Sigma_{k-1,m}^k(t_k),
\end{split}
\end{equation}
where
\begin{equation}\label{eqn:trajectory measure}
\begin{split}
 d\Sigma_{l,m}^k(s)=&    \Big\{\prod_{j=l+1}^{k-1}d\sigma\left(v_j,v_{j-1}\right)\Big\}\\
 & \times \Big\{e^{-|v_l|^2(t_l-s)} e^{-|v_l|^2(\theta-t_l)}  e^{-[\frac{1}{4T_M}-\frac{1}{2T_w(x_l)}]|v_l|^2}
d\sigma(v_l,v_{l-1})\Big\}\\
    & \times\Big\{\prod_{j=1}^{l-1} e^{[\frac{1}{2T_w(x_j)}-\frac{1}{2T_w(x_{j+1})}]|v_j|^2}  d\sigma\left(v_j,v_{j-1}\right)\Big\}.
\end{split}
\end{equation}
Here we use a notation
\begin{equation}\label{eqn: gamma^m}
\Gamma_{\text{gain}}^{m-l}(s):=\Gamma_{\text{gain}}\left(\frac{h^{m-l}(s,X^{l+1}(s),v_{l})}{e^{|v_l|^2(\theta-s)}},\frac{h^m(s,X^{l+1}(s),v_l)}{e^{|v_l|^2(\theta-s)}}\right) \text{ for $0\leq l\leq m$ }.
\end{equation}
\end{lemma}

\begin{proof}
For simplicity, we denote
\begin{equation}\label{eqn: mu tilde}
\tilde{\mu}(t,x,v):=e^{-|v|^2(\theta-t)}e^{-[\frac{1}{4T_M}-\frac{1}{2T_w(x)}]|v|^2}.
\end{equation}
 From~\eqref{eqn:formula of f^(m+1)}, for $0\leq s\leq t$,
we apply the fundamental theorem of calculus to get
\[\frac{d}{ds}\int_s^t -\nu^md\tau=\frac{d}{ds}\int_t^s \nu^md\tau=\nu^m.\]
Thus based on~\eqref{eqn:formula of f^(m+1)},
\begin{equation}\label{eqn: inte factor}
\frac{d}{ds}\left[e^{-\int_s^t \nu^m d\tau} h^{m+1}(s,X^1(s),v)\right]=e^{-\int_s^t \nu^m d\tau}e^{|v|^2(\theta-s)}\Gamma_{\text{gain}}^m(s).
\end{equation}
By~\eqref{eqn: num},
\begin{equation}\label{eqn: v^2}
e^{-\int_s^t \nu^m d\tau} \leq e^{-|v|^2(t-s)}\leq 0.
\end{equation}
Combining~\eqref{eqn: inte factor} and~\eqref{eqn: v^2}, we derive that if $t_1\leq 0$, then we have~\eqref{eqn: Duhamal principle for case1}.

If $t_1(t,x,v)>0$, then
\begin{equation}\label{eqn: proof for the intial step in the Duhamul principle}
\begin{split}
   & |h^{m+1}(t,x,v)\textbf{1}_{\{t_1>0\}}|\leq  |h^{m+1}\left(t_1,x_1,v\right)|e^{-|v|^2(t-t_1)}\\
    & +\int_{t_1}^t e^{-|v|^2(t-s)} e^{|v|^2(\theta-s)}|\Gamma_{\text{gain}}^m(s)|ds.
\end{split}
\end{equation}
We use an induction of $k$ to prove~\eqref{eqn: Duhamel principle for case 2}. The first term of the RHS of~\eqref{eqn: proof for the intial step in the Duhamul principle} can be expressed by the boundary condition. For $1\leq k\leq m$, we rewrite the boundary condition~\eqref{eqn: BC for fm+1} using~\eqref{eqn: mu tilde} as
\begin{equation}\label{eqn: bc for hm+1}
  h^{m-k+2}(t_k,x_k,v_{k-1})=\frac{1}{\tilde{\mu}\left(t_k,x_k,v_{k-1}\right)}\int_{\mathcal{V}_k}  h^{m-k+1}(t_k,x_k,v_k)\tilde{\mu}(t_k,x_k,v_k)  d\sigma\left(v_k,v_{k-1}\right).
\end{equation}
Directly applying~\eqref{eqn: bc for hm+1} with $k=1$ the first term of the RHS of~\eqref{eqn: proof for the intial step in the Duhamul principle} is bounded by
\begin{equation}\label{eqn: above term}
\frac{1}{\tilde{\mu}\left(t_1,x_1,v\right)}\int_{\mathcal{V}_1}h^{m}(t_1,x_1,v_1)\tilde{\mu}(t_1,x_1,v_1)d\sigma(v_1,v).
\end{equation}
Then we apply~\eqref{eqn: Duhamal principle for case1} and~\eqref{eqn: proof for the intial step in the Duhamul principle} to derive
\[~\eqref{eqn: above term}\leq \frac{1}{\tilde{\mu}(t_1,x_1,v)}\Big[\int_{\mathcal{V}_1}\mathbf{1}_{\{t_2\leq 0<t_1\}} e^{-|v_1|^2t_1} h^{m}(0,X^{2}(0),v_1)\tilde{\mu}(t_1,x_1,v_1)d\sigma(v_1,v)\]

\[+\int_{\mathcal{V}_1}\int_0^{t_1}\mathbf{1}_{\{t_2\leq 0<t_1\}}e^{-|v_1|^2(t_1-s)}  e^{|v_1|^2(\theta-s)}  |\Gamma_{\text{gain}}^{m-1}(s)|\tilde{\mu}(t_1,x_1,v_1)d\sigma(v_1,v) ds\]

\[+\int_{\mathcal{V}_1}\mathbf{1}_{\{t_2>0\}}e^{-|v_1|^2(t_1-t_2)} |h^{m}(t_2,x_2,v_1)\tilde{\mu}(t_1,x_1,v_1)d\sigma(v_1,v)\]

\[+\int_{\mathcal{V}_1}\int_{t_2}^{t_1}\mathbf{1}_{\{t_2> 0\}}e^{-|v_1|^2(t_1-s)}  e^{|v_1|^2(\theta-s)}  |\Gamma_{\text{gain}}^{m-1}(s)|\tilde{\mu}(t_1,x_1,v_1)d\sigma\left(v_1,v\right) ds \Big].\]
Therefore, the formula~\eqref{eqn: Duhamel principle for case 2} is valid for $k=2$.

Assume~\eqref{eqn: Duhamel principle for case 2} is valid for $k\geq 2$ (induction hypothesis). Now we prove that~\eqref{eqn: Duhamel principle for case 2} holds for $k+1$. We express the last term in~\eqref{eqn: formula for H} using the boundary condition. In~\eqref{eqn: bc for hm+1}, since $\frac{1}{\tilde{\mu}\left(t_k,x_k,v_{k-1}\right)}$ depends on $v_{k-1}$, we move this term to the integration over $\mathcal{V}_{k-1}$ in~\eqref{eqn: Duhamel principle for case 2}. Using the second line of~\eqref{eqn:trajectory measure}, the integration over $\mathcal{V}_{k-1}$ is
\begin{equation}\label{eqn: wl}
\int_{\mathcal{V}_{k-1}}   e^{-|v_{k-1}|^2(t_{k-1}-t_k)}\tilde{\mu}(t_{k-1},x_{k-1},v_{k-1}) /\tilde{\mu}\left(t_k,x_k,v_{k-1}\right)   d\sigma(v_{k-1},v_{k-2}).
\end{equation}
We have
\[e^{-|v_{k-1}|^2(t_{k-1}-t_k)}\tilde{\mu}(t_{k-1},x_{k-1},v_{k-1}) /\tilde{\mu}\left(t_k,x_k,v_{k-1}\right)\]
\[=e^{-|v_{k-1}|^2(t_{k-1}-t_k)}e^{|v_{k-1}|^2(t_{k-1}-t_k)}e^{[\frac{1}{2T_w(x_{k-1})}-\frac{1}{2T_w(x_k)}]|v_{k-1}|^2}=e^{[\frac{1}{2T_w(x_{k-1})}-\frac{1}{2T_w(x_k)}]|v_{k-1}|^2}.\]
Therefore, by~\eqref{eqn: wl} the integration over $\mathcal{V}_{k-1}$ reads
\begin{equation}\label{eqn: 6}
\int_{\mathcal{V}_{k-1}}   e^{[\frac{1}{2T_w(x_{k-1})}-\frac{1}{2T_w(x_k)}]|v_{k-1}|^2}   d\sigma(v_{k-1},v_{k-2}),
\end{equation}
 which is consistent with third line in~\eqref{eqn:trajectory measure} with $l=k-1$.

For the remaining integration in~\eqref{eqn: bc for hm+1}, we split the integration over $\mathcal{V}_k$ into two terms as
\begin{equation}\label{eqn: sh}
\int_{\mathcal{V}_k}h^{m-k+1}(t_k,x_k,v_k) \tilde{\mu}(t_k,x_k,v_k) d\sigma(v_k,v_{k-1})=\underbrace{\int_{\mathcal{V}_k}\mathbf{1}_{\{t_{k+1}\leq 0<t_k\}}}_{\eqref{eqn: sh}_1}+\underbrace{\int_{\mathcal{V}_k}\mathbf{1}_{\{t_{k+1}>0\}}}_{\eqref{eqn: sh}_2}.
\end{equation}
For the first term of the RHS of~\eqref{eqn: sh}, we use the similar bound of~\eqref{eqn: Duhamal principle for case1} and derive that
\begin{equation}\label{eqn: wy}
\begin{split}
   &\eqref{eqn: sh}_1\leq  \int_{\mathcal{V}_k} \mathbf{1}_{\{t_{k+1}\leq 0<t_k\}} e^{- |v_k|^2t_k}h^{m-k+1}(0,X^{k+1}(0),v_k)\tilde{\mu}(t_k,x_k,v_k)d\sigma(v_k,v_{k-1}) \\
    & +\int_{\mathcal{V}_k}\int_0^{t_k}\mathbf{1}_{\{t_{k+1}\leq 0<t_k\}}e^{-|v_k|^2(t_k-s)}e^{ |v_k|^2(\theta-s)}\Gamma_{\text{gain}}^{m-k}(s)\tilde{\mu}(t_k,x_k,v_k) d\sigma(v_k,v_{k-1}) ds.
\end{split}
\end{equation}
In the first line of~\eqref{eqn: wy},
\[e^{-|v_k|^2t_k}\tilde{\mu}(t_k,x_k,v_k)d\sigma(v_k,v_{k-1}),\]
is consistent with the second line of~\eqref{eqn:trajectory measure} with $l=k$, $s=t_k$. In the second line of~\eqref{eqn: wy}
\[e^{-|v_k|^2(t_k-s)}\tilde{\mu}(t_k,x_k,v_k)d\sigma(v_k,v_{k-1}),\]
is consistent with the second line of~\eqref{eqn:trajectory measure} with $l=k$.

From the induction hypothesis(~\eqref{eqn: Duhamel principle for case 2} is valid for $k$) and~\eqref{eqn: 6}, we derive the integration over $\mathcal{V}_j$ for $j\leq k-1$ is consistent with the third line of~\eqref{eqn:trajectory measure}.
After taking integration $\int_{\prod_{j=1}^{k-1} \mathcal{V}_j}$ we change $d\Sigma_{k-1,m}^k$ in~\eqref{eqn:trajectory measure} to $d\Sigma_{k,m}^{k+1}$. Thus the contribution of~\eqref{eqn: wy} is
\begin{equation}\label{eqn: qs}
  \begin{split}
     &\int_{\prod_{j=1}^{k} \mathcal{V}_j} \mathbf{1}_{\{t_{k+1}\leq 0<t_k\}}|h_0\left(X^{k+1}(0),v_k\right)|d\Sigma_{k,m}^{k+1}(0) \\
      & +\int_{\prod_{j=1}^{k} \mathcal{V}_j}\int_{0}^{t_k} e^{ |v_k|^2(\theta-s)}\Gamma_{\text{gain}}^{m-k}(s)d\Sigma_{k,m}^{k+1}(s)ds.
  \end{split}
\end{equation}

For the second term of the RHS of~\eqref{eqn: sh}, we use the same estimate as~\eqref{eqn: Duhamal principle for case1} and we derive
\begin{equation}\label{eqn: wyl}
  \begin{split}
     &  \eqref{eqn: sh}_2\leq \int_{\mathcal{V}_k}\mathbf{1}_{\{t_{k+1}>0\}}e^{-|v_k|^2(t_{k}-t_{k+1})}h^{m-k+1}\left(t_{k+1},x_{k+1},v_{k}\right)\tilde{\mu}(t_k,x_k,v_k)d\sigma\left(v_k,v_{k-1}\right)+\\
      & \int_{\mathcal{V}_k}\int_{t_{k+1}}^{t_k}\mathbf{1}_{\{t_{k+1}> 0\}}e^{-|v_k|^2(t_k-s)}e^{ |v_k|^2(\theta-s)}\Gamma_{\text{gain}}^{m-k}(s)\tilde{\mu}(t_k,x_k,v_k) d\sigma(v_k,v_{k-1}) ds.
  \end{split}
\end{equation}
Similar to~\eqref{eqn: qs}, after taking integration over $\int_{\prod_{j=1}^{k-1}\mathcal{V}_j}$ the contribution of~\eqref{eqn: wyl} is
\begin{equation}\label{eqn: nxh}
\begin{split}
   & \int_{\prod_{j=1}^{k} \mathcal{V}_j} \mathbf{1}_{\{t_{k+1}>0\}}|h^{m-k+1}\left(t_{k+1},x_{k+1},v_k\right)|d\Sigma_{k,m}^{k+1}(t_{k+1})\\
    & +\int_{\prod_{j=1}^{k} \mathcal{V}_j}\int_{t_{k+1}}^{t_k}e^{ |v_k|^2(\theta-s) }\Gamma_{\text{gain}}^{m-k}(s)d\Sigma_{k,m}^{k+1}(s) ds.
\end{split}
\end{equation}

From~\eqref{eqn: nxh}~\eqref{eqn: qs}, the summation in the first and second lines of~\eqref{eqn: formula for H} extends to $k$. And the index of the third line of~\eqref{eqn: formula for H} changes from $k$ to $k+1$.
For the rest terms, the index $l\leq k-1$, we haven't done any change to them. Thus their integration are over $\prod_{l=1}^{k-1} \mathcal{V}_j$. We add $\int_{\mathcal{V}_k}d\sigma(v_k,v_{k-1})=1$ to all of them, so that all the integrations are over $\prod_{l=1}^k \mathcal{V}_j$ and we change $d\Sigma_{l,m}^{k-1}$ to $d\Sigma_{l,m}^k$ by
\[d\Sigma_{l,m}^{k}=d\sigma\left(v_k,v_{k-1}\right)d\Sigma_{l,m}^{k-1}.\]
Therefore, the formula~\eqref{eqn: formula for H} is valid for $k+1$ and we derive the lemma.
\end{proof}

The next lemma is the key to prove the $L^\infty$ bound for $h^{m+1}$. Below we define several notation: let
\begin{equation}\label{eqn: def of r}
r_{max}:=\max(r_\parallel(2-r_\parallel),r_\perp),~~~~~~~~~~~~ r_{min}:=\min(r_\parallel(2-r_\parallel),r_\perp).
\end{equation}
Then we have
\begin{equation}\label{eqn: r>0}
1\geq r_{max}\geq r_{min}>0.
\end{equation}
Define
\[\xi:=\frac{1}{4T_M\theta},\]
where $\theta<\frac{1}{4T_M}$ is given in~\eqref{eqn: hm+1}. Then we have
\begin{equation}\label{eqn: xi}
  \theta=\frac{1}{4T_M\xi},\quad \xi>1.
\end{equation}
We inductively define:
\begin{equation}\label{eqn: definition of T_p}
T_{l,l}=\frac{2\xi}{\xi+1}T_M,    \quad T_{l,l-1}=r_{min}T_M +(1-r_{min})T_{l,l}, \cdots,\quad  T_{l,1}= r_{min}T_M +(1-r_{min})T_{l,2}.
\end{equation}
By a direct computation, for $1\leq i\leq l$, we have
\begin{equation}\label{eqn: formula of Tp}
T_{l,i}=\frac{2\xi}{\xi+1}T_M+(T_M-\frac{2\xi}{\xi+1}T_M)[1-(1-r_{min})^{l-i}]
\end{equation}
Moreover, let
\begin{equation}\label{eqn: big phi}
\begin{split}
 d\Phi_{p,m}^{k,l}(s):=  & \{\prod_{j=l+1}^{k-1}d\sigma(v_j,v_{j-1})\}\\
 &\times \{e^{-|v_l|^2(t_l-s)} e^{-|v_l|^2(\theta-t_l)} e^{-[\frac{1}{4T_M}-\frac{1}{2T_w(x_l)}]|v_l|^2}
d\sigma(v_l,v_{l-1})\}  \\
    & \times \{\prod_{j=p}^{l-1}   e^{[\frac{1}{2T_w(x_j)}-\frac{1}{2T_w(x_{j+1})}]|v_j|^2}  d\sigma(v_j,v_{j-1})\}.
\end{split}
\end{equation}
Note that if $p=1$, $d\Phi_{1,m}^{k,l}(s)=d\Sigma_{l,m}^{k}(s)$ where $d\Sigma_{l,m}^{k}(s)$ is defined in~\eqref{eqn:trajectory measure}. And let
\begin{equation}\label{eqn: upsilon}
d\Upsilon_{p}^{p'}:=\{\prod_{j=p}^{p'}   e^{[\frac{1}{2T_w(x_j)}-\frac{1}{2T_w(x_{j+1})}]|v_j|^2}  d\sigma(v_j,v_{j-1})\}.
\end{equation}
Then by the definition of~\eqref{eqn: big phi} and~\eqref{eqn:trajectory measure}, we have
\begin{equation}\label{eqn: ppt for Phi}
d\Phi_{p,m}^{k,l}(s)=d\Phi_{p',m}^{k,l}(s)d\Upsilon_{p}^{p'-1},
\end{equation}
\begin{equation}\label{eqn: property for Gamma}
d\Sigma_{l,m}^k(s)=d\Phi_{p,m}^{k,l}(s)d\Upsilon_1^{p-1}.
\end{equation}
\begin{remark}
We aim to bound the multiple integral in the trajectory formula in Lemma \ref{lemma: the tracjectory formula for f^(m+1)}. Each integral in the formula involves the variable $T_w(x),T_M,r_\perp,r_\parallel$, thus we need to find the pattern of the upper bound for each fold integral. This is the reason we define these inductive notations.

\end{remark}

Now we state the lemma.

\begin{lemma}\label{lemma: boundedness}
Given the formula for $h^{m+1}$ in~\eqref{eqn: Duhamal principle for case1} and~\eqref{eqn: Duhamel principle for case 2} of lemma~\ref{lemma: the tracjectory formula for f^(m+1)}, there exists
\begin{equation}\label{eqn: t*}
t^*=t^*(T_M,\xi,\mathcal{C},k)
\end{equation}
such that when $t\leq t^*$, for any $0\leq s\leq t_l$ we have
\begin{equation}\label{eqn: boundedness for l-p+1 fold integration}
   \int_{\prod_{j=p}^{k-1}\mathcal{V}_j}     \mathbf{1}_{\{t_l>0\}}  d\Phi_{p,m}^{k,l}(s) \leq (C_{T_M,\xi})^{2(l-p+1)}\mathcal{A}_{l,p}.
\end{equation}
where we define:
\begin{equation}\label{eqn: Elp}
\mathcal{A}_{l,p}=\exp\left(\big[ \frac{[T_{l,p}-T_w(x_{p})][1-r_{min}]}{2T_w(x_{p})[T_{l,p}(1-r_{min})+r_{min} T_w(x_{p})]} + \mathcal{C}^{l-p+1}t\big]|v_{p-1}|^2\right).
\end{equation}
Here $C_{T_M,\xi}$ is a constant defined in~\eqref{eqn: 1 one} and $\mathcal{C}$ is constant defined in~\eqref{eqn: cal C}.

Moreover, for any $p<p'\leq l$, we have
\begin{equation}\label{eqn: structure}
\int_{\prod_{j=p}^{k-1}\mathcal{V}_j} \mathbf{1}_{\{t_l>0\}}  d\Phi_{p,m}^{k,l}(s)\leq (C_{T_M,\xi})^{2(l-p'+1)} \mathcal{A}_{l,p'}\int_{\prod_{j=p}^{p'-1} \mathcal{V}_j}  \mathbf{1}_{\{t_l>0\}} d\Upsilon_p^{p'-1}\leq (C_{T_M,\xi})^{2(l-p+1)}\mathcal{A}_{l,p}.
\end{equation}

\end{lemma}

\begin{proof}
From~\eqref{eqn: normalization} and~\eqref{eqn:probability measure}, for the first bracket of the first line in~\eqref{eqn:trajectory measure} with $l+1\leq j\leq k-1$, we have
\[\int_{\prod_{j=l+1}^{k-1} \mathcal{V}_j}    \prod_{j=l+1}^{k-1}d\sigma(v_j,v_{j-1})=1.\]
Without loss of generality we can assume $k=l+1$. Thus $d\Phi_{p,m}^{k,l}=d\Phi_{p,m}^{l+1,l}$. We use an induction of $p$ with $1\leq p\leq l$ to prove~\eqref{eqn: boundedness for l-p+1 fold integration}.

When $p=l$, by the second line of~\eqref{eqn: big phi}, the integration over $\mathcal{V}_l$ is written as
\begin{equation}\label{eqn: dsigmal}
\int_{\mathcal{V}_l} e^{-|v_l|^2(t_l-s)} e^{-|v_l|^2(\theta-t_l)} e^{-[\frac{1}{4T_M}-\frac{1}{2T_w(x_l)}]|v_l|^2}d\sigma(v_l,v_{l-1}).
\end{equation}
By $\theta=\frac{1}{4T_M\xi}$ in~\eqref{eqn: xi} and $s\leq t_l$, we bound~\eqref{eqn: dsigmal} by
\begin{equation}\label{eqn: dsigma11}
\int_{\mathcal{V}_l} e^{-[\frac{1}{2T_M}\frac{\xi+1}{2\xi}-\frac{1}{2T_w(x_l)}-t_l]|v_l|^2}d\sigma(v_l,v_{l-1}).
\end{equation}
Expanding $d\sigma(v_l,v_{l-1})$ with~\eqref{eqn: Formula for R} and~\eqref{eqn:probability measure} we rewrite~\eqref{eqn: dsigma11} as
\begin{equation}\label{eqn: int over V_l}
\begin{split}
   &\int_{\mathcal{V}_{l,\perp}} \frac{2}{r_\perp}\frac{|v_{l,\perp}|}{2T_w(x_l)}  e^{-[\frac{1}{2T_M}\frac{\xi+1}{2\xi}-\frac{1}{2T_w(x_l)}-t_l]|v_{l,\perp}|^2}I_0\left(\frac{(1-r_\perp)^{1/2}v_{l,\perp}v_{l-1,\perp}}{T_w(x_l)r_\perp}\right)e^{-\frac{|v_{l,\perp}|^2+(1-r_\perp)|v_{l-1,\perp}|^2}{2T_w(x_l)r_\perp}}  dv_{l,\perp} \\
    &\times \int_{\mathcal{V}_{l,\parallel}}\frac{1}{\pi r_\parallel(2-r_\parallel)(2T_w(x_l))}e^{-[\frac{1}{2T_M}\frac{\xi+1}{2\xi}-\frac{1}{2T_w(x_l)}-t_l]|v_{l,\parallel}|^2}e^{-\frac{1}{2T_w(x_l)}\frac{|v_{l,\parallel}-(1-r_\parallel)v_{l-1,\parallel}|^2}{r_\parallel(2-r_\parallel)}}dv_{l,\parallel},
\end{split}
\end{equation}
where $v_{l,\parallel}$, $v_{l,\perp}$, $\mathcal{V}_{l,\perp}$ and $\mathcal{V}_{l,\parallel}$ are defined as
\begin{equation}\label{eqn: Define space}
v_{l,\perp}=v_l\cdot n(x_l),~~ v_{l,\parallel}=v_l-v_{l,\perp}n(x_l),~~\mathcal{V}_{l,\perp}=\{v_{l,\perp}:v_l\in \mathcal{V}_l\},~~\mathcal{V}_{l,\parallel}=\{v_{l,\parallel}:v_l\in \mathcal{V}_l\}.
\end{equation}
$v_{l-1,\parallel}$ and $v_{l-1,\perp}$ are defined similarly.

First we compute the integration over $\mathcal{V}_{l,\parallel}$, the second line of~\eqref{eqn: int over V_l}. To apply~\eqref{eqn: coe abc} in Lemma \ref{Lemma: abc}, we set
\[\e=t_l,~w=(1-r_\parallel)v_{l-1,\parallel}~,v=v_{l,\parallel},\]
\begin{equation}\label{eqn: coefficient a and b}
a=-[\frac{1}{2T_M\frac{2\xi}{\xi+1}}-\frac{1}{2T_w(x_l)}],~ b=\frac{1}{2T_w(x_l)r_\parallel(2-r_\parallel)}.
\end{equation}
By $\xi>1$ in~\eqref{eqn: xi}, we take $t^*=t^*(\xi,T_M)\ll 1$ such that when $t_l<t\leq t^*$, we have
\begin{equation}\label{eqn: b-a-e}
b-a-\e=\frac{1}{2T_w(x_l)r_\parallel(2-r_\parallel)}-\frac{1}{2T_w(x_l)}+\frac{1}{2T_M\frac{2\xi}{\xi+1}}-t_l
\geq \frac{1}{2T_M\frac{2\xi}{\xi+1}}-t\geq \frac{1}{4T_M}.
\end{equation}
Also we take $t^*=t^*(\xi,T_M)$ to be small enough to obtain $1+4T_M t_l\leq 1+4T_M t\leq 2$ when $t\leq t^*$. Thus the $t^*$ we choose here is consistent with~\eqref{eqn: t*}. Hence
\[\frac{b}{b-a-\e}=\frac{b}{b-a}[1+\frac{\e}{b-a-\e}]\leq \frac{\frac{2\xi}{\xi+1}T_M}{\frac{2\xi}{\xi+1}T_M+[T_w(x_l)-\frac{2\xi}{\xi+1}T_M]r_\parallel(2-r_\parallel)}[1+4T_Mt_l]\] \begin{equation}\label{eqn: 1 one}
\leq \frac{\frac{4\xi}{\xi+1}T_M}{\frac{2\xi}{\xi+1}T_M+[\min\{T_w(x)\}-\frac{2\xi}{\xi+1}T_M]r_{max}}:=C_{T_M,\xi},
\end{equation}
where we use~\eqref{eqn: def of r}.

In regard to~\eqref{eqn: coe abc}, we have
\begin{equation}\label{eqn: com}
\frac{(a+\e)b}{b-a-\e}=\frac{ab}{b-a}[1+\frac{\e}{b-a-\e}]+\frac{b}{b-a-\e}\e.
\end{equation}
By~\eqref{eqn: 1 one} and $t_l<t$, we obtain
\[\frac{b}{b-a-\e}\e\leq \frac{\frac{4\xi}{\xi+1}T_M}{\frac{2\xi}{\xi+1}T_M+[\min\{T_w(x)\}-\frac{2\xi}{\xi+1}T_M]r_{\max}}t.\]
By~\eqref{eqn: coefficient a and b}, we have
\[\frac{ab}{b-a}=\frac{\frac{2\xi}{\xi+1}T_M-T_w(x_l)}{2T_w(x_l)[\frac{2\xi}{\xi+1}T_M+[T_w(x_l)-\frac{2\xi}{\xi+1}T_M]r_\parallel(2-r_\parallel)]}.\]
Therefore, by~\eqref{eqn: b-a-e} and~\eqref{eqn: com} we obtain
\begin{equation}\label{eqn: 2 two}
\begin{split}
  &  \frac{(a+\e)b}{b-a-\e}\leq \frac{\frac{2\xi}{\xi+1}T_M-T_w(x_l)}{2T_w(x_l)[\frac{2\xi}{\xi+1}T_M+[T_w(x_l)-\frac{2\xi}{\xi+1}T_M]r_\parallel(2-r_\parallel)]}+\mathcal{C} t.
\end{split}
\end{equation}
where we define
\begin{equation}\label{eqn: cal C}
\mathcal{C}:=\frac{4T_M\big(\frac{2\xi}{\xi+1}T_M-\min\{T_w(x)\}\big)}{2\min\{T_w(x)\}[\frac{2\xi}{\xi+1}T_M+[\min\{T_w(x)\}-\frac{2\xi}{\xi+1}T_M]r_{max}]}+\frac{\frac{4\xi}{\xi+1}T_M}{\frac{2\xi}{\xi+1}T_M+[\min\{T_w(x)\}-\frac{2\xi}{\xi+1}T_M]r_{max}}.
\end{equation}

By~\eqref{eqn: 1 one},~\eqref{eqn: 2 two} and Lemma~\ref{Lemma: abc}, using $w=(1-r_\parallel)v_{l-1,\parallel}$ we bound the second line of~\eqref{eqn: int over V_l} by
\begin{equation}\label{eqn: result for para}
C_{T_M,\xi}\exp\bigg(\Big[\frac{[\frac{2\xi}{\xi+1}T_M-T_w(x_l)]}{2T_w(x_l)[\frac{2\xi}{\xi+1}T_M(1-r_\parallel)^2+r_\parallel(2-r_\parallel) T_w(x_l)]}+\mathcal{C}t \Big]|(1-r_\parallel)v_{l-1,\parallel}|^2\bigg)
\end{equation}
\begin{equation}\label{eqn: result of dsigmal}
\leq C_{T_M,\xi}\exp\bigg(\Big[\frac{[\frac{2\xi}{\xi+1}T_M-T_w(x_l)][1-r_{min}]}{2T_w(x_l)\big[\frac{2\xi}{\xi+1}T_M(1-r_{min})+r_{min} T_w(x_l)\big]}+\mathcal{C}t\Big]|v_{l-1,\parallel}|^2\bigg).
\end{equation}
where we use~\eqref{eqn: def of r} and~\eqref{eqn: r>0}.

Next we compute first line of~\eqref{eqn: int over V_l}. To apply~\eqref{eqn: coe abc perp} in Lemma \ref{Lemma: perp abc}, we set
\[\e=t_l,~w=\sqrt{1-r_\parallel}v_{l-1,\perp}~,v=v_{l,\perp},\]
\[a=-[\frac{1}{2T_M\frac{2\xi}{\xi+1}}-\frac{1}{2T_w(x_l)}],~ b=\frac{1}{2T_w(x_l)r_\perp}.\]
Thus we can compute $\frac{b}{b-a-\e}$ and $\frac{(a+\e)b}{b-a-\e}$ using the exactly the way as~\eqref{eqn: 1 one} and~\eqref{eqn: 2 two} with replacing $r_\parallel(2-r_\parallel)$ by $r_\perp$. Hence replacing $r_\parallel(2-r_\parallel)$ by $r_\perp$ and replacing $v_{l-1,\parallel}$ by $v_{l-1,\perp}$ in~\eqref{eqn: result for para}, we bound the first line of~\eqref{eqn: int over V_l} by
\[C_{T_M,\xi}\exp\bigg(\Big[\frac{[\frac{2\xi}{\xi+1}T_M-T_w(x_l)]}{2T_w(x_l)[\frac{2\xi}{\xi+1}T_M(1-r_\perp)+r_\perp T_w(x_l)]}+\mathcal{C}t_l \Big]|\sqrt{1-r_\perp}v_{l-1,\perp}|^2\bigg)\]
\begin{equation}\label{eqn: result of dsigmal normal}
\leq C_{T_M,\xi}\exp\bigg(\Big[\frac{[\frac{2\xi}{\xi+1}T_M-T_w(x_l)][1-r_{min}]}{2T_w(x_l)\big[\frac{2\xi}{\xi+1}T_M(1-r_{min})+r_{min} T_w(x_l)\big]}+\mathcal{C}t\Big]|v_{l-1,\perp}|^2\bigg).
\end{equation}
where we use~\eqref{eqn: def of r} and~\eqref{eqn: r>0}.

Collecting~\eqref{eqn: result of dsigmal}~\eqref{eqn: result of dsigmal normal}, we derive
\[\eqref{eqn: int over V_l}\leq (C_{T_M,\xi})^2\exp\left(\left[\frac{[\frac{2\xi}{\xi+1}T_M-T_w(x_l)][1-r_{min}]}{2T_w(x_l)\big[\frac{2\xi}{\xi+1}T_M(1-r_{min})+r_{min} T_w(x_l)\big]}+\mathcal{C}t\right]|v_{l-1}|^2\right)=(C_{T_M,\xi})^2\mathcal{A}_{l,l},\]
where $\mathcal{A}_{l,l}$ is defined in~\eqref{eqn: Elp} and $T_{l,l}=\frac{2\xi}{\xi+1}T_M$.

Therefore,~\eqref{eqn: boundedness for l-p+1 fold integration} is valid for $p=l$.

Suppose~\eqref{eqn: boundedness for l-p+1 fold integration} is valid for the $p=q+1$(induction hypothesis) with $q+1\leq l$, then
\[\int_{\prod_{j=q+1}^{l}\mathcal{V}_j}        \mathbf{1}_{\{t_l>0\}} d\Phi_{q+1,m}^{l+1,l}(s)\leq (C_{T_M,\xi})^{2(l-q)}\mathcal{A}_{l,q+1}.\]
We want to show~\eqref{eqn: boundedness for l-p+1 fold integration} holds for $p=q$. By the hypothesis and the third line of~\eqref{eqn: big phi},
\begin{equation}\label{eqn: dsigmaq}
  \int_{\prod_{j=q}^{l}\mathcal{V}_j}        \mathbf{1}_{\{t_l>0\}} d\Phi_{q,m}^{l+1,l}(s) \leq  (C_{T_M,\xi})^{2(l-q)}\mathcal{A}_{l,q+1}\int_{\mathcal{V}_q}  e^{[\frac{1}{2T_w(x_{q})}-\frac{1}{2T_w(x_{q+1})}]|v_{q}|^2}  d\sigma(v_q,v_{q-1}).
\end{equation}
Using the definition of $\mathcal{A}_{l,q+1}$ in~\eqref{eqn: Elp}, we obtain
\begin{equation}\label{eqn: dsigmap}
\begin{split}
   &\eqref{eqn: dsigmaq} \leq   (C_{T_M,\xi})^{2(l-q)}  \int_{\mathcal{V}_{q}}    \exp\bigg(\frac{(T_{l,q+1}-T_w(x_{q+1}))(1-r_{min})}{2T_w(x_{q+1})[T_{l,q+1}(1-r_{min})+r_{min} T_w(x_{q+1})]}|v_q|^2+\mathcal{C}^{l-q}t|v_q|^2\bigg)                                     \\
    &   e^{[\frac{1}{2T_w(x_{q})}-\frac{1}{2T_w(x_{q+1})}]|v_{q}|^2}  d\sigma(v_q,v_{q-1}).
\end{split}
\end{equation}
We focus on the coefficient of $|v_q|^2$ in~\eqref{eqn: dsigmap}, we derive
\[\frac{(T_{l,q+1}-T_w(x_{q+1}))(1-r_{min})}{2T_w(x_{q+1})[T_{l,q+1}(1-r_{min})+r_{min} T_w(x_{q+1})]}|v_q|^2+[\frac{1}{2T_w(x_{q})}-\frac{1}{2T_w(x_{q+1})}]|v_q|^2\]
\[=       \frac{(T_{l,q+1}-T_w(x_{q+1}))(1-r_{min})-[T_{l,q+1}(1-r_{min})+r_{min} T_w(x_{q+1})]}{2T_w(x_{q+1})[T_{l,q+1}(1-r_{min})+r_{min} T_w(x_{q+1})]}|v_q|^2                             + \frac{|v_q|^2}{2T_w(x_q)}\]
\[=    \frac{ -T_w(x_{q+1})(1-r_{min})-r_{min} T_w(x_{q+1})}{2T_w(x_{q+1})[T_{l,q+1}(1-r_{min})+r_{min} T_w(x_{q+1})]}|v_q|^2                             + \frac{|v_q|^2}{2T_w(x_q)}\]

\[=    \frac{-|v_q|^2}{2[T_{l,q+1}(1-r_{min})+r_{min}T_w(x_{q+1})]}                + \frac{|v_q|^2}{2T_w(x_q)}.\]

By the Definition~\ref{Def:Back time cycle}, $x_{q+1}=x_{q+1}(t,x,v,v_1,\cdots,v_{q})$, thus $T_w(x_{q+1})$ depends on $v_{q}$. In order to explicitly compute the integration over $\mathcal{V}_q$, we need to get rid of the dependence of the $T_w(x_{q+1})$ on $v_{q}$. Then we bound
\begin{equation}\label{eqn: Help to integrate x_p over v_p-1}
\begin{split}
   & \exp\left(\frac{-|v_{q}|^2}{2[T_{l,q+1}(1-r_{min})+r_{min} T_w(x_{q+1})]}\right)\leq \exp\left(\frac{-|v_{q}|^2}{2[T_{l,q+1}(1-r_{min})+r_{min}T_M]}\right)= \exp\left(\frac{-|v_{q}|^2}{2T_{l,q}}\right),
\end{split}
\end{equation}
where we use~\eqref{eqn: definition of T_p}.

Hence by~\eqref{eqn:probability measure}~\eqref{eqn: Formula for R} and~\eqref{eqn: Help to integrate x_p over v_p-1}, we derive
\begin{equation}\label{eqn: int over V_p}
\begin{split}
  &\eqref{eqn: dsigmap}\leq (C_{T_M,\xi})^{2(l-q)} \\
   &  \int_{\mathcal{V}_{q,\perp}} \frac{2}{r_\perp}\frac{|v_{q,\perp}|}{2T_w(x_q)}  e^{-[\frac{1}{2T_{l,q}}-\frac{1}{2T_w(x_q)}-\mathcal{C}^{l-q}t]|v_{q,\perp}|^2}I_0\left(\frac{(1-r_\perp)^{1/2}v_{q,\perp}v_{q-1,\perp}}{T_w(x_q)r_\perp}\right)e^{-\frac{|v_{q,\perp}|^2+(1-r_\perp)|v_{q-1,\perp}|^2}{2T_w(x_q)r_\perp}}  dv_{q,\perp} \\
    & \times\int_{\mathcal{V}_{q,\parallel}}\frac{1}{\pi r_\parallel(2-r_\parallel)(2T_w(x_q))}e^{-[\frac{1}{2T_{l,q}}-\frac{1}{2T_w(x_q)}-\mathcal{C}^{l-q}t]|v_{q,\parallel}|^2}e^{-\frac{1}{2T_w(x_q)}\frac{|v_{q,\parallel}-(1-r_\parallel)v_{q-1,\parallel}|^2}{r_\parallel(2-r_\parallel)}}dv_{q,\parallel}.
\end{split}
\end{equation}
In the third line of~\eqref{eqn: int over V_p}, to apply~\eqref{eqn: coe abc} in Lemma \ref{Lemma: abc}, we set
\[a=-[\frac{1}{2T_{l,q}}-\frac{1}{2T_w(x_{q})}],~ b=\frac{1}{2T_w(x_{q})r_\parallel(2-r_\parallel)},~\e=\mathcal{C}^{l-q}t,~w=(1-r_\parallel)v_{q-1,\parallel}.\]
Taking~\eqref{eqn: coefficient a and b} for comparison, we can replace $\frac{2\xi}{\xi+1}T_M$ by $T_{l,q}$ and replace $t$ by $\mathcal{C}^{l-q}t$. Then we apply the replacement to~\eqref{eqn: b-a-e} and obtain
\[b-a-\e\geq \frac{1}{2T_{l,q}}-\mathcal{C}^{l-q}t\geq \frac{1}{2T_M\frac{2\xi}{\xi+1}}-\mathcal{C}^k t\geq \frac{1}{4T_M},\]
where we take $t^*=t^*(T_M,\xi,\mathcal{C},k)$ to be small enough and $t\leq t^*$. Also we require the $t$ satisfy
\[\frac{\e}{b-a-\e}\leq 4T_M \mathcal{C}^k t\leq 2.\]
We conclude the $t^*$ only depends on the parameter in~\eqref{eqn: t*}. Thus by the same computation as~\eqref{eqn: 1 one} we obtain
\[\frac{b}{b-a-\e}\leq \frac{2T_{l,q}}{T_{l,q}+[\min\{T_w(x)\}-T_{l,q}]r_\parallel(2-r_\parallel)}\leq C_{T_M,\xi},\]
where we use $T_{l,q}\leq \frac{2\xi}{\xi+1}T_M$ from~\eqref{eqn: definition of T_p} and \eqref{eqn: def of r}. $C_{T_M,\xi}$ is defined in~\eqref{eqn: 1 one}.

By the same computation as~\eqref{eqn: 2 two}, we obtain
\[\frac{(a+\e)b}{b-a-\e}= \frac{ab}{b-a}+\frac{ab}{b-a}\frac{\e}{b-a-\e}+\frac{b}{b-a-\e}\e\]
\[\leq  \frac{T_{l,q}-T_w(x_q)}{2T_w(x_q)[T_{l,q}+[T_w(x_q)-T_{l,q}]r_\parallel(2-r_\parallel)]}+\mathcal{C}^{l-q+1}t.\]
Here we use $T_{l,q}\leq \frac{2\xi}{\xi+1}T_M$ and \eqref{eqn: def of r} to obtain
\[\frac{ab}{b-a}\frac{\e}{b-a-\e}+ \frac{b\e}{b-a-\e}\leq \frac{4T_M\big(T_{l,q}-\min\{T_w(x)\}\big)}{2\min\{T_w(x)\}[T_{l,q}+[\min\{T_w(x)\}-T_{l,q}]r_\parallel(2-r_\parallel)]}\mathcal{C}^{l-q}t\]
\[ +\frac{2T_{l,q}}{\frac{2\xi}{\xi+1}T+[\min\{T_w(x)\}-T_{l,q}]r_\parallel(2-r_\parallel)}\mathcal{C}^{l-q}t \leq \mathcal{C}^{l-q+1}t,\]
with $\mathcal{C}$ defined in~\eqref{eqn: cal C}.

Thus by Lemma \ref{Lemma: abc} with $w=(1-r_\parallel)v_{q-1,\parallel}$, the third line of~\eqref{eqn: int over V_p} is bounded by
\[C_{T_M,\xi}\exp\left(\big[ \frac{[T_{l,q}-T_w(x_{q})]}{2T_w(x_{q})[T_{l,q}(1-r_\parallel)^2+r(2-r_\parallel) T_w(x_{q})]} + \mathcal{C}^{l-q+1}t\big]|(1-r_\parallel)v_{q-1,\parallel}|^2\right)\]
\begin{equation}\label{eqn: Vq para}
\leq C_{T_M,\xi}\exp\left(\big[ \frac{[T_{l,q}-T_w(x_{q})][1-r_{min}]}{2T_w(x_{q})[T_{l,q}(1-r_{min})+r_{min} T_w(x_{q})]} + \mathcal{C}^{l-q+1}t\big]|v_{q-1,\parallel}|^2\right).
\end{equation}
By the same computation the second line of~\eqref{eqn: int over V_p} is bounded by
\begin{equation}\label{eqn: Vq perp}
C_{T_M,\xi}\exp\left(\big[ \frac{[T_{l,q}-T_w(x_{q})][1-r_{min}]}{2T_w(x_{q})[T_{l,q}(1-r_{min})+r_{min} T_w(x_{q})]} + \mathcal{C}^{l-q+1}t\big]|v_{q-1,\perp}|^2\right).
\end{equation}
By~\eqref{eqn: Vq para} and~\eqref{eqn: Vq perp}, we derive that
\[\eqref{eqn: int over V_p}\leq (C_{T_M,\xi})^{2(l-q+1)}\exp\left(\big[ \frac{[T_{l,q}-T_w(x_{q})][1-r_{min}]}{2T_w(x_{q})[T_{l,q}(1-r_{min})+r_{min} T_w(x_{q})]} + \mathcal{C}^{l-q+1}t\big]|v_{q-1}|^2\right)=(C_{T_M,\xi})^{2(l-q+1)}\mathcal{A}_{l,q},\]
which is consistent with~\eqref{eqn: boundedness for l-p+1 fold integration} with $p=q$. The induction is valid we derive~\eqref{eqn: boundedness for l-p+1 fold integration}.

Now we focus on~\eqref{eqn: structure}. The first inequality in~\eqref{eqn: structure} follows directly from~\eqref{eqn: boundedness for l-p+1 fold integration} and~\eqref{eqn: ppt for Phi}. For the second inequality, by~\eqref{eqn: upsilon} we have
\[(C_{T_M,\xi})^{2(l-p'+1)} \mathcal{A}_{l,p'}\int_{\prod_{j=p}^{p'-1} \mathcal{V}_j}  \mathbf{1}_{\{t_l>0\}} d\Upsilon_p^{p'-1}\]
\begin{equation}\label{eqn: second ineq}
\leq (C_{T_M,\xi})^{2(l-p'+1)} \mathcal{A}_{l,p'}\int_{\prod_{j=p}^{p'-2} \mathcal{V}_j} \int_{\mathcal{V}_{p'-1}} \mathbf{1}_{\{t_l>0\}}   e^{[\frac{1}{2T_w(x_{p'-1})}-\frac{1}{2T_w(x_{p'})}]|v_{p'-1}|^2} d\sigma(v_{p'-1},v_{p'-2})  d\Upsilon_p^{p'-2}.
\end{equation}
In the proof for~\eqref{eqn: boundedness for l-p+1 fold integration} we have
\[\eqref{eqn: dsigmaq}\leq \eqref{eqn: dsigmap}\leq \eqref{eqn: int over V_p} \leq (C_{T_M,\xi})^{2(l-q+1)}\mathcal{A}_{l,q}.\]
Then by replacing $q$ by $p'-1$ in the estimate $~\eqref{eqn: dsigmaq}\leq (C_{T_M,\xi})^{2(l-q+1)}\mathcal{A}_{l,q}$ we have

\[\eqref{eqn: second ineq}\leq (C_{T_M,\xi})^{2(l-p'+2)}\mathcal{A}_{l,p'-1}\int_{\prod_{j=p}^{p'-2} \mathcal{V}_j} \mathbf{1}_{\{t_l>0\}}  d\Upsilon_p^{p'-2}   .\]
Keep doing this computation until integrating over $\mathcal{V}_p$ we obtain the second inequality in~\eqref{eqn: structure}.

\end{proof}

The next result is the Lemma \ref{lemma: t^k}, which is the smallness of the last term of~\eqref{eqn: formula for H}.
\begin{lemma}\label{lemma: t^k}
Assume
\begin{equation}\label{eqn: assume T}
\frac{\min(T_w(x))}{T_M}>\max\Big(\frac{1-r_\parallel}{2-r_\parallel},\frac{\sqrt{1-r_\perp}-(1-r_\perp)}{r_\perp}\Big).
\end{equation}
 For the last term of~\eqref{eqn: formula for H}, there exists
\begin{equation}\label{eqn: k_0 dependence}
k_0=k_0(\Omega,C_{T_M,\xi},\mathcal{C},T_M,r_\perp,r_\parallel,\min\{T_w(x)\},\xi)\gg 1,
\end{equation}
\begin{equation}\label{eqn: t'}
 t'=t'(k_0,\xi,T_M,\min\{T_w(x)\},\mathcal{C},r_\perp,r_\parallel)\ll 1
\end{equation}
such that for all $t\in [0,t']$, we have
\begin{equation}\label{eqn: 1/2 decay}
   \int_{\prod_{j=1}^{k_0-1}\mathcal{V}_j} \mathbf{1}_{\{t_{k_0}>0\}}d\Sigma_{k_0-1,m}^{k_0}(t_{k_0})  \leq (\frac{1}{2})^{k_0} \mathcal{A}_{k_0-1,1}.
\end{equation}
where $\mathcal{A}_{k_0-1,1}$ is defined in~\eqref{eqn: Elp}.
\end{lemma}

\begin{remark}
The difference between this lemma and Lemma \ref{lemma: boundedness} is that we have the small term $(\frac{1}{2})^{k_0}$. This lemma implies when $k=k_0$ is large enough, the measure of the last term of~\eqref{eqn: formula for H} is small.

\end{remark}

We need several lemmas to prove it.

\begin{lemma}\label{Lemma: (2)}
For $1\leq i\leq k-1$, if
\begin{equation}\label{eqn: 2 condition}
|v_i\cdot n(x_i)|<\delta,
\end{equation}
then
\begin{equation}\label{eqn: 2}
    \int_{\prod_{j=i}^{k-1} \mathcal{V}_j} \mathbf{1}_{\{v_i\in \mathcal{V}_i:|v_i\cdot n(x_i)|\delta\}}   \mathbf{1}_{\{t_k>0\}} d\Phi_{i,m}^{k,k-1}(t_k) \leq \delta    (C_{T_M,\xi})^{2(k-i)}\mathcal{A}_{k-1,i}.
\end{equation}

If
\begin{equation}\label{eqn: b condition}
|v_{i,\parallel}-\eta_{i,\parallel}v_{i-1,\parallel}|>\delta^{-1},
\end{equation}
then
\begin{equation}\label{eqn: case b}
\int_{\prod_{j=i}^{k-1} \mathcal{V}_j}\mathbf{1}_{\{t_k>0\}}\mathbf{1}_{\{|v_{i,\parallel}-\eta_{i,\parallel}v_{i-1,\parallel}|>\delta^{-1}\}}d\Phi_{i,m}^{k,k-1}(t_k)
 \leq \delta    (C_{T_M,\xi})^{2(k-i)}\mathcal{A}_{k-1,i}.
\end{equation}
Here $\eta_{i,\parallel}$ is a constant defined in~\eqref{eqn: eta i para}.

If
\begin{equation}\label{eqn: d condition}
|v_{i,\perp}-\eta_{i,\perp}v_{i-1,\perp}|>\delta^{-1},
\end{equation}
then
\begin{equation}\label{eqn: case d}
\int_{\prod_{j=i}^{k-1} \mathcal{V}_j}\mathbf{1}_{\{t_k>0\}}\mathbf{1}_{\{|v_{i,\perp}-\eta_{i,\perp}v_{i-1,\perp}|>\delta^{-1}\}}d\Phi_{i,m}^{k,k-1}(t_k)
 \leq \delta    (C_{T_M,\xi})^{2(k-i)}\mathcal{A}_{k-1,i}.
\end{equation}
Here $\eta_{i,\perp}$ is a constant defined in~\eqref{eqn: eta i perp}.

\end{lemma}

\begin{proof}
First we focus on~\eqref{eqn: 2}. By~\eqref{eqn: int over V_p} in Lemma \ref{lemma: boundedness}, we can replace $l$ by $k-1$ and replace $q$ by $i$ to obtain
\begin{equation}\label{eqn: int V_i}
\begin{split}
  & \int_{\prod_{j=i}^{k-1} \mathcal{V}_j}    \mathbf{1}_{\{t_k>0\}} d\Phi_{i,m}^{k,k-1}(t_k)\leq(C_{T_M,\xi})^{2(k-i)} \\
   &  \int_{\mathcal{V}_{i,\perp}} \frac{2}{r_\perp}\frac{|v_{i,\perp}|}{2T_w(x_i)}  e^{-[\frac{1}{2T_{k-1,i}}-\frac{1}{2T_w(x_i)}-\mathcal{C}^{k-i}t]|v_{i,\perp}|^2}I_0\left(\frac{(1-r_\perp)^{1/2}v_{i,\perp}v_{i-1,\perp}}{T_w(x_i)r_\perp}\right)e^{\frac{|v_{i,\perp}|^2+(1-r_\perp)|v_{i-1,\perp}|^2}{2T_w(x)r_\perp}}  dv_{i,\perp} \\
    & \times\int_{\mathcal{V}_{i,\parallel}}\frac{1}{\pi r_\parallel(2-r_\parallel)(2T_w(x_i))}e^{-[\frac{1}{2T_{k-1,i}}-\frac{1}{2T_w(x_i)}-\mathcal{C}^{k-i}t]|v_{i,\parallel}|^2}e^{-\frac{1}{2T_w(x_i)}\frac{|v_{i,\parallel}-(1-r_\parallel)v_{i-1,\parallel}|^2}{r_\parallel(2-r_\parallel)}}dv_{i,\parallel}.
\end{split}
\end{equation}
Under the condition~\eqref{eqn: 2 condition}, we consider the second line of~\eqref{eqn: int V_i} with integrating over $\{v_{i,\perp}\in \mathcal{V}_{i,\perp}:|v_i\cdot n(x_i)|<\frac{1-\eta}{2(1+\eta)}\delta\}$. To apply~\eqref{eqn: coe abc perp small} in Lemma \ref{Lemma: perp abc}, we set
\[a=-[\frac{1}{2T_{k-1,i}}-\frac{1}{2T_w(x_{i})}],~ b=\frac{1}{2T_w(x_{i})r_\perp},~\e=\mathcal{C}^{k-i}t,~w=\sqrt{1-r_\perp}v_{i-1,\perp}.\]
Under the condition $|v_i\cdot n(x_i)|<\frac{1-\eta}{2(1+\eta)}\delta$, applying~\eqref{eqn: coe abc perp small} in Lemma \ref{Lemma: perp abc} and using~\eqref{eqn: Vq perp} with $q=i,l=k-1$, we bound the second line of~\eqref{eqn: int V_i} by
\begin{equation}\label{eqn: perp small}
\delta C_{T_M,\xi}\exp\left(\big[ \frac{[T_{k-1,i}-T_w(x_{i})][1-r_{min}]}{2T_w(x_{i})[T_{k-1,i}(1-r_{min})+r_{min} T_w(x_{i})]} + \mathcal{C}^{k-i}t\big]|v_{i-1,\perp}|^2\right).
\end{equation}
Taking~\eqref{eqn: Vq perp} for comparison, we conclude the second line of~\eqref{eqn: int V_i} provides one more constant term $\delta$. The third line of~\eqref{eqn: int V_i} is bounded by~\eqref{eqn: Vq para} with $q=i,l=k-1$. Therefore, we derive~\eqref{eqn: 2}.

Then we focus on~\eqref{eqn: case b}. We consider the third line of~\eqref{eqn: int V_i}. To apply~\eqref{eqn: coe abc small} in Lemma \ref{Lemma: abc}, we set
\begin{equation}\label{eqn: abe}
     a=-\frac{1}{2T_{k-1,i}} +\frac{1}{2T_w(x_i)},\quad b=\frac{1}{2T_w(x_i)r_\parallel(2-r_\parallel)},\quad \e=\mathcal{C}^{k-i}t,~w=(1-r_\parallel)v_{i-1,\parallel}.
\end{equation}
We define
\begin{equation}\label{eqn: B i para def}
B_{i,\parallel}:=b-a-\e.
\end{equation}
In regard to~\eqref{eqn: coe abc small},
\[    \frac{b}{b-a-\e}w=\frac{b}{b-a}[1+\frac{\e}{b-a-\e}] w. \]
By~\eqref{eqn: abe},
\[\frac{b}{b-a}=\frac{T_{k-1,i}}{T_{k-1,i}(1-r_\parallel)^2+T_w(x_i)r_\parallel(2-r_\parallel)},\quad \frac{\e}{b-a-\e}=\frac{\mathcal{C}^{k-i}t}{B_{i,\parallel}}.\]
Thus we obtain
\begin{equation}\label{eqn: constant for the t^k}
    \frac{b}{b-a-\e}w=\eta_{i,\parallel}v_{i-1,\parallel},
\end{equation}
where we define
\begin{equation}\label{eqn: eta i para}
\eta_{i,\parallel}:=\frac{T_{k-1,i}[1+\mathcal{C}^{k-i}t/B_{i,\parallel}]}{T_{k-1,i}(1-r_\parallel)^2+T_w(x_i)r_\parallel(2-r_\parallel)}(1-r_\parallel).
\end{equation}
Thus under the condition~\eqref{eqn: b condition}, applying~\eqref{eqn: coe abc small} in Lemma \ref{eqn: coe abc} with $\frac{b}{b-a-\e}w=\eta_{i,\parallel}v_{i-1,\parallel}$ and using~\eqref{eqn: Vq para} with $q=i,l=k-1$, we bound the third line of~\eqref{eqn: int V_i} by
\[\delta C_{T_M,\xi}\exp\left(\big[ \frac{[T_{k-1,i}-T_w(x_{i})][1-r_{min}]}{2T_w(x_{i})[T_{k-1,i}(1-r_{min})+r_{min} T_w(x_{i})]} + \mathcal{C}^{k-i}t\big]|v_{i-1,\parallel}|^2\right).\]
By the same computation in Lemma \ref{Lemma: (2)}, we derive~\eqref{eqn: case b} because of the extra constant $\delta$.

Last we focus on~\eqref{eqn: case d}. We consider the second line of~\eqref{eqn: int V_i} with integrating over $\{v_{i,\perp}:v_{i,\perp}\in \mathcal{V}_{i,\perp},|v_{i,\perp}|>\frac{1+\eta}{1-\eta}\delta^{-1}\}$. To apply~\eqref{eqn: coe abc perp small} in Lemma \ref{Lemma: integrate normal small}, we set
\begin{equation}\label{eqn: abe perp}
     a=-\frac{1}{2T_{k-1,i}} +\frac{1}{2T_w(x_i)},\quad b=\frac{1}{2T_w(x_i)r_\perp},\quad \e=\mathcal{C}^{k-i}t,~ w=\sqrt{1-r_\perp}v_{i-1,\perp}.
\end{equation}
Define
\begin{equation}\label{eqn: B i perp}
B_{i,\perp}:=b-a-\e.
\end{equation}
By the same computation as~\eqref{eqn: constant for the t^k},
\[\frac{b}{b-a-\e}w=\eta_{i,\perp}v_{i-1,\perp},\]
where we define
\begin{equation}\label{eqn: eta i perp}
\eta_{i,\perp}:= \frac{T_{k-1,i}[1+\frac{\mathcal{C}^{k-i}t}{B_{i,\perp}}]}{T_{k-1,i}(1-r_\perp)+T_w(x_i)r_\perp}\sqrt{1-r_\perp}.
\end{equation}
Thus under the condition~\eqref{eqn: d condition}, applying~\eqref{eqn: coe perp small 2} in Lemma~\ref{Lemma: integrate normal small} with $\frac{b}{b-a-\e}w=\eta_{i,\perp}v_{i-1,\perp}$ and using~\eqref{eqn: Vq perp} with $q=i,l=k-1$, we bound the second line of~\eqref{eqn: int V_i} by
\[\delta C_{T_M,\xi}\exp\left(\big[ \frac{[T_{k-1,i}-T_w(x_{i})][1-r_{min}]}{2T_w(x_{i})[T_{k-1,i}(1-r_{min})+r_{min} T_w(x_{i})]} + \mathcal{C}^{k-i}t\big]|v_{i-1,\perp}|^2\right).\]
Then we derive~\eqref{eqn: case b} because of the extra constant $\delta$.

\end{proof}

\begin{lemma}\label{Lemma:  (a)(c)}
For $\eta_{i,\parallel}$ and $\eta_{i,\perp}$ defined in Lemma \ref{Lemma: (2)}, we suppose there exists $\eta<1$ such that
\begin{equation}\label{eqn: eta condition}
  \max\{\eta_{i,\parallel},\eta_{i,\perp}\}<\eta<1.
\end{equation}
Then If
\begin{equation}\label{eqn: (a) condition}
  |v_{i,\parallel}|>\frac{1+\eta}{1-\eta}\delta^{-1} \text{ and } |v_{i,\parallel}-\eta_{i,\parallel}v_{i-1,\parallel}|<\delta^{-1},
\end{equation}
we have
\begin{equation}\label{eqn: (a)}
  |v_{i-1,\parallel}|>|v_{i,\parallel}|+\delta^{-1}.
\end{equation}

Also if
\begin{equation}\label{eqn: (c) condition}
  |v_{i,\perp}|>\frac{1+\eta}{1-\eta}\delta^{-1} \text{ and } |v_{i,\perp}-\eta_{i,\perp}v_{i-1,\perp}|<\delta^{-1},
\end{equation}
then we have
\begin{equation}\label{eqn: (c)}
  |v_{i-1,\perp}|>|v_{i,\perp}|+\delta^{-1}.
\end{equation}

\end{lemma}
\begin{remark}
Lemma \ref{Lemma: (2)} includes the cases that are controllable because of the small magnitude number $\delta$, which is the "good" factor for us to establish the Lemma \ref{lemma: t^k}. This lemma discusses those "bad" cases, which are the main difficulty since they do not directly provide $\delta$.
\end{remark}

\begin{proof}
Under the condition~\eqref{eqn: (a) condition} we have
\[\eta_{i,\parallel}|v_{i-1,\parallel}|> |v_{i,\parallel}|-\delta^{-1}.\]
Thus we derive
\[|v_{i-1,\parallel}|>|v_{i,\parallel}|+\frac{1-\eta_{i,\parallel}}{\eta_{i,\parallel}}|v_{i,\parallel}|-\frac{1}{\eta_{i,\parallel}}\delta^{-1}\]
\[>|v_{i,\parallel}|+\frac{1-\eta_{i,\parallel}}{\eta_{i,\parallel}}\frac{1+\eta}{1-\eta}\delta^{-1}-\frac{1}{\eta_{i,\parallel}}\delta^{-1}
\]
\[>|v_{i,\parallel}|+\frac{1-\eta_{i,\parallel}}{\eta_{i,\parallel}}\frac{1+\eta_{i,\parallel}}{1-\eta_{i,\parallel}}\delta^{-1}-\frac{1}{\eta_{i,\parallel}}\delta^{-1}
\]
\[>|v_{i,\parallel}|+\frac{1+\eta_{i,\parallel}}{\eta_{i,\parallel}}\delta^{-1}-\frac{1}{\eta_{i,\parallel}}\delta^{-1}>|v_{i,\parallel}|+\delta^{-1},\]
where we use $|v_{i,\parallel}|>\frac{1+\eta}{1-\eta}\delta^{-1}$ in the second line and $1>\eta\geq \eta_{i,\parallel}$ in the third line. Then we obtain~\eqref{eqn: (a)}.

Under the condition~\eqref{eqn: (c) condition}, we apply the same computation above to obtain~\eqref{eqn: (c)}.

\end{proof}

\begin{lemma}\label{Lemma: accumulate}
Suppose there are $n$ number of $v_j$ such that
\begin{equation}\label{eqn: satisfy condition}
|v_{j,\parallel}-\eta_{j,\parallel}v_{j-1,\parallel}|\geq \delta^{-1},
\end{equation}
and also suppose the index $j$ in these $v_j$ are $i_1<i_2<\cdots<i_n$, then
\begin{equation}\label{eqn: claim M}
\int_{\prod_{j={i_1}}^{k-1} \mathcal{V}_j}\mathbf{1}_{\{t_k>0\}}\mathbf{1}_{\{\text{~\eqref{eqn: satisfy condition} holds for $j=i_1,i_2,\cdots, i_n$}\}}d\Phi_{i_1,m}^{k,k-1}(t_k) \leq  (\delta)^{n}   (C_{T_M,\xi})^{2(k-i_1)}\mathcal{A}_{k-1,i_1}.
\end{equation}

\end{lemma}

\begin{proof}
By~\eqref{eqn: structure} in Lemma 2 with $l=k-1$, $p=i_1$, $p'=i_n$ and using~\eqref{eqn: case b} with $i=i_n$, we have
\[\int_{\prod_{j=i_1}^{k-1} \mathcal{V}_j}\mathbf{1}_{\{t_k>0\}}\mathbf{1}_{\{\text{~\eqref{eqn: satisfy condition} holds for $j=i_1,\cdots,i_n$}\}}d\Phi_{i_1,m}^{k,k-1}(t_k)\]
\[ \leq \delta (C_{T_M,\xi})^{2(k-i_n)} \mathcal{A}_{k-1,i_n}\int_{\prod_{j=i_1}^{i_n-1} \mathcal{V}_j} \mathbf{1}_{\{t_k>0\}}\mathbf{1}_{\{\text{~\eqref{eqn: satisfy condition} holds for $j=i_1,\cdots,i_{n-1}$}\}}        d\Upsilon_{i_1}^{i_n-1}\]
\begin{equation}\label{eqn: split iM}
=\delta (C_{T_M,\xi})^{2(k-i_n)} \mathcal{A}_{k-1,i_n}\int_{\prod_{j=i_1}^{i_{n-1}-1}\mathcal{V}_j}\int_{\prod_{j=i_{n-1}}^{(i_n)-1} \mathcal{V}_j} \mathbf{1}_{\{t_k>0\}}\mathbf{1}_{\{\text{~\eqref{eqn: satisfy condition} holds for $j=i_1,\cdots,i_{n-1}$}\}}      d\Upsilon_{i_{n-1}}^{(i_{n})-1}d\Upsilon_{i_1}^{i_{n-1}-1}.
\end{equation}
Again by~\eqref{eqn: structure} and~\eqref{eqn: case b} with $i=i_{n-1}$ we have
\[\eqref{eqn: split iM}\leq \delta^2 (C_{T_M,\xi})^{2(k-i_{n-1})}\mathcal{A}_{k-1,i_{n-1}}\int_{\prod_{j=i_1}^{i_{n-1}-1} \mathcal{V}_j} \mathbf{1}_{\{t_k>0\}}\mathbf{1}_{\{\text{~\eqref{eqn: satisfy condition} holds for $j=i_1,\cdots,i_{n-2}$}\}}        d\Upsilon_{i_1}^{i_{n-1}-1}.\]
Keep doing this computation until integrating over $\mathcal{V}_{i_1}$ we derive~\eqref{eqn: claim M}.

\end{proof}

\begin{lemma}\label{Lemma: Step3}
For $0<\delta\ll 1$, we define
\begin{equation}\label{eqn: decom}
  \mathcal{V}_{j}^{\delta}:=\{v_j\in \mathcal{V}_j:|v_j\cdot n(x_j)|>\delta,|v_j|\leq \delta^{-1}\}.
\end{equation}
For the sequence $\{v_1,v_2,\cdots,v_{k-1}\}$, consider a subsequence  $\{v_{l+1},v_{l+2},\cdots,v_{l+L}\}$ with $l+1<l+L\leq k-1$ as follows:
\begin{equation}\label{eqn: sequence}
  \underbrace{v_{l}}_{\in \mathcal{V}_l^{\frac{1-\eta}{2(1+\eta)}\delta}},\underbrace{v_{l+1},v_{l+2}\cdots v_{l+L}}_{\text{all}\in \mathcal{V}_{l+j}\backslash \mathcal{V}_{l+j}^{\frac{1-\eta}{2(1+\eta)}\delta}},\quad \quad\underbrace{v_{l+L+1}}_{\in \mathcal{V}_{l+L+1}^{\frac{1-\eta}{2(1+\eta)}\delta}}.
\end{equation}
In~\eqref{eqn: sequence}, if $L\geq 100\frac{1+\eta}{1-\eta}$, then we have
\begin{equation}\label{eqn: Step3}
\int_{\prod_{j={l}}^{k-1} \mathcal{V}_j}\mathbf{1}_{\{t_k>0\}}\mathbf{1}_{\{v_{l+j}\in \mathcal{V}_{l+j}\backslash \mathcal{V}_{l+j}^{\frac{1-\eta}{2(1+\eta)}\delta} \text{ for } 1\leq j\leq L\}}d\Phi_{l,m}^{k,k-1}(t_k) \leq  (3\delta)^{L/2}   (C_{T_M,\xi})^{2(k-l)}\mathcal{A}_{k-1,l}.
\end{equation}
Here the $\eta$ satisfies the condition~\eqref{eqn: eta condition}.

\end{lemma}
\begin{remark}
In this lemma we combine the estimates and properties in Lemma \ref{Lemma: (2)} and Lemma \ref{Lemma:  (a)(c)}. In the proof we will address the difficulty stated in Lemma \ref{Lemma:  (a)(c)} to obtain the key factor $(3\delta)^{L/2}$.

\end{remark}

\begin{proof}
By the definition~\eqref{eqn: decom} we have
\[\mathcal{V}_{i}\backslash \mathcal{V}_{i}^{\frac{1-\eta}{2(1+\eta)}\delta}  =\{v_i\in \mathcal{V}_i:|v_i\cdot n(x_i)|<\frac{1-\eta}{2(1+\eta)}\delta \text{ or }|v_i|\geq \frac{2(1+\eta)}{1-\eta}\delta^{-1}\}.\]
Here we summarize the result of Lemma \ref{Lemma: (2)} and Lemma \ref{Lemma:  (a)(c)}.
With $\frac{1-\eta}{1+\eta}\delta<\delta$, when $v_i\in \mathcal{V}_i\backslash \mathcal{V}_i^{\frac{1-\eta}{2(1+\eta)}\delta}$
\begin{enumerate}

  \item When $|v_i\cdot n(x_i)|<\frac{1-\eta}{2(1+\eta)}\delta$, then we have~\eqref{eqn: 2}.

  \item When $|v_{i}|>\frac{2(1+\eta)}{1-\eta}\delta^{-1}$,
   \begin{enumerate}
     \item when $|v_{i,\parallel}|>\frac{1+\eta}{1-\eta}\delta^{-1}$, if $|v_{i,\parallel}-\eta_{i,\parallel}v_{i-1,\parallel}|<\delta^{-1}$, then $|v_{i-1,\parallel}|>|v_{i,\parallel}|+\delta^{-1}$. \\
     \item when $|v_{i,\parallel}|>\frac{1+\eta}{1-\eta}\delta^{-1}$, if $|v_{i,\parallel}-\eta_{i,\parallel}v_{i-1,\parallel}|\geq \delta^{-1}$, then we have~\eqref{eqn: case b}. \\
     \item when $|v_{i,\perp}|>\frac{1+\eta}{1-\eta}\delta^{-1}$, if $|v_{i,\perp}-\eta_{i,\perp}v_{i-1,\perp}|<\delta^{-1}$, then $|v_{i-1,\perp}|>|v_{i,\perp}|+\delta^{-1}$ .\\
     \item when $|v_{i,\perp}|>\frac{1+\eta}{1-\eta}\delta^{-1}$, if $|v_{i,\perp}-\eta_{i,\perp}v_{i-1,\perp}|\geq \delta^{-1}$, then we have~\eqref{eqn: case d}.\\
   \end{enumerate}

\end{enumerate}

We define $\mathcal{W}_{i,\delta}$ as the space that provides the smallness:
\[\mathcal{W}_{i,\delta}:=\{v_i\in \mathcal{V}_i:|v_{i,\perp}|<\frac{1-\eta}{2(1+\eta)}\delta\}\bigcup \{v_i\in \mathcal{V}_i:|v_{i,\perp}|>\frac{1+\eta}{1-\eta}\delta^{-1}\text{ and }|v_{i,\perp}-\eta_{i,\perp}v_{i-1,\perp}|>\delta^{-1}\}\]
\[\bigcup \{v_i\in \mathcal{V}_i:|v_{i,\parallel}|>\frac{1+\eta}{1-\eta}\delta^{-1}\text{ and }|v_{i,\parallel}-\eta_{i,\parallel}v_{i-1,\parallel}|>\delta^{-1}\}.\]
Then we have
\begin{equation}\label{eqn: subset}
  \begin{split}
     & \mathcal{V}_{i}\backslash \mathcal{V}_{i}^{\frac{1-\eta}{2(1+\eta)}\delta} \subset \mathcal{W}_{i,\delta} \bigcup \{v_{i,\perp}\in \mathcal{V}_{i,\perp}|v_{i,\perp}|>\frac{1+\eta}{1-\eta}\delta^{-1}~\text{and}~|v_{i,\perp}-\eta_{i,\perp}v_{i-1,\perp}|<\delta^{-1}\}
 \\
      & \bigcup \{v_{i,\parallel}\in \mathcal{V}_{i,\parallel}|v_{i,\parallel}|>\frac{1+\eta}{1-\eta}\delta^{-1}~\text{and}~|v_{i,\parallel}-\eta_{i,\parallel}v_{i-1,\parallel}|<\delta^{-1}\}.
  \end{split}
\end{equation}
By~\eqref{eqn: 2},~\eqref{eqn: case b} and~\eqref{eqn: case d} with $\frac{1-\eta}{1+\eta}\delta<\delta$, we obtain
\begin{equation}\label{eqn: V_i,delta}
  \int_{\prod_{j=i}^{k-1}\mathcal{V}_j}  \mathbf{1}_{\{v_i\in \mathcal{W}_{i,\delta}\}}   \mathbf{1}_{\{t_k>0\}}  d\Phi_{i,m}^{k,k-1}(t_k) \leq 3\delta    (C_{T_M,\xi})^{2(k-i)}\mathcal{A}_{k-1,i}.
\end{equation}

For the subsequence $\{v_{l+1},\cdots,v_{l+L}\}$ in~\eqref{eqn: sequence}, when the number of $v_j\in \mathcal{W}_{j,\delta}$ is larger than $L/2$, by~\eqref{eqn: claim M} in Lemma \ref{Lemma: accumulate} with $n=L/2$ and replacing the condition~\eqref{eqn: satisfy condition} by $v_j\in \mathcal{W}_{j,\delta}$, we obtain
\[\int_{\prod_{j=l}^{k-1} \mathcal{V}_j}   \mathbf{1}_{\{\text{Number of }v_j\in \mathcal{W}_{j,\delta} \text{ is large than }L/2\}}    \mathbf{1}_{\{t_k>0\}} d\Phi_{l,m}^{k,k-1}(t_k)\]
\begin{equation}\label{eqn: 3delta}
  \leq (3\delta)^{L/2}    (C_{T_M,\xi})^{2(k-l_i)}\mathcal{A}_{k-1,l}.
\end{equation}
We finish the discussion with the case(1),(2b),(2d). Then we focus on the case (2a),(2c).

When the number of $v_j \notin \mathcal{W}_{j,\delta}$ is larger than $L/2$, by~\eqref{eqn: subset} we further consider two cases. The first case is that the number of $v_j\in \{v_j:|v_{j,\parallel}|>\frac{1+\eta}{1-\eta}\delta^{-1}~\text{and}~|v_{j,\parallel}-\eta_{j,\parallel}v_{j-1,\parallel}|<\delta^{-1}\}$ is larger than $L/4$. According to the relation of $v_{j,\parallel}$ and $v_{j-1,\parallel}$, we categorize them into
\begin{description}
  \item[Set1] $\{v_j\notin \mathcal{W}_{j,\delta}:|v_{j,\parallel}|>\frac{1+\eta}{1-\eta}\delta^{-1}~\text{and}~|v_{j,\parallel}-\eta_{j,\parallel}v_{j-1,\parallel}|<\delta^{-1}\}$.
\end{description}
Denote $M=|\text{Set1}|$ and the corresponding index in Set1 as $j=p_1,p_2,\cdots,p_{M}$. Then we have
\begin{equation}\label{eqn: Mi'}
  L/4\leq M\leq L.
\end{equation}
By~\eqref{eqn: (a)} in Lemma \ref{Lemma:  (a)(c)}, for those $v_{p_j}$, we have
\begin{equation}\label{eqn: increase large}
|v_{p_j,\parallel}|-|v_{p_j-1,\parallel}|<-\delta^{-1}.
\end{equation}

\begin{description}
  \item[Set2]$\{v_j\in \mathcal{V}_j\backslash \mathcal{V}_j^{\frac{1-\eta}{2(1+\eta)\delta}}:|v_{j,\parallel}|\leq |v_{j-1,\parallel}|\leq |v_{j,\parallel}|+\delta^{-1}\}$.
\end{description}

Denote $\mathcal{M}=|\text{Set2}|$ and the corresponding index in Set2 as $j=q_1,q_2,\cdots,q_{\mathcal{M}}$. By~\eqref{eqn: Mi'} we have
\begin{equation}\label{eqn: mathcal M}
1\leq \mathcal{M}\leq L-M\leq \frac{3}{4}L.
\end{equation}
Then for those $v_{q_j}$ we define
\begin{equation}\label{eqn: ai def}
a_j:=|v_{q_j,\parallel}|-|v_{q_j-1,\parallel}|>0.
\end{equation}

 \begin{description}

  \item[Set3] $\{v_j\in \mathcal{V}_j\backslash \mathcal{V}_j^{\frac{1-\eta}{2(1+\eta)\delta}}:|v_{j,\parallel}|\leq |v_{j-1,\parallel}|\leq |v_{j,\parallel}|+\delta^{-1}\}$.
\end{description}

Denote $N=|\text{Set3}|$ and the corresponding index in Set3 as $j=o_1,o_2,\cdots,o_N$. Then for those $o_j$, we have
\begin{equation}\label{eqn: increase small}
|v_{o_j,\parallel}|\leq |v_{o_j-1,\parallel}|\leq |v_{o_j,\parallel}|+\delta^{-1}.
\end{equation}

From~\eqref{eqn: sequence}, we have $v_{l}\in \mathcal{V}_{l}^{\frac{1-\eta}{2(1+\eta)}\delta}$ and $v_{l+L+1}\in \mathcal{V}_{l+L+1}^{\frac{1-\eta}{2(1+\eta)}\delta}$, thus we can obtain
\begin{equation}\label{eqn: xiangjian}
 -\frac{2(1+\eta)}{1-\eta}\delta^{-1}<|v_{l+L+1,\parallel}|-|v_{l,\parallel}|= \sum_{j=1}^{L+1} |v_{l+j,\parallel}|-|v_{l+j-1,\parallel}|.
\end{equation}
By~\eqref{eqn: increase large},~\eqref{eqn: ai def} and~\eqref{eqn: increase small}, we derive that
\[\frac{-2(1+\eta)}{1-\eta}\delta^{-1}<\sum_{j=1}^{M} \big(|v_{p_j,\parallel}|-|v_{p_j-1,\parallel}|\big)+\sum_{j=1}^{\mathcal{M}} \big(|v_{q_j,\parallel}|-|v_{q_j-1,\parallel}|\big)+\sum_{j=1}^{N} \big(|v_{o_j,\parallel}|-|v_{o_j-1,\parallel}|\big) \]
\[\leq -M \delta^{-1}+\sum_{j=1}^{\mathcal{M}}a_j.\]
Therefore, by $L\geq 100\frac{1+\eta}{1-\eta}$ and~\eqref{eqn: Mi'}, we obtain
\[\frac{2(1+\eta)}{1-\eta}\delta^{-1}\leq \frac{L}{10}\delta^{-1}\leq \frac{M}{2}\delta^{-1}\]
and thus
\begin{equation}\label{eqn: ai sum}
 \sum_{j=1}^{\mathcal{M}}a_j\geq M\delta^{-1}-\frac{2(1+\eta)}{1-\eta}\delta^{-1}>\frac{M\delta^{-1}}{2}.
\end{equation}
We focus on integrating over $\mathcal{V}_{q_i}$, those index satisfy~\eqref{eqn: ai def}. Let $1\leq i\leq \mathcal{M}$, we consider the third line of~\eqref{eqn: int V_i} with $i=q_i$ and with integrating over $\{v_{q_i,\parallel}\in \mathcal{V}_{q_i,\parallel}:|v_{q_i,\parallel}|-|v_{q_i-1,\parallel}|= a_i\}$. To apply~\eqref{eqn: coe abc smaller} in Lemma \ref{Lemma: abc}, we set
\[a=-\frac{1}{2T_{k-1,q_i}}+\frac{1}{2T_w(x_{q_i})},\quad b=\frac{1}{2T_w(x_{q_i})r_\parallel(2-r_\parallel)}, \quad \e=\mathcal{C}^{k-q_i}t.\]
By the same computation as~\eqref{eqn: B i para}, we have
\begin{equation}\label{eqn: k_1}
  a+\e-b=-\frac{1}{2T_{k-1,q_i}}+\frac{1}{2T_w(x_{q_i})}-\frac{1}{2T_w(x_{q_i})r_\parallel(2-r_\parallel)}+\mathcal{C}^{k-q_i}t <-\frac{1}{4T_M}.
\end{equation}
Then we use $\eta_{q_i,\parallel}<1$ to obtain
\begin{equation}\label{eqn: bound a_i}
 \mathbf{1}_{\{|v_{q_i,\parallel}|-|v_{q_i-1,\parallel}|= a_i\}}\leq  \mathbf{1}_{\{|v_{q_i,\parallel}|-\eta_{q_i,\parallel}|v_{q_i-1,\parallel}|>a_i\}}  \leq  \mathbf{1}_{\{|v_{q_i,\parallel}-\eta_{q_i,\parallel}v_{q_i-1,\parallel}|>a_i\}}.
\end{equation}
By~\eqref{eqn: coe abc smaller} in Lemma \ref{Lemma: abc} and~\eqref{eqn: bound a_i}, we apply~\eqref{eqn: Vq para} with $q=q_i$ to bound the third line of~\eqref{eqn: int V_i}( the integration over $\mathcal{V}_{q_i,\parallel}$ ) by
\begin{equation}\label{eqn: 4aiTM}
e^{-\frac{a_i^2}{4T_M}} C_{T_M,\xi}\exp\left(\big[ \frac{[T_{k-1,q_i}-T_w(x_{q_i})][1-r_{min}]}{2T_w(x_{q_i})[T_{k-1,q_i}(1-r_{min})+r_{min} T_w(x_{q_i})]} + \mathcal{C}^{k-q_i}t\big]|v_{q_i-1,\parallel}|^2\right).
\end{equation}
Hence by the constant in~\eqref{eqn: 4aiTM} we draw a similar conclusion as~\eqref{eqn: V_i,delta}:
\begin{equation}\label{eqn: case a_i}
   \int_{\prod_{j=q_i}^{k-1} \mathcal{V}_j}\mathbf{1}_{\{t_k>0\}}\mathbf{1}_{\{|v_{q_i,\parallel}|-|v_{q_i-1,\parallel}|= a_i\}}d\Phi_{q_i,m}^{k,k-1}(t_k)
\leq e^{-\frac{a_i^2}{4T_M}}  (C_{T_M,\xi})^{2(k-q_i)}\mathcal{A}_{k-1,q_i}.
\end{equation}
Therefore, by Lemma \ref{Lemma: accumulate}, after integrating over $\mathcal{V}_{q_1,\parallel},\mathcal{V}_{q_2,\parallel},\cdots,\mathcal{V}_{q_\mathcal{M},\parallel}$ we obtain an extra constant
\[e^{-[a_i^2+a_2^2+\cdots +a_{\mathcal{M}}^2]/4T_M}\leq e^{-[a_i+a_2+\cdots +a_{\mathcal{M}}]^2/(4T_M\mathcal{M})}\leq e^{-[M\delta^{-1}/2]^2/(4T_M\mathcal{M})}\]
\[\leq e^{-[\frac{L}{8}\delta^{-1}]^2/(4T_M\frac{3}{4}L)}\leq e^{-\frac{1}{96T_M}L(\delta^{-1})^2}\leq e^{-L\delta^{-1}}.\]
Here we use~\eqref{eqn: ai sum} in the last step of first line and use~\eqref{eqn: Mi'},~\eqref{eqn: mathcal M} in the first step of second line and take $\delta\ll 1$ in the last step of second line. Then $e^{-L\delta^{-1}}$ is smaller than $(3\delta)^{L/2}$ in~\eqref{eqn: 3delta} and we conclude
\begin{equation}\label{eqn: 3delta 1}
  \int_{\prod_{j=l}^{k-1} \mathcal{V}_j}   \mathbf{1}_{\{ M=|\text{Set1}|\geq L/4\}}    \mathbf{1}_{\{t_k>0\}} d\Phi_{l,m}^{k,k-1}(t_k)\leq (3\delta)^{L/2}    (C_{T_M,\xi})^{2(k-l_i)}\mathcal{A}_{k-1,l}.
\end{equation}

The second case is that the number of $v_j\in \{v_j\notin \mathcal{W}_{j,\delta}:|v_{j,\perp}|>\frac{1+\eta}{1-\eta}\delta^{-1}\}$ is larger than $L/4$. We categorize $v_{j,\perp}$ into

\begin{description}
  \item[Set4] $\{v_j\notin \mathcal{W}_{j,\delta}:|v_{j,\perp}|>\frac{1+\eta}{1-\eta}\delta^{-1}~\text{and}~|v_{j,\perp}-\eta_{j,\perp}v_{j-1,\perp}|<\delta^{-1}\}$.
\end{description}

\begin{description}
  \item[Set5]$\{v_j\in \mathcal{V}_j\backslash \mathcal{V}_j^{\frac{1-\eta}{2(1+\eta)\delta}}:|v_{j,\perp}|>|v_{j-1,\perp}|\}$.
\end{description}

 \begin{description}

  \item[Set6] $\{v_j\in \mathcal{V}_j\backslash \mathcal{V}_j^{\frac{1-\eta}{2(1+\eta)\delta}}:|v_{j,\perp}|\leq |v_{j-1,\perp}|\leq |v_{j,\perp}|+\delta^{-1}\}$.
\end{description}
Denote $|\text{Set4}|=M_1$ and the corresponding index as $p'_1,p'_2,\cdots,p'_{M_1}$, $|\text{Set5}|=\mathcal{M}_1$ and the corresponding index as $q'_1,q'_2,\cdots,q'_{\mathcal{M}_1}$, $|\text{Set6}|=N_1$ and the corresponding index as $o'_1,o'_2,\cdots,o'_{N_1}$. Also define $b_j:=|v_{q'_j,\perp}|-|v_{q'_j-1,\perp}|$. By the same computation as~\eqref{eqn: ai sum}, we have
\[ \sum_{j=1}^{\mathcal{M}_1}b_j\geq M_1\delta^{-1}-\frac{2(1+\eta)}{1-\eta}\delta^{-1}>\frac{M_1\delta^{-1}}{2}.\]
We focus on the integration over $v_{q'_j}$. Let $1\leq i\leq \mathcal{M}_1$, we consider the second line of~\eqref{eqn: int V_i} with $i=q'_i$ and with integrating over $\{v_{q'_i,\perp}\in \mathcal{V}_{q'_i,\perp}:|v_{q'_i,\perp}|-|v_{q'_i-1,\perp}|= b_i\}$. To apply~\eqref{eqn: coe perp smaller 2} in Lemma \ref{Lemma: abc}, we set
\[a=-\frac{1}{2T_{k-1,q_i'}}+\frac{1}{2T_w(x_{q_i'})},\quad b=\frac{1}{2T_w(x_{q_i'})r_\perp}, \quad \e=\mathcal{C}^{k-q_i'}t.\]
By the same computation as~\eqref{eqn: B i para}, we have
\begin{equation}\label{eqn: k_1}
  a+\e-b=-\frac{1}{2T_{k-1,q_i'}}+\frac{1}{2T_w(x_{q_i'})}-\frac{1}{2T_w(x_{q_i'})r_\perp}+\mathcal{C}^{k-q_i'}t <-\frac{1}{4T_M}.
\end{equation}
Similar to~\eqref{eqn: bound a_i}, we have
\[ \mathbf{1}_{\{|v_{q'_i,\perp}|-|v_{q'_i-1,\perp}|= b_i\}}\leq \mathbf{1}_{\{|v_{q'_i,\perp}-\eta_{q'_i,\perp}v_{q'_i-1,\perp}|>b_i\}}.\]
Hence by~\eqref{eqn: coe perp smaller 2} in Lemma \ref{Lemma: integrate normal small} and applying~\eqref{eqn: Vq perp}, we bound the integration over $\mathcal{V}_{q'_i,\perp}$ by
\[e^{-\frac{b_i^2}{16T_M}} C_{T_M,\xi}\exp\left(\big[ \frac{[T_{k-1,q_i'}-T_w(x_{q_i'})][1-r_{min}]}{2T_w(x_{q_i'})[T_{k-1,q_i'}(1-r_{min})+r_{min} T_w(x_{q_i'})]} + \mathcal{C}^{k-q_i'}t\big]|v_{q_i'-1,\perp}|^2\right).\]
Therefore,
\[\int_{\prod_{j=q_i'}^{k-1} \mathcal{V}_j}\mathbf{1}_{\{t_k>0\}}\mathbf{1}_{\{|v_{q_i',\perp}|-|v_{q_i'-1,\perp}|= b_i\}}d\Phi_{q_i',m}^{k,k-1}(t_k)\leq e^{-\frac{b_i^2}{16T_M}}  (C_{T_M,\xi})^{2(k-q_i')} \mathcal{A}_{k-1,q_i'}.\]
 The integration over $\mathcal{V}_{q'_1,\perp},\mathcal{V}_{q'_2,\perp},\cdots,\mathcal{V}_{q'_{\mathcal{M}_1},\perp}$ provides an extra constant
\[e^{-[b_1^2+b_2^2+\cdots +b_{\mathcal{M}_1}^2]/16T_M}\leq e^{-\frac{1}{400T_M}L (\delta^{-1})^2}\leq e^{-L\delta^{-1}},\]
where we set $\delta\ll 1$ in the last step. Then $e^{-L\delta^{-1}}$ is smaller than $(3\delta)^{L/2}$ in~\eqref{eqn: 3delta} and we conclude
\begin{equation}\label{eqn: 3delta 2}
\int_{\prod_{j=l}^{k-1} \mathcal{V}_j}   \mathbf{1}_{\{ M_1=|\text{Set4}|\geq L/4\}}    \mathbf{1}_{\{t_k>0\}} d\Phi_{l,m}^{k,k-1}(t_k) \leq (3\delta)^{L/2}    (C_{T_M,\xi})^{2(k-l)}\mathcal{A}_{k-1,l}.
\end{equation}

Finally collecting~\eqref{eqn: 3delta},~\eqref{eqn: 3delta 1} and~\eqref{eqn: 3delta 2} we derive the lemma.

\end{proof}

Now we prove the Lemma \ref{lemma: t^k}.

\begin{proof}[Proof of Lemma \ref{lemma: t^k}]

\textbf{Step 1}

To prove~\eqref{eqn: 1/2 decay} holds for the C-L boundary condition, we mainly use the decomposition~\eqref{eqn: decom} done by \cite{CKL} and~\cite{GKTT} for the diffuse boundary condition. In order to apply Lemma \ref{Lemma: Step3}, here we consider the space $\mathcal{V}_i^{\frac{1-\eta}{2(1+\eta)}\delta}$ and ensure $\eta$ satisfy the condition~\eqref{eqn: eta condition}. In this step we mainly focus on constructing the $\eta$, which is defined in~\eqref{eqn: eta}.

First we consider $\eta_{i,\parallel}$, which is defined in~\eqref{eqn: eta i para}. In regard to~\eqref{eqn: abe} and~\eqref{eqn: B i para def}, we take $t'=t'(\xi,k,T_M)$( consistent with~\eqref{eqn: t'} ) to be small enough and set $t\leq t'$ to obtain
\begin{equation}\label{eqn: B i para}
B_{i,\parallel}\geq \frac{1}{2T_{k-1,i}}-\mathcal{C}^{k-i}t\geq \frac{1}{2\frac{2\xi}{\xi+1}T_M}-\mathcal{C}^{k}t\geq \frac{1}{4T_M}.
\end{equation}
By~\eqref{eqn: formula of Tp}, $T_{k-1,i}\to T_M$ as $k-i\to \infty$. For any $\e_1>0$, there exists $k_1$ s.t when
\begin{equation}\label{eqn: e1}
k\geq k_1,\quad i\leq k/2, \text{ we have }T_{k-1,i}\leq (1+\e_1)T_M.
\end{equation}
Moreover, by~\eqref{eqn: assume T}, there exists $\e_2$ s.t
\begin{equation}\label{eqn: e2 def}
\frac{\min\{T_w(x)\}}{T_M}>\frac{1-r_\parallel}{2-r_\parallel}(1+\e_2).
\end{equation}
Then we have
\begin{equation}\label{eqn: e2 dependence}
\e_2=\e_2(\min\{T_w(x)\},T_M,r_\parallel,r_\perp).
\end{equation}

Thus we can bound $T_w(x_i)$ in the $\eta_{i,\parallel}$( defined in~\eqref{eqn: eta i para}) below as
\begin{equation}\label{eqn: bound below}
T_w(x_i)=T_{k-1,i}\frac{T_w(x_i)}{T_{k-1,i}}\geq T_{k-1,i}\frac{T_w(x_i)}{T_M}\frac{1}{1+\e_1}> \frac{1-r_\parallel}{2-r_\parallel}T_{k-1,i}\frac{1+\e_2}{1+\e_1}.
\end{equation}
Thus we obtain
\begin{equation}\label{eqn: eta i para bounded}
\eta_{i,\parallel}<\frac{1+\frac{\mathcal{C}^{k-i}t}{B_{i,\parallel}}}{(1-r_\parallel)^2+ \frac{1-r_\parallel}{2-r_\parallel}\frac{1+\e_2}{1+\e_1}r_\parallel(2-r_\parallel)}(1-r_\parallel)= \frac{1+\frac{\mathcal{C}^{k-i}t}{B_{i,\parallel}}}{1-r_\parallel+r_\parallel\frac{1+\e_2}{1+\e_1}}                                                                 .
\end{equation}
By~\eqref{eqn: e1}, we take
\begin{equation}\label{eqn: k_1 dependence}
k=k_1=k_1(\e_2,T_M,r_{\min})
\end{equation}
to be large enough such that $\e_1<\e_2/4$. By~\eqref{eqn: B i para} and~\eqref{eqn: eta i para bounded}, we derive that when $k=k_1$,
\begin{equation}\label{eqn: sup  less 1}
\sup_{i\leq k/2}\eta_{i,\parallel}\leq \frac{1+4T_M\mathcal{C}^{k}t}{1-r_\parallel+r_\parallel\frac{1+\e_2}{1+\e_2/4}}<\eta_\parallel<1.
\end{equation}
Here we define
\begin{equation}\label{eqn: eta_paral}
\eta_\parallel:=\frac{1}{1-r_\parallel+r_\parallel\frac{1+\e_2}{1+\e_2/2}}<1
\end{equation}
and we take $t'=t'(k,T_M,\e_2,\mathcal{C},r_\parallel)$ to be small enough and $t\leq t'$ such that $4T_M\mathcal{C}^{k}t\ll 1$ to ensure the second inequality in~\eqref{eqn: sup  less 1}.
Combining~\eqref{eqn: e2 dependence} and~\eqref{eqn: k_1 dependence}, we conclude the $t'$ we choose only depends on the parameter in~\eqref{eqn: t'}.

Then we consider $\eta_{i,\perp}$, which is defined in~\eqref{eqn: eta i perp}. In regard to~\eqref{eqn: abe perp} and~\eqref{eqn: B i perp}, by~\eqref{eqn: B i para} we have $B_{i,\perp}\geq \frac{1}{4T_M}.$ By $\frac{\min\{T_w(x)\}}{T_M}>\frac{\sqrt{1-r_\perp}-(1-r_\perp)}{r_\perp}$ in~\eqref{eqn: assume T} we can use the same computation as~\eqref{eqn: bound below} to obtain
\[T_w(x_i)>  \frac{\sqrt{1-r_\perp}-(1-r_\perp)}{r_\perp}  T_{k-1,i}\frac{1+\e_2}{1+\e_1},\]
with $\e_1<\e_2/4$. Thus we obtain
\[\eta_{i,\perp}<\eta_\perp <1           ,                                                     \]
where we define
\begin{equation}\label{eqn: eta perp}
\eta_{\perp}:=\frac{1}{\sqrt{1-r_\perp}+(1-\sqrt{1-r_\perp})\frac{1+\e_2}{1+\e_2/2}}<1,
\end{equation}
with $t'=t'(k,T_M,\e_2,\mathcal{C},r_\parallel)$( consistent with~\eqref{eqn: t'} ) small enough and $t\leq t'$.

Finally we define
\begin{equation}\label{eqn: eta}
  \eta:=\max\{\eta_\perp,\eta_\parallel\}<1.
\end{equation}

\textbf{Step 2}

\textbf{Claim}: We have
\begin{equation}\label{eqn:claim_delta}
  |t_{j}-t_{j+1}|\gtrsim_\Omega \Big( \frac{1-\eta}{2(1+\eta)}\delta\Big)^3,\text{ for }v_j\in \mathcal{V}_j^{\frac{1-\eta}{2(1+\eta)}\delta},0\leq t_j.
\end{equation}
\begin{proof}

For $t_j\leq 1$,
\[|\int_{t_j}^{t_{j+1}}v_j ds|^2=|x_{j+1}-x_j|^2\gtrsim |(x_{j+1}-x_j)\cdot n(x_j)|\]
\[=|\int_{t_j}^{t_{j+1}}v_j\cdot n(x_j)ds|=|v_j\cdot n(x_j)||t_j-t_{j+1}|.\]
Here we use the fact that if $x,y\in \partial \Omega$ and $\partial \Omega$ is $C^2$ and $\Omega$ is bounded then $|x-y|^2\gtrsim_\Omega |(x-y)\cdot n(x)|$( see the proof in~\cite{EGKM} ). Thus
\begin{equation}\label{}
  |v_j\cdot n(x_j)|\lesssim \frac{1}{|t_j-t_{j+1}|}|\int_{t_j}^{t_{j+1}} v_jds|^2
  \lesssim |t_j-t_{j+1}||v_j|^2.
\end{equation}
Since $v_j\in \mathcal{V}^{\frac{1-\eta}{2(1+\eta)}\delta}_j$, $t_j\leq 0$, let $0\leq t\leq t'$, we have
\begin{equation}\label{}
  |v_j\cdot n(x_j)|\lesssim |t_j-t_{j+1}|\Big(\frac{1-\eta}{2(1+\eta)}\delta \Big)^{-2}.
\end{equation}
Then we prove~\eqref{eqn:claim_delta}.
\end{proof}

In consequence, when $t_k> 0$, by~\eqref{eqn:claim_delta} and $t\ll 1$, there can be at most $\{[C_{\Omega}(\frac{2(1+\eta)}{(1-\eta)\delta})^3]+1\}$ numbers of $v_j\in \mathcal{V}_j^{\frac{1-\eta}{2(1+\eta)}\delta}$.
Equivalently there are at least $k-2-[C_{\Omega}(\frac{2(1+\eta)}{(1-\eta)\delta})^3]+1$ numbers of $v_j\in \mathcal{V}_j\backslash \mathcal{V}_j^{\frac{1-\eta}{2(1+\eta)}\delta}$.

\textbf{Step 3}

In this step we combine Step 1 and Step 2 and focus on the integration over $\prod_{j=1}^{k-1} \mathcal{V}_j$.

By~\eqref{eqn:claim_delta} in Step 2, we define
\begin{equation}\label{eqn: N}
N:=\Big[C_{\Omega}\big(\frac{2(1+\eta)}{\delta(1-\eta)}\big)^3\Big]+1.
\end{equation}
For the sequence $\{v_1,v_2,\cdots,v_{k-1}\}$, suppose there are $p$ number of $v_j\in \mathcal{V}_j^{\frac{1-\eta}{2(1+\eta)}\delta}$ with $p\leq N$, we conclude there are at most $\left(
                                                                                  \begin{array}{c}
                                                                                    k-1 \\
                                                                                    p \\
                                                                                  \end{array}
                                                                                \right)
$ number of these sequences. Below we only consider a single sequence of them.

In order to get~\eqref{eqn: eta_paral},\eqref{eqn: eta perp}$<1$, we need to ensure the condition~\eqref{eqn: e1}. Thus we take $k=k_1(T_M,\xi,r_\perp,r_\parallel)$ and only use the decomposition $\mathcal{V}_j=\Big(\mathcal{V}_j\backslash \mathcal{V}_j^{\frac{1-\eta}{2(1+\eta)}\delta} \Big)   \cup \mathcal{V}_j^{\frac{1-\eta}{2(1+\eta)}\delta}$ for $\prod_{j=1}^{k/2} \mathcal{V}_j$. Then we only consider the half sequence $\{v_1,v_2,\cdots,v_{k/2}\}$. We derive that when $t_k>0$, there are at most $N$ number of $v_j\in \mathcal{V}_j^{\frac{1-\eta}{2(1+\eta)}\delta}$ and at least $k/2-1-N$ number of $v_j\in \mathcal{V}_j\backslash \mathcal{V}_j^{\frac{1-\eta}{2(1+\eta)}\delta}$ in $\prod_{j=1}^{k/2}\mathcal{V}_j$.

In this single half sequence $\{v_1,\cdots, v_{k/2}\}$, in order to apply Lemma \ref{Lemma: Step3}, we only want to consider the subsequence~\eqref{eqn: sequence} with $l+1<l+L\leq k/2$ and $L\geq 100\frac{1+\eta}{1-\eta}$. Thus we need to ignore those subsequence with $L<100\frac{1+\eta}{1-\eta}$. By~\eqref{eqn: sequence}, we conclude that at the end of this subsequence, it is adjacent to a $v_l\in \mathcal{V}_{l}^{\frac{1-\eta}{2(1+\eta)}\delta}$. By~\eqref{eqn: N}, we conclude
\begin{equation}\label{Conclude}
      \textit{There are at most $N$ number of subsequences~\eqref{eqn: sequence} with $L\leq 100\frac{1+\eta}{1-\eta}$}.
\end{equation}
We ignore these subsequences. Then we define the parameters for the remaining subsequence( with $L\geq 100\frac{1+\eta}{1-\eta}$ ) as:
\[M_1:= \text{the number of $v_j\in \mathcal{V}_j\backslash \mathcal{V}_j^{\frac{1-\eta}{2(1+\eta)}\delta}$ in the first subsequence starting from $v_1$},\]
\[n:= \text{the number of these subsequences}.\]
Similarly we can define $M_2,M_3,\cdots, M_n$ as the number in the second, third, $\cdots$, $n$-th subsequence. Recall that we only consider $\prod_{j=1}^{k/2} \mathcal{V}_j$, thus we have
\begin{equation}\label{eqn: Mi number}
100\frac{1+\eta}{1-\eta}\leq M_i\leq k/2,   \text{ for } 1\leq i\leq n.
\end{equation}
By~\eqref{Conclude}, we obtain
\begin{equation}\label{eqn: sum of M_i}
  k/2 \geq M_1+\cdots M_n\geq k/2-1-100\frac{1+\eta}{1-\eta}N>\frac{k}{2}-101\frac{1+\eta}{1-\eta}N.
\end{equation}
Take $M_i$ with $1\leq i\leq n$ as an example. Suppose this subsequence starts from $v_{l_i+1}$ to $v_{l_i+M_i}$, by~\eqref{eqn: Step3} in Lemma \ref{Lemma: Step3} with replacing $l$ by $l_i$ and $L$ by $M_i$, we obtain
\begin{equation}\label{eqn: M_i conclusion}
\int_{\prod_{j={l_i}}^{k-1} \mathcal{V}_j}\mathbf{1}_{\{t_k>0\}}\mathbf{1}_{\{v_{l_i+j}\in \mathcal{V}_{l_i+j}\backslash \mathcal{V}_{l_i+j}^{\frac{1-\eta}{2(1+\eta)}\delta} \text{ for } 1\leq j\leq M_i\}}d\Phi_{l_i,m}^{k,k-1}(t_k) \leq  (3\delta)^{M_i/2}   (C_{T_M,\xi})^{2(k-l)}\mathcal{A}_{k-1,l_i}.
\end{equation}

Since~\eqref{eqn: M_i conclusion} holds for all $1\leq i\leq n$, by Lemma \ref{Lemma: accumulate} we can draw the conclusion for the Step 3 as follows. For a single sequence $\{v_1,v_2,\cdots,v_{k-1}\}$, when there are $p$ number $v_j\in \mathcal{V}_{j}^{\frac{1-\eta}{2(1+\eta)}\delta}$, we have
\[\int_{\prod_{j=1}^{k-1} \mathcal{V}_j}   \mathbf{1}_{\{\text{$p$ number $v_j\in \mathcal{V}_{j}^{\frac{1-\eta}{2(1+\eta)}\delta}$ for a single sequence}\}}    \mathbf{1}_{\{t_k>0\}} d\Sigma_{k-1,m}^{k}(t_k)\]
\begin{equation}\label{eqn: Step 3 conclusion}
\leq (3\delta)^{(M_1+\cdots+M_n)/2}  (C_{T_M,\xi})^{2k} \mathcal{A}_{k-1,1}.
\end{equation}

\textbf{Step 4}

Now we are ready to prove the lemma. By~\eqref{eqn: N}, we have
\[\int_{\prod_{j=1}^{k-1}\mathcal{V}_j} \mathbf{1}_{\{t_k>0\}} d\Sigma_{k-1,m}^k(t_k)\]
\begin{equation}\label{eqn: proof step3}
\leq \sum_{p=1}^{N}\int_{\{\text{Exactly $p$ number of $v_j\in \mathcal{V}_j^{\frac{1-\eta}{2(1+\eta)}\delta}$ }\}} \mathbf{1}_{\{t_k>0\}} d\Sigma_{k-1,m}^k(t_k).
\end{equation}
Since~\eqref{eqn: Step 3 conclusion} holds for a single sequence, we derive
\[\eqref{eqn: proof step3}\leq (C_{T_M,\xi})^{2k}\sum_{p=1}^N\left(
    \begin{array}{c}
      k-1 \\
     p \\
    \end{array}
  \right)(3\delta)^{(M_1+M_2+\cdots M_n)/2} \mathcal{A}_{k-1,1}
\]
\begin{equation}\label{eqn: tk coe}
\leq (C_{T_M,\xi})^{2k}N(k-1)^N(3\delta)^{k/4-101\frac{1+\eta}{1-\eta}N}\mathcal{A}_{k-1,1},
\end{equation}
where we use~\eqref{eqn: sum of M_i} in the second line.

Take $k=N^3$, the coefficient in~\eqref{eqn: tk coe} is bounded by
\begin{equation}\label{eqn: Finally}
(C_{T_M,\xi})^{2N^3}N^{3N+1} (3\delta)^{N^3/4-101\frac{1+\eta}{1-\eta}N}\leq (C_{T_M,\xi})^{2N^3} N^{4N}(3\delta)^{N^3/5},
\end{equation}
where we choose $N=N(\eta)$ large such that $N^3/4-101\frac{1+\eta}{1-\eta}N\geq N^3/5$.

Using~\eqref{eqn: N}, we derive
\[3\delta=C(\Omega,\eta)N^{-1/3}.\]
Finally we bound~\eqref{eqn: Finally} by
\[(C_{T_M,\xi})^{2N^3}N^{4N}(C(\Omega,\eta)N^{-1/3})^{N^3/5}\leq e^{2N^3\log(C_{T_M,\xi})} e^{4N\log N}e^{(N^3/5)\log(C(\Omega,\eta)N^{-1/3})}\]
\[=e^{4N \log N}e^{(N^3/5)(log(C(\Omega,\eta))-\frac{1}{3}\log N)}e^{2N^3 \log(C_{T_M,\xi})}= e^{4N\log N-\frac{N^3}{15}(\log N-3\log C_{\Omega,\eta}-30\log C_{T_M,\xi})}\]
\[\leq e^{4N\log N-\frac{N^3}{30}\log N}\leq e^{-\frac{N^3}{50}\log N}=e^{-\frac{k}{150}\log k}\leq (\frac{1}{2})^k,\]
where we choose $\delta$ to be small enough in the second line such that $N=N(\Omega,\eta,C_{T_M,\xi})$ is large enough to satisfy
\[\log N -3\log C(\Omega,\eta)-30 \log C_{T_M,\xi}\geq \frac{\log N}{2},\]
\[4N\log N-\frac{N^3}{30}\log N\leq -\frac{N^3}{50}\log N.\]
And thus we choose $k=N^3=k_2=k_2(\Omega,\eta,C_{T_M,\xi})$ and we also require $\log k>150$ in the last step. Then we get~\eqref{eqn: 1/2 decay}.

Therefore, by the condition~\eqref{eqn: e1}, we choose $k=k_0=\max\{k_1,k_2\}$. By the definition of $\eta$~\eqref{eqn: eta} with~\eqref{eqn: eta_paral} and~\eqref{eqn: eta perp}, we obtain $\eta=\eta(T_M,\mathcal{C},r_\perp,r_\parallel,\e_2)$. Thus by~\eqref{eqn: e2 dependence} and~\eqref{eqn: k_1 dependence}, we conclude the $k_0$ we choose here does not depend on $t$ and only depends on the parameter in~\eqref{eqn: k_0 dependence}. We derive the lemma.

\end{proof}

\begin{proof}[\textbf{Proof of Proposition \ref{proposition: boundedness}}]
First we take
\begin{equation}\label{eqn: first condition for tinf}
t_{\infty}\leq t'.
\end{equation}
with $t'$ defined in~\eqref{eqn: t'}. Then we let $k=k_0$ with $k_0$ defined in~\eqref{eqn: k_0 dependence} so that we can apply Lemma \ref{lemma: t^k} and Lemma \ref{lemma: boundedness}. Define the constant in~\eqref{eqn: fm is bounded} as
\begin{equation}\label{eqn: Cinfty}
  C_\infty=3(C_{T_M,\xi})^{k_0}.
\end{equation}

We mainly use the formula given in Lemma \ref{lemma: the tracjectory formula for f^(m+1)}. We consider two cases.
\begin{description}
\item[Case1] $t_1\leq 0$,
\end{description}
By~\eqref{eqn: Duhamal principle for case1} and using the definition of $\Gamma^m_{\text{gain}}(s)$ in~\eqref{eqn: gamma^m} we have
\begin{equation}\label{eqn: first term}
|h^{m+1}(t,x,v)|\leq | h_0(X^1(0;t,x,v),v)|
\end{equation}
\begin{equation}\label{eqn: second term}
   +\int_0^t e^{|v|^2(\theta-t)}  \int_{\mathbb{R}^3\times \mathbb{S}^2}B(v-u,w)\sqrt{\mu(u)} \Big|\frac{h^{m}(s,X^1(s),u')}{e^{|u'|^2(\theta-s)}}\Big| \Big|\frac{h^{m}(s,X^1(s),v')}{e^{|v'|^2(\theta-s)}}\Big|      d\omega duds,
\end{equation}
where $u'=u'(u,v)$ and $v'=v'(u,v)$ are defined by~\eqref{eqn: u' v'}. Then we have
\[\eqref{eqn: second term}\leq(\sup_{0\leq s\leq t} \Vert h^m(s)\Vert_{L^{\infty}})^2 \times\int_0^t  \int_{\mathbb{R}^3\times \mathbb{S}^2}e^{|v|^2(\theta-t)}  B\big(v-u,w\big)\]
\[ \sqrt{\mu(u)}  e^{(|u|^2+|v|^2)(s-\theta)}   d\omega duds\]
\[\lesssim (\sup_{0\leq s\leq t} \Vert h^m(s)\Vert_{L^{\infty}})^2 \int_0^t \int_{\mathbb{R}^3}e^{|v|^2(s-t)} |v-u|^\mathcal{K}\sqrt{\mu} e^{|u|^2(s-\theta)}   du ds\]
\[\lesssim_{C_\infty} \Vert h_0\Vert_{L^\infty}^2\int_0^t e^{|v|^2(s-t)} \langle v\rangle^{\mathcal{K}+3} ds\]
\[\leq\Vert h_0\Vert_{L^\infty}^2\int_0^t e^{|v|^2(s-t)} \langle v\rangle^{4} \{\mathbf{1}_{|v|>N}+\mathbf{1}_{|v|\leq N}\}ds\]
\[\lesssim_{\Vert h_0\Vert_\infty}\big(\frac{1}{N^2}+Nt\big),\]
where $-3<\mathcal{K}\leq 1$. Therefore, we obtain
\begin{equation}\label{eqn: Gamma bounded by}
~\eqref{eqn: second term}\leq C(C_\infty,\Vert h_0\Vert_\infty)(\frac{1}{N^2}+Nt)\leq \frac{1}{k_0}\Vert h_0\Vert_\infty,
\end{equation}
where we choose
\begin{equation}\label{eqn: second condition for tinf}
N=N(C_\infty,\Vert h_0\Vert_\infty,k_0)\gg 1,\quad
t_\infty=t_\infty(N,C_\infty,\Vert h_0\Vert_\infty,k_0)\ll 1,
\end{equation}
with $t\leq t_\infty$ to obtain the last inequality in~\eqref{eqn: Gamma bounded by}.

Finally collecting~\eqref{eqn: first term} and~\eqref{eqn: second term} we obtain
\begin{equation}\label{eqn: hm+1 bounded case 1}
 \Vert h^{m+1}(t,x,v)\mathbf{1}_{\{t_1\leq 0\}}\Vert_\infty \leq   2\Vert h_0\Vert_\infty\leq C_\infty\Vert h_0\Vert_\infty,
\end{equation}
where $C_\infty$ is defined in~\eqref{eqn: Cinfty}.

\begin{description}
\item[Case2] $t_1\geq 0$,
\end{description}
We consider~\eqref{eqn: Duhamel principle for case 2} in Lemma \ref{lemma: the tracjectory formula for f^(m+1)}. First we focus on the first line. By~\eqref{eqn: Gamma bounded by} we obtain
\begin{equation}\label{eqn: first line bounded}
\int_{t_1}^t e^{|v|^2(\theta-t)} \Gamma_{\text{gain}}^m(s)ds \leq \frac{1}{k_0}\Vert h_0\Vert_\infty.
\end{equation}
Then we focus on the second line of~\eqref{eqn: Duhamel principle for case 2}. Using $\theta=\frac{1}{4T_M\xi}$ we bound the second line of~\eqref{eqn: Duhamel principle for case 2} by
\begin{equation}\label{eqn: extra term to cancel}
\exp\bigg(\big[\frac{1}{2T_M\frac{2\xi}{\xi+1}}-\frac{1}{2T_w(x_1)}\big]|v|^2\bigg)\int_{\prod_{j=1}^{k_0-1}\mathcal{V}_j}H.
\end{equation}
Now we focus on $\int_{\prod_{j=1}^{k_0-1}\mathcal{V}_j}H$. We compute $H$ term by term with the formula given in~\eqref{eqn: formula for H}. First we compute the first line of~\eqref{eqn: formula for H}. By Lemma~\ref{lemma: boundedness} with $p=1$, for every $1\leq l\leq k_0-1$, we have
\[\int_{\prod_{j=1}^{k_0-1}\mathcal{V}_j}  \mathbf{1}_{\{t_{l+1}\leq 0<t_l\}} |h_0\big(X^{m-l}(0),V^{m-l}(0)\big)|  d\Sigma_{l,m}^{k_0}(0)\leq \Vert h_0\Vert_\infty  \int_{\prod_{j=1}^{k_0-1}\mathcal{V}_j}  \mathbf{1}_{\{t_{l+1}\leq 0<t_l\}}   d\Sigma_{l,m}^{k_0}(0)\]
\begin{equation}\label{eqn: l term}
     \leq (C_{T_M,\xi})^l\Vert h_0\Vert_{\infty}\exp\bigg(\frac{(T_{l,1}-T_w(x_1))(1-r_{min})}{2T_w(x_1)[T_{l,1}(1-r_{min})+r_{min} T_w(x_1)]}|v|^2+\mathcal{C}^{l}t|v|^2\bigg).
\end{equation}
In regard to~\eqref{eqn: extra term to cancel} we have
\[\exp\bigg(\big[\frac{1}{2T_M\frac{2\xi}{\xi+1}}-\frac{1}{2T_w(x_1)}\big]|v|^2\bigg)\times ~\eqref{eqn: l term}=\]
\[(C_{T_M,\xi})^l\Vert h_0\Vert_{\infty}\exp\bigg(\Big[\frac{-1}{2\big(T_w(x_1)r_{min}+T_{l,1}(1-r_{min})\big)}+\frac{1}{2T_M\frac{2\xi}{\xi+1}}\Big]|v|^2+(\mathcal{C})^{l}t|v|^2\bigg).\]
Using the definition~\eqref{eqn: definition of T_p} we have $T_w(x_1)<\frac{2\xi}{\xi+1}T_M$ and $T_{l,1}<\frac{2\xi}{\xi+1}T_M$. Then we take
\begin{equation}\label{eqn: third condition for tinfty}
t_\infty=t_\infty(T_M,k_0,\xi,\mathcal{C})
\end{equation}
to be small enough and $t\leq t_\infty$ so that the coefficient for $|v|^2$ is
\[\frac{-1}{2\big(T_w(x_1)r_{min}+T_{l,1}(1-r_{min})\big)}+\frac{1}{2T_M\frac{2\xi}{\xi+1}}+(\mathcal{C})^{l}t\]
\begin{equation}\label{eqn: less than 0}
\leq \frac{-1}{2\big(T_Mr_{min}+T_{l,1}(1-r_{min})\big)}+\frac{1}{2T_M\frac{2\xi}{\xi+1}}+(\mathcal{C})^{k_0}t\leq 0.
\end{equation}
Since~\eqref{eqn: l term} holds for all $1\leq l\leq k_0-1$, by~\eqref{eqn: less than 0} the contribution of the first line of~\eqref{eqn: formula for H} in ~\eqref{eqn: extra term to cancel} is bounded by
\begin{equation}\label{eqn: first term bounded}
(C_{T_M,\xi})^{k_0}\Vert h_0\Vert_{\infty}.
\end{equation}

Then we compute the second line of~\eqref{eqn: formula for H}. For each $1\leq l\leq k_0-1$ such that $\max\{0,t_{l+1}\}\leq s\leq t_l$, by~\eqref{eqn:trajectory measure}, we have
\[d\Sigma_{l,m}^{k_0}(s)=e^{-|v_l|^2(t_l-s)} d\Sigma_{l,m}^{k_0}(t_l).\]
Therefore, we derive
\[\int_{\max\{0,t_{l+1}\}}^{t_l}\int_{\prod_{j=1}^{k_0-1}\mathcal{V}_{j}} e^{|v_l|^2(\theta-s)} |\Gamma_{\text{gain}}^{m-l}(s)|d\Sigma_{l,m}^{k_0}(s)ds\]
\[\leq \int_{\prod_{j=1}^{k_0-1}\mathcal{V}_{j}}\int_{\max\{0,t_{l+1}\}}^{t_l}e^{|v_l|^2(\theta-t_l)} |\Gamma_{\text{gain}}^{m-l}(s)| ds d\Sigma_{l,m}^{k_0}(t_l)\]
\[\leq \frac{1}{k_0}\Vert h_0\Vert_\infty\int_{\prod_{j=1}^{k_0-1}\mathcal{V}_j} \Sigma_{l,m}^{k_0}(t_l)\]
\begin{equation}\label{eqn: last line second}
\leq \frac{1}{k_0}\Vert h_0\Vert_\infty (C_{T_M,\xi})^l \exp\bigg(\frac{(T_{l,1}-T_w(x_1))(1-r_{min})}{2T_w(x_1)[T_{l,1}(1-r_{min})+r_{min} T_w(x_1)]}|v|^2+(\mathcal{C})^{l}t|v|^2\bigg),
\end{equation}
where we apply~\eqref{eqn: Gamma bounded by} in the third line and we apply Lemma \ref{lemma: boundedness} in the last line.

In regard to~\eqref{eqn: extra term to cancel}, by~\eqref{eqn: less than 0} we obtain
\[\exp\bigg(\big[\frac{1}{2T_M\frac{2\xi}{\xi+1}}-\frac{1}{2T_w(x_1)}\big]|v|^2\bigg)\times ~\eqref{eqn: last line second}\leq \frac{1}{k_0}(C_{T_M,\xi})^l \Vert h_0\Vert_\infty.\]

Since~\eqref{eqn: last line second} holds for all $1\leq l\leq k_0-1$, the contribution of the second line of~\eqref{eqn: formula for H} in~\eqref{eqn: extra term to cancel} is bounded by
\begin{equation}\label{eqn: second term bounded}
  \frac{k_0-1}{k_0}(C_{T_M,\xi})^{k_0}\Vert h_0\Vert_\infty.
\end{equation}

Last we compute the third term of~\eqref{eqn: formula for H}. By Lemma \ref{lemma: t^k} and the assumption~\eqref{eqn: fm is bounded} we obtain
\[\int_{\prod_{j=1}^{k_0-1}\mathcal{V}_j}  \mathbf{1}_{\{0<t_{k_0}\}} |h^{m-k_0+2}\big(t_{k_0},x_{k_0},V^{m-k_0+1}(t_{k_0})\big)|  d\Sigma_{k_0-1,m}^{k_0}(t_{k_0})\leq \Vert h^{m-k_0+2}\Vert_\infty  \int_{\prod_{j=1}^{k_0-1}\mathcal{V}_j}  \mathbf{1}_{\{0<t_{k_0}\}}   d\Sigma_{k_0-1,m}^{k_0}(t_{k_0})\]
\begin{equation}\label{eqn: third term bounded in H}
       \leq      3(C_{T_M,\xi})^{k_0} (\frac{1}{2})^{k_0}\Vert h_0\Vert_\infty\exp\bigg(\frac{(T_{l,1}-T_w(x_1))(1-r_{min})}{2T_w(x_1)[T_{l,1}(1-r_{min})+r_{min} T_w(x_1)]}|v|^2+(\mathcal{C})^{l}t|v|^2\bigg).
\end{equation}

In regard to~\eqref{eqn: extra term to cancel}, by~\eqref{eqn: less than 0} we have
\[\exp\bigg(\big[\frac{1}{2T_M\frac{2\xi}{\xi+1}}-\frac{1}{2T_w(x_1)}\big]|v|^2\bigg)\times ~\eqref{eqn: third term bounded in H}\leq (C_{T_M,\xi})^{k_0} \Vert h_0\Vert_\infty.\]
Thus the contribution of the third line of~\eqref{eqn: formula for H} in~\eqref{eqn: extra term to cancel} is bounded by
\begin{equation}\label{eqn: third term bounded}
(C_{T_M,\xi})^{k_0} \Vert h_0(x,v)\Vert_\infty.
\end{equation}

Collecting~\eqref{eqn: first term bounded}~\eqref{eqn: second term bounded}~\eqref{eqn: third term bounded} we conclude that the second line of~\eqref{eqn: Duhamel principle for case 2} is bounded by
\begin{equation}\label{eqn: second line bounded}
(C_{T_M,\xi})^{k_0}\times (2+\frac{k_0-1}{k_0}) \Vert h_0\Vert_\infty.
\end{equation}
Adding~\eqref{eqn: second line bounded} to~\eqref{eqn: first line bounded} we use~\eqref{eqn: Duhamel principle for case 2} to derive
\begin{equation}\label{eqn: hm+1 bounded case 2}
\Vert h^{m+1}(t,x,v)\mathbf{1}_{\{t_{1}\geq 0\}}\Vert_\infty\leq 3(C_{T_M,\xi})^{k_0} \Vert h_0\Vert_\infty=C_\infty \Vert h_0\Vert_\infty.
\end{equation}

Combining~\eqref{eqn: hm+1 bounded case 1} and~\eqref{eqn: hm+1 bounded case 2} we derive~\eqref{eqn: L_infty bound for f^m+1}.

Last we focus the parameters for $t_\infty$ in~\eqref{eqn: t_1}. In the proof the constraints for $t_\infty$ are~\eqref{eqn: first condition for tinf},~\eqref{eqn: second condition for tinf} and~\eqref{eqn: third condition for tinfty}. We obtain
\[t_\infty=t_\infty(t',N,C_\infty,\Vert h_0\Vert_\infty,T_M,k_0,\xi,\mathcal{C})=t_\infty(k_0,\xi,T_M,\min\{T_w(x)\},\mathcal{C},r_\perp,r_\parallel,C_{T_M,\xi},\Vert h_0\Vert_\infty).\]
By the definition of $k_0$ in~\eqref{eqn: k_0 dependence}, definition of $C_{T_M,\xi}$ in~\eqref{eqn: 1 one}, definition of $\mathcal{C}$ in~\eqref{eqn: cal C}, we derive~\eqref{eqn: t_1}.

\end{proof}

Then we can conclude the well-posedness.

\begin{proof}[\textbf{Proof of Theorem~\ref{local_existence}}]
First of all we take $t<t_\infty$, where $t_\infty$ is defined in~\eqref{eqn: t_1} so that we can apply Proposition \ref{proposition: boundedness}. We have
\[\sup_{m}\Vert h^m\Vert_\infty\lesssim \Vert h(0)\Vert_\infty.\]

 \begin{itemize}
    \item Existence
  \end{itemize}
For $h^m$ given in~\eqref{eqn: hm+1}, we take the difference $h^{m+1}-h^m$ and deduce that
\[\partial_t [h^{m+1}-h^m]+v\cdot \nabla_x [h^{m+1}-h^m]+\nu^m(h^{m+1}-h^m)=e^{(\theta-t)|v|^2}\Lambda^m,\]
 \[[h^{m+1}-h^m]_-=e^{(\theta-t)|v|^2}e^{[\frac{1}{4T_M}-\frac{1}{2T_w(x)}]|v|^2}\int_{n(x)\cdot u>0} [h^{m+1}(u)-h^m(u)]e^{-[\frac{1}{4T_M}-\frac{1}{2T_w(x)}]|u|^2}e^{-(\theta-t)|u|^2}   d\sigma(u,v),\]
where
\[\Lambda^m=\Gamma_{\text{gain}}\Big(\frac{h^m-h^{m-1}}{e^{(\theta-t)|v|^2}},\frac{h^m}{e^{(\theta-t)|v|^2}}\Big)+\Gamma_{\text{gain}}\Big(\frac{h^{m-1}}{e^{(\theta-t)|v|^2}},\frac{h^m-h^{m-1}}{e^{(\theta-t)|v|^2}}\Big)+[\nu(F^{m-1})-\nu(F^m)]h^{m-1}.\]

By the same derivation as~\eqref{eqn: Duhamal principle for case1}~\eqref{eqn: Duhamel principle for case 2}, when $t_1\leq 0$, we have
\[|h^{m+1}-h^m|(t,x,v)\leq  \int_0^t e^{|v|^2(\theta-t)}\int_{\mathbb{R}^3\times \mathbb{S}^2} B(v-u,w)\sqrt{\mu}\Big[ \Big|\frac{(h^m-h^{m-1})(s,X^1(s),u')}{e^{|u'|(\theta-s)}}\Big| \Big|\frac{h^m(s,X^1(s),v')}{e^{|v'|(\theta-s)}} \Big|     \]
\[+\Big|\frac{h^m(s,X^1(s),u')}{e^{|u'|(\theta-s)}}\Big|\Big|\frac{(h^m-h^{m-1})(s,X^1(s),v')}{e^{|v'|(\theta-s)}}\Big| +\Big|\frac{(h^m-h^{m-1})(s,X^1(s),u)}{e^{|u|^2(\theta-s)}}\Big|\Big|\frac{h^{m-1}(s,X^1(s),v)}{e^{|v|^2(\theta-s)}}\Big|\bigg]d\omega duds,\]
where we use $h^{m+1}(0)=h^m(0)$.

Then we follow the computation for~\eqref{eqn: second term} to obtain
\[|h^{m+1}-h^m|(t,x,v)\lesssim (\Vert h^m-h^{m-1}\Vert_\infty)\Vert h^m\Vert_\infty\times \int_0^t \int_{\mathbb{R}^3\times\mathbb{S}^2}e^{|v|^2(\theta-t)}B(v-u,\omega)\]
 \[\sqrt{\mu(u)}e^{(|u|^2+|v|^2)(s-\theta)}d\omega duds\]
 \begin{equation}\label{eqn: o(1)}
 \lesssim \Vert h^{m}-h^{m-1}\Vert_\infty \Vert h^m\Vert_\infty (\frac{1}{N^2}+Nt)\lesssim o(1)\Vert h^m-h^{m-1}\Vert_\infty ,
 \end{equation}
where we take $N=N(\Vert h^m\Vert_\infty)$ to be large and $t<t_\infty=t_\infty(N)$ to be small as in~\eqref{eqn: second condition for tinf}.

When $t_1>0$, by the same derivation as~\eqref{eqn: Duhamel principle for case 2}, we have
\[|h^{m+1}-h^m|(t,x,v)\leq \int_{t_1}^t e^{|v|^2(\theta-t)}\Lambda^m ds+e^{|v|^2(\theta-t_1)}e^{[\frac{1}{4T_M}-\frac{1}{2T_w(x_1)}]|v|^2}\int_{\prod_{j=1}^{k-1}\mathcal{V}_j}H_d,\]
where $H_d$ is bounded by
\begin{equation}\label{eqn: Hd}
\begin{split}
    & \sum_{l=1}^{k-1}\int_{\max\{0,t_{l+1}\}}^{t_l}e^{|v_l|^2(\theta-s)}|\Lambda^m(s)|d\Sigma_{l,m}^k(s)ds \\
     & +\mathbf{1}_{\{t_k>0\}}|h^{m-k+2}-h^{m-k+1}|(t_k,x_k,v_{k-1})d\Sigma_{k-1,m}^k(t_k).
\end{split}
\end{equation}
By~\eqref{eqn: last line second} and~\eqref{eqn: o(1)}, the first line of~\eqref{eqn: Hd} is bounded by
\[k_0O(t)\sup_{\ell\leq m}\Vert h^{\ell}-h^{\ell-1}\Vert_\infty=o(1)\sup_{\ell\leq m}\Vert h^{\ell}-h^{\ell-1}\Vert_\infty,\]
where we take $t<t_\infty=t_\infty(k_0)$ to be small.

Then we apply~\eqref{eqn: third term bounded in H}~\eqref{eqn: third term bounded} with replacing $\Vert h^{m-k_0+2}\Vert_\infty$ by $\Vert h^{m-k_0+2}-h^{m-k_0+1}\Vert_\infty$. Thus we obtain the second line of~\eqref{eqn: Hd} is bounded by
\[\Big(\frac{1}{2}\Big)^{k_0}\sup_{\ell\leq m}\Vert h^{\ell}-h^{\ell-1}\Vert_\infty.\]
Thus in the case $t_1>0$ we obtain
\begin{equation}\label{eqn: Linfty Cauchy}
\Vert h^{m+1}-h^m\Vert_\infty\leq o(1)\sup_{\ell\leq m}\Vert h^{\ell}-h^{\ell-1}\Vert_\infty.
\end{equation}

Therefore, $h^m$ is a Cauchy-sequence in $L^\infty$. The existence follows by taking the limit $m\to\infty$ and the solution $h=e^{(\theta-t)|v|^2}f$ satisfies
\begin{equation}\label{eqn: equation for h}
\partial_t h+v\cdot \nabla_x h+|v|^2 h =e^{(\theta-t)|v|^2} \Gamma\left(\frac{h}{e^{(\theta-t)|v|^2}},\frac{h}{e^{(\theta-t)|v|^2}}\right).
\end{equation}
Moreover, we have
\begin{equation}\label{eqn: h linfty}
\Vert h\Vert_\infty\leq \sup_m\Vert h^m\Vert_\infty \lesssim \Vert h(0)\Vert_\infty.
\end{equation}
This concludes the existence of $f$ and~\eqref{infty_local_bound}.

 \begin{itemize}
   \item Stability
 \end{itemize}
Suppose there are two solutions $h_1$ and $h_2$ satisfy~\eqref{eqn: equation for h}. Also suppose there initial condition satisfy
\[\Vert h_1(0)\Vert_\infty,\Vert h_2\Vert_\infty<\infty.\]
When $t_1\leq 0$, by the same derivation as~\eqref{eqn: Gamma bounded by} and~\eqref{eqn: o(1)} we have
\[|h_1-h_2|(t,x,v)\lesssim |h_1-h_2|(0)+(\Vert h_1\Vert_\infty+\Vert h_2\Vert_\infty)\int_0^t \Vert h_1-h_2\Vert_\infty e^{|v|^2(s-t)}\langle v\rangle^{4}\{\mathbf{1}_{|v|>N}+\mathbf{1}_{|v|\leq N}\}ds\]
\[\lesssim \Vert (h_1-h_2)(0)\Vert_\infty+(\Vert h_2\Vert_\infty+\Vert h_1\Vert_\infty)\Big[ O(\frac{1}{N^2}) \Vert h_1-h_2\Vert_\infty+\int_0^t N\Vert h_1-h_2\Vert_\infty ds\Big].\]
By taking $N=N(\Vert h_1\Vert_\infty,\Vert h_2\Vert_\infty)$ to be large as in~\eqref{eqn: second condition for tinf} so that $(\Vert h_2\Vert_\infty+\Vert h_1\Vert_\infty) O(\frac{1}{N^2})\ll 1$, we derive the $L^\infty$ stability by the Gronwall's inequality.

When $t_1>0$, the argument is exactly the same as the existence part and we conclude the $L^\infty$ stability for all cases. The uniqueness follows immediately by setting $h_1(0)=h_2(0)$.

The positivity follows from the the property that iteration equation~\eqref{eqn: Fm+1} is positive preserving and~\eqref{eqn: Linfty Cauchy}.

\end{proof}

\section{Steady problem with C-L boundary condition}
This section is devoted to the steady solution to the Boltzmann equation with the Cercignani-Lampis boundary condition as mention in Section 1.2.
\begin{remark}
The setting of the steady solution is given in Section 1.2. We remark here that in this section we no longer use notation $\mu$. Instead we put the subscript $\mu_0,\delta_0$ only for this section in order to avoid confusion.

\end{remark}

To prove Corollary \ref{Thm: steady solution} we need the following Proposition.
\begin{proposition}[Proposition 4.1 of \cite{EGKM}]\label{proposition: Linfty of steady solution}
Define a weight function scaled with parameter $\varrho$ as
\begin{equation}
w_{\varrho}(v)= w_{\varrho, \beta, \zeta}(v) \equiv (1+\varrho^2|v|^2)^{%
\frac{\beta}{2}} e^{\zeta|v|^2}.  \label{weight}
\end{equation}
Assume
\begin{equation}\label{eqn: g,r condition}
\iint_{\Omega\times \mathbb{R}^3} g(x,v)\sqrt{\mu_0}dxdv=0,\quad \quad \int_{\gamma_-}r\sqrt{\mu_0}d\gamma=0
\end{equation}
and $\beta>4$. Then the solution $f$
to the linear Boltzmann equation
\begin{equation}\label{eqn: linearized equation}
  v\cdot \nabla_x f+Lf=g,\quad \quad f_-=P_\gamma f+r
\end{equation}
 satisfies $\Vert w_{\varrho
}f\Vert _{\infty }+|w_{\varrho }f|_{\infty }\lesssim \Vert w_{\varrho }g\Vert _{\infty }+|w_{\varrho }\langle v\rangle r|_{\infty }.$
\end{proposition}

For the purpose of applying Proposition \ref{proposition: Linfty of steady solution}, we focus on the boundary condition for the linearized equation $f_s$.
\begin{lemma}\label{Prop: perturbation}
For $F_s=\mu_0+\sqrt{\mu_0}f_s$ with $F_s$ satisfying the boundary condition~\eqref{eqn:BC},~\eqref{eqn: Formula for R}, the boundary condition for $f_s$ can be represented as
\begin{equation}\label{eqn: pertubation}
f_s|_-(x,v)=P_\gamma f_s +r
\end{equation}
such that
\begin{equation}\label{eqn: int r is 0}
  \int_{\gamma_-}r\sqrt{\mu_0}=0.
\end{equation}
Moreover,
\begin{equation}\label{eqn: bound for r}
| r|_\infty\lesssim \delta_0+ \sup_{0\leq s\leq t}\delta_0|f(s)|_\infty.
\end{equation}

\end{lemma}
Before proving this lemma we need the following lemma for the C-L boundary condition.

\begin{lemma}\label{Lemma: mu in bc}
In regard to the boundary condition~\eqref{eqn: Formula for R}, we have
\begin{equation}\label{eqn: mu in the bc}
\frac{1}{|n(x)\cdot v|}\int_{n(x)\cdot u>0}R(u\to v;x,t)\mu_0 \{n(x)\cdot u\}du=\mu_{x,r_\parallel,r_\perp},
\end{equation}
where
\begin{equation}\label{eqn: mu_xr}
\begin{split}
   & \mu_{x,r_\parallel,r_\perp}=\frac{1}{2\pi [T_0(1-r_\parallel)^2+T_w(x)r_\parallel(2-r_\parallel)]}e^{-\frac{|v_\parallel|^2}{2[T_0(1-r_\parallel)^2+T_w(x)r_\parallel(2-r_\parallel)]}}.
 \\
    & \times \frac{1}{T_0(1-r_\perp)+T_w(x)r_\perp}e^{-\frac{|v_\perp|^2}{2[T_0(1-r_\perp)+T_w(x)r_\perp]}}.
\end{split}
\end{equation}
Moreover, for any $x\in \partial \Omega$ and $r_\parallel,r_\perp$, we have
\begin{equation}\label{eqn: int is 1}
  \int_{n(x)\cdot v>0}\mu_{x,r_\parallel,r_\perp}\{n(x)\cdot v\}dv=1.
\end{equation}

\end{lemma}

\begin{proof}
Using the definition of $R(u\to v;x,t)$ in~\eqref{eqn: Formula for R} we can write the LHS of~\eqref{eqn: mu in the bc} as
\begin{equation}\label{eqn: LHS of mu in bc}
\begin{split}
   &\int_{\mathbb{R}^+}\frac{|u_\perp|}{r_\perp T_w(x)}    \exp\Big(-\frac{1}{2T_w(x)}\big[\frac{|v_\perp|^2+(1-r_\perp)|u_\perp|^2}{r_\perp} \big] \Big)\times I_0\Big(\frac{(1-r_\perp)^{1/2}v_\perp u_\perp}{T_w(x)r_\perp} \Big)\frac{1}{T_0}\exp\Big(-\frac{|u_\perp|^2}{2T_0} \Big)dv_\perp \\
    & \times\int_{\mathbb{R}^2}\frac{1}{2T_w(x)r_\parallel(2-r_\parallel)\pi}\exp\Big(-\frac{1}{2T_w(x)}\frac{|v_\parallel-(1-r_\parallel)u_\parallel|^2}{r_\parallel(2-r_\parallel)} \Big)\frac{1}{2\pi T_0}\exp\Big(-\frac{|u_\parallel|^2}{2T_0} \Big)dv_\parallel.
\end{split}
\end{equation}
First we compute the second line of~\eqref{eqn: LHS of mu in bc}, in order to apply Lemma \ref{Lemma: abc}, we set
\[a=-\frac{1}{2T_0},\quad b=\frac{(1-r_\parallel)^2}{2T_w(x)r_\parallel(2-r_\parallel)},\quad v=u_\parallel,\quad w=\frac{1}{1-r_\parallel}v_\parallel,\quad \e=0,\]
\[b-a=\frac{(1-r_\parallel)^2}{2T_w(x)r_\parallel(2-r_\parallel)}+\frac{1}{2T_0}.\]
Then the second line of~\eqref{eqn: LHS of mu in bc} equals to
\[  \frac{1}{(1-r_\parallel)^2}\frac{b}{b-a}\exp\Big(\frac{ab}{b-a}|\frac{v_\parallel}{1-r_\parallel}|^2\Big)\]
\[=\frac{1}{2\pi}\frac{1}{T_0(1-r_\parallel)^2+T_w(x)r_\parallel(2-r_\parallel)}\exp\Big(-\frac{|v_\parallel|^2}{2\big[T_0(1-r_\parallel)^2+T_w(x)r_\parallel(2-r_\parallel)\big]} \Big) . \]
Then we compute the first line of~\eqref{eqn: LHS of mu in bc}, in order to apply Lemma \ref{Lemma: perp abc}, we set
\[a=-\frac{1}{2T_0},\quad b=\frac{1-r_\perp}{2T_w(x)r_\perp},\quad v=u_\perp,\quad w=\frac{1}{\sqrt{1-r_\perp}}v_\perp,\quad \e=0,\]
\[b-a=\frac{1-r_\perp}{2T_w(x)r_\perp}+\frac{1}{2T_0}.\]
Then the first line of~\eqref{eqn: LHS of mu in bc} is equal to
\[\frac{1}{1-r_\perp}\frac{b}{b-a}e^{\frac{ab}{b-a}|\frac{v_\perp}{\sqrt{1-r_\perp}}|^2}=\frac{1}{2\pi[T_0(1-r_\perp)+T_w(x)r_\perp]}\exp\Big( \frac{|v_\perp|^2}{2\pi[T_0(1-r_\perp)+T_w(x)r_\perp]}\Big).\]
Thus we conclude~\eqref{eqn: mu in the bc}.

Then we focus on~\eqref{eqn: int is 1}. The LHS of~\eqref{eqn: int is 1} can be written as
\begin{equation}\label{eqn: mymy}
\begin{split}
   & \int_{\mathbb{R}_+}\frac{v_\perp}{T_0(1-r_\perp)+T_w(x)r_\perp}e^{-\frac{|v_\perp|^2}{2[T_0(1-r_\perp)+T_w(x)r_\perp]}}       dv_\perp    \\
    &\times \int_{\mathbb{R}^2}\frac{1}{2\pi [T_0(1-r_\parallel)^2+T_w(x)r_\parallel(2-r_\parallel)]}e^{-\frac{|v_\parallel|^2}{2[T_0(1-r_\parallel)^2+T_w(x)r_\parallel(2-r_\parallel)]}}dv_\parallel.
\end{split}
\end{equation}
Clearly $\eqref{eqn: mymy}=1$.

\end{proof}

\begin{proof}[Proof of Lemma \ref{Prop: perturbation}]
By plugging the linearization $F_s=\mu_0+\sqrt{\mu_0}f_s$ into the boundary condition~\eqref{eqn:BC} and using Lemma \ref{Lemma: mu in bc} we obtain
\[\mu_0+\sqrt{\mu_0}f_s=\mu_{x,r_\parallel,r_\perp}+\frac{1}{|n(x)\cdot v|}\int_{n(x)\cdot u>0}R(u\to v;x,t)\sqrt{\mu_0(u)}f_s(u) \{n(x)\cdot u\}du.\]
Thus
\[f_s(v)=\underbrace{\frac{\mu_{x,r_\parallel,r_\perp}-\mu_0}{\sqrt{\mu_0}}}_{r_1}+      \underbrace{\frac{1}{\sqrt{\mu_0}}\frac{1}{|n(x)\cdot v|}\int_{n(x)\cdot u>0}R(u\to v;x,t)\sqrt{\mu_0(u)}f_s(u) \{n(x)\cdot u\}du}_{r_2(f_s)}.\]
We can rewrite the boundary condition into
\begin{equation}\label{eqn: bdr of f_s in r}
f_s(v)=r_1+r_2(f_s)-P_\gamma f_s+P_\gamma f_s.
\end{equation}
Clearly by~\eqref{eqn: int is 1} in Lemma \ref{Lemma: mu in bc} we have
\begin{equation}\label{eqn: r1 int is 0}
  \int_{\gamma_-}r_1\sqrt{\mu_0}=0.
\end{equation}
To prove the Lemma we just need to focus on $r_2(f_s)-P_\gamma f_s$. By Tonelli theorem, we have
\[\int_{\gamma_-} (r_2(f_s)-P_\gamma f_s)\sqrt{\mu_0}=\int_{n(x)\cdot v<0}\Big[ R(u\to v;x,t)-|n(x)\cdot v|\mu_0(v) \Big]dv\int_{n(x)\cdot u>0} \sqrt{\mu_0(u)}f_s(u)\{n(x)\cdot u\}du  \]
\[=[1-1]\times \int_{n(x)\cdot u>0} \sqrt{\mu_0(u)}f_s(u)\{n(x)\cdot u\}du  =0.\]
Thus we prove~\eqref{eqn: int r is 0}.

Then we focus on~\eqref{eqn: bound for r}. By the assumption in~\eqref{eqn: small pert condition} and $\zeta<\frac{1}{\theta(4+2\delta_0)}$, for $x\in\partial \Omega$ we have
\[|w_\rho(v)r|_\infty=|w_\rho(v)\frac{\mu_{x,r_\parallel,r_\perp}-\mu_0}{\sqrt{\mu_0}}|_\infty\lesssim \delta_0.\]
Then
\[|w_\rho(v)\big[r_2(f_s)-P_{\gamma}f_s\big]|\leq |f|_\infty    w_{\rho}(v)\frac{1}{\sqrt{\mu_0}}\int_{n(x)\cdot u>0} \Big[\frac{R(u\to v;x,t)}{|n(x)\cdot v|}-\mu_0(v)\Big]\sqrt{\mu_0(u)}f_s(u)\{n(x)\cdot u\}du.\]
\[\leq |f|_\infty    \Big|w_{\rho}(v)\frac{\mu_{x,r_\parallel,r_\perp}-\mu_0}{\sqrt{\mu_0}}\Big|_\infty\lesssim \delta_0 |f|_\infty,\]
where we apply Lemma \ref{Lemma: mu in bc} in the last line. Then we conclude the Lemma.

\end{proof}

\begin{proof}[Proof of Corollary \ref{Thm: steady solution}]

We consider the following iterative
sequence
\begin{equation}
v\cdot \nabla _{x}f^{\ell +1}+Lf^{\ell +1}=\Gamma (f^{\ell },f^{\ell }),
\label{nsteady}
\end{equation}%
with the boundary condition given in the form~\eqref{eqn: bdr of f_s in r}
\[f_{{-}}^{\ell +1}=P_{\gamma }f^{\ell +1}+r_1+r_2(f^{\ell})-P_\gamma f^{\ell}.\]
We set $f^{0}=0$.  By Lemma \ref{Prop: perturbation} we have
\begin{equation*}
\int_{\gamma _{-}}\sqrt{\mu_0 }\left\{ r_1+r_2(f^{\ell})-P_\gamma f^{\ell}\right\} d\gamma =0.
\end{equation*}%
Since $\int \Gamma(f^{\ell},f^\ell)\sqrt{\mu_0}=0$, we apply Proposition \ref{proposition: Linfty of steady solution} with~\eqref{eqn: bound for r} in Lemma \ref{Prop: perturbation} to get
\begin{equation*}
\Vert w_{\varrho }f^{\ell +1}\Vert _{\infty }+|w_{\varrho }f^{\ell
+1}|_{\infty }\lesssim \left\Vert \frac{w_{\varrho }\Gamma (f^{\ell
},f^{\ell })}{\langle v\rangle }\right\Vert _{\infty }+\delta_0 |w_{\varrho
}f^{\ell }|_{\infty ,{+}}+\delta_0 .
\end{equation*}%
Since $\left\Vert \frac{w_{\varrho }\Gamma (f^{\ell },f^{\ell })}{\langle
v\rangle }\right\Vert _{\infty }\lesssim \Vert w_{\varrho }f^{\ell }\Vert
_{\infty }^{2}$, we deduce
\begin{equation*}
\Vert w_{\varrho }f^{\ell +1}\Vert _{\infty }+|w_{\varrho }f^{\ell
+1}|_{\infty }\lesssim \Vert w_{\varrho }f^{\ell }\Vert _{\infty
}^{2}+\delta_0 |w_{\varrho }f^{\ell }|_{\infty ,+}+\delta_0 ,
\end{equation*}%
so that for $\delta_0 $ small, $\Vert w_{\varrho }f^{\ell +1}\Vert _{\infty
}+|w_{\varrho }f^{\ell +1}|_{\infty }\lesssim \delta_0 .$ Upon taking
differences, we have
\begin{eqnarray*}
&&[f^{\ell +1}-f^{\ell }]+v\cdot \nabla _{x}[f^{\ell +1}-f^{\ell
}]+L[f^{\ell +1}-f^{\ell }] =\Gamma (f^{\ell }-f^{\ell -1},f^{\ell })+\Gamma
(f^{\ell -1},f^{\ell }-f^{\ell -1}), \\
&&f_{{-}}^{\ell +1}-f_{{-}}^{\ell }=P_{\gamma }[f^{\ell +1}-f^{\ell }]+r_2(f^\ell)-P_\gamma f^\ell+P_\gamma f^{\ell-1}-r_2(f^{\ell-1}).
\end{eqnarray*}%
And by Proposition \ref{proposition: Linfty of steady solution} again for $f^{\ell +1}-f^{\ell },$
\begin{equation*}
\Vert w_{\varrho }[f^{\ell +1}-f^{\ell }]\Vert _{\infty }+|w_{\varrho
}[f^{\ell +1}-f^{\ell }]|_{\infty }\lesssim \delta_0 \big\{\Vert w_{\varrho
}[f^{\ell }-f^{\ell -1}]\Vert _{\infty }+|w_{\varrho
}[f^{\ell}-f^{\ell-1}]|_{\infty }\big\}.
\end{equation*}%
Hence $f^{\ell }$ is Cauchy in $L^{\infty }$ and we construct our solution
by taking the limit $f^{\ell }\rightarrow f_{s}$. Uniqueness follows in the
standard way.

\end{proof}

Then we focus on the dynamical stability, which is the Corollary \ref{Thm: dynamical stability}. We need this Proposition.
\begin{proposition}[Proposition 7.1 from \cite{EGKM}]\label{Prop: Linfty bound for dynamical}
\label{dlinearlinfty} Let $\|w_\rho f_{0}\|_{\infty }+|\langle v\rangle
w_\rho r|_{\infty }+\|w_\rho g\|_{\infty }<+\infty $ and $\iint \sqrt{\mu_0 }%
g=\int_{\gamma }r\sqrt{\mu_0 } = \iint f_0 \sqrt{\mu_0} =0$. Then the solution $%
f $
\begin{equation}\label{eqn: linearzied dynamical}
\partial_t f+v\cdot \nabla_x f+Lf=g,\quad f(0)=f_0,\quad \text{ in } \Omega\times \mathbb{R}^3\times \mathbb{R}_+
\end{equation}
satisfies
\begin{equation*}
\|w_\rho f(t)\|_{\infty }+|w_\rho f(t)|_{\infty }\leq e^{-\lambda t}\big\{%
\|w_\rho f_{0}\|_{\infty }+\sup e^{\lambda s}\|w_\rho g\|_{\infty
}+\int_{0}^{t}e^{\lambda s}|\langle v\rangle w_\rho r(s)|_{\infty }ds\big\}.
\end{equation*}%

\end{proposition}

\begin{proof}[Proof of Corollary \ref{Thm: dynamical stability}]
With the stationary solution for~\eqref{eqn: Steady Boltzmann} given in Corollary \ref{Thm: steady solution}, we set the solution to~\eqref{eqn: VPB equation} as
\[F=F_s+\sqrt{\mu_0}f,\quad F_s=\sqrt{\mu_0}f+f_s.\]
Then the equation for $f$ reads
\[\partial_t f+v\cdot \nabla_x f+Lf=L_{\sqrt{\mu_0}f_s}f+\Gamma(f,f),\]
where
\[L_{\sqrt{\mu_0}f_s}f=[Q(\sqrt{\mu_0}f_s,\sqrt{\mu_0}f)+Q(\sqrt{\mu_0}f,\sqrt{\mu_0}f_s)]/\sqrt{\mu_0}.\]

We consider the following iteration sequence
\begin{equation*}
\partial _{t}f^{\ell +1}+v\cdot \nabla _{x}f^{\ell +1}+Lf^{\ell +1}=L_{\sqrt{%
\mu_0 }f_{s}}f^{\ell }+\Gamma (f^{\ell },f^{\ell }),
\end{equation*}%
with
\[f_{{-}}^{\ell +1}=P_{\gamma }f^{\ell +1}+r_1+r_2(f^{\ell})-P_\gamma f^{\ell}.\]
Clearly $\iint \{L_{\sqrt{\mu_0 }f_{s}}f^{\ell }+\Gamma (f^{\ell },f^{\ell })\}%
\sqrt{\mu_0 }=0.$ Recall $w_{\varrho
}(v)=(1+\varrho ^{2}|v|^{2})^{\frac{\beta }{2}}e^{\zeta |v|^{2}}$ in (\ref%
{weight}). Note that for $0\leq \zeta <\frac{1}{4},$
\begin{equation*}
\left\Vert e^{\frac{\lambda s}{2}}w_{\varrho }\left\{ \frac{1}{\langle
v\rangle }[L_{\sqrt{\mu_0 }f_{s}}f^{\ell }+\Gamma (f^{\ell },f^{\ell
})(s)\right\} \right\Vert _{\infty }\lesssim  \delta_0 \sup_{0\leq s\leq t}\Vert e^{\frac{\lambda s}{2}}w_{\varrho }f^{\ell
}(s)\Vert _{\infty }+\left\{ \sup_{0\leq s\leq t}\Vert e^{\frac{\lambda s}{2}%
}w_{\varrho }f^{\ell }(s)\Vert _{\infty }\right\} ^{2}.
\end{equation*}%
By Proposition \ref{dlinearlinfty} and Lemma \ref{Prop: perturbation}, we deduce
\begin{eqnarray*}
&&\sup_{0\leq s\leq t}\Vert e^{\frac{\lambda s}{2}}w_{\varrho }f^{\ell
+1}(s)\Vert _{\infty }+\sup_{0\leq s\leq t}|e^{\frac{\lambda s}{2}%
}w_{\varrho }f^{\ell +1}(s)|_{\infty } \\
&\lesssim &\Vert w_{\varrho }f_{0}\Vert _{\infty }+\delta_0 \sup_{0\leq s\leq
t}\Vert e^{\frac{\lambda s}{2}}w_{\varrho }f^{\ell }(s)\Vert _{\infty
}\\
&& \ \ \ \   +\delta_0 \sup_{0\leq s\leq t}|e^{\frac{\lambda s}{2}}w_{\varrho }f^{\ell
}(s)|_{\infty }+\left\{ \sup_{0\leq s\leq t}\Vert e^{\frac{\lambda s}{2}%
}w_{\varrho }f^{\ell }(s)\Vert _{\infty }\right\} ^{2}.
\end{eqnarray*}%
For $\delta_0 $ small, there exists a $\varepsilon _{0}$ (uniform in $\delta_0 $%
) such that, if the initial data satisfy (\ref{epsilon0}), then
\begin{equation*}
\sup_{0\leq s\leq t}\Vert e^{\frac{\lambda s}{2}}w_{\varrho }f^{\ell
+1}(s)\Vert _{\infty }+\sup_{0\leq s\leq t}|e^{\frac{\lambda s}{2}%
}w_{\varrho }f^{\ell +1}(s)|_{\infty }\lesssim \Vert w_{\varrho }f_{0}\Vert
_{\infty }.
\end{equation*}%
By taking difference $f^{\ell +1}-f^{\ell }$, we deduce that%
\begin{eqnarray*}
&&\partial _{t}[f^{\ell +1}-f^{\ell }]+v\cdot \nabla _{x}[f^{\ell
+1}-f^{\ell }]+L[f^{\ell +1}-f^{\ell }] \\
&&\ \ \ \ =L_{\sqrt{\mu_0 }f_{s}}[f^{\ell }-f^{\ell -1}]+\Gamma (f^{\ell
}-f^{\ell -1},f^{\ell })+\Gamma (f^{\ell -1},f^{\ell }-f^{\ell -1}), \\
&&[f^{\ell +1}-f^{\ell }]_{-}=P_{\gamma }[f^{\ell +1}-f^{\ell }]+\frac{\mu_{x,r_\parallel,r_\perp}-\mu_0 }{\sqrt{\mu_0 }}\int_{\gamma _{+}}[f^{\ell }-f^{\ell
-1}](n(x)\cdot v)dv,
\end{eqnarray*}%
with $f^{\ell +1}-f^{\ell }=0$ initially. Repeating the same argument, we
obtain
\begin{eqnarray*}
 &&\sup_{0\leq s\leq t}\Vert e^{\frac{\lambda s}{2}}w_{\varrho }[f^{\ell
+1}-f^{\ell }](s)\Vert _{\infty }+\sup_{0\leq s\leq t}|e^{\frac{\lambda s}{2}%
}w_{\varrho }[f^{\ell +1}-f^{\ell }](s)|_{\infty } \\
& & \ \ \ \ \ \ \ \lesssim[\delta_0 +\sup_{0\leq s\leq t}\Vert e^{\frac{\lambda s}{2}%
}w_{\varrho }f^{\ell }(s)\Vert _{\infty } \\
&& \ \ \ \ \ \ \ \ \ \ \ \ +\sup_{0\leq s\leq t}\Vert e^{\frac{\lambda s}{2}}w_{\varrho }f^{\ell
-1}(s)\Vert _{\infty }]\sup_{0\leq s\leq t}\Vert e^{\frac{\lambda s}{2}%
}w_{\varrho }[f^{\ell }-f^{\ell -1}](s)\Vert _{\infty }.
\end{eqnarray*}%
This implies that $f^{\ell +1}$ is a Cauchy sequence. The uniqueness is
standard.

To conclude the positivity, we use another sequence in~\cite{EGKM},
\begin{equation*}
\partial_t F^{\ell+1}+v\cdot \nabla_x F^{\ell+1}+\nu(F^{\ell})F^{\ell+1}=Q_{\text{gain}}(F^{\ell},F^{\ell}).
\end{equation*}
We pose $F^\ell=F_s+\sqrt{\mu_0}f^\ell$, then the equation for $f^{\ell}$ reads
\[\partial_t f^{\ell+1}+v\cdot \nabla_x f^{\ell+1}+\nu(v)f^{\ell+1}-K f^{\ell}\]
\[=\Gamma_{\text{gain}}(f^{\ell},f^{\ell})-\nu(\sqrt{\mu_0}f^{\ell})f^{\ell+1}-\nu(\sqrt{\mu_0}f_s)f^{\ell+1}-\nu(\sqrt{\mu_0}f^{\ell})f_s\]
\[+\frac{1}{\sqrt{\mu_0}}\Big\{Q_{\text{gain}}(\sqrt{\mu_0}f^\ell,\sqrt{\mu}f_s)+Q_{\text{gain}}(\sqrt{\mu_0}f_s,\sqrt{\mu_0}f^\ell)  \Big\}.\]
It is shown in~\cite{EGKM} that $f^{\ell}$ is a Cauchy sequence. Thus by the uniqueness of the solution we conclude the positivity of $F$ and $F_s$ by positive preserving property of this sequence solution.

\end{proof}

\section{Appendix}
\begin{lemma}\label{Lemma: Prob measure}
For $R(u\to v;x,t)$ given by~\eqref{eqn: Formula for R} and any $u$ such that $u\cdot n(x)>0$, we have
\begin{equation}\label{eqn: integrate 1}
  \int_{n(x)\cdot v<0}R(u\to v;x,t)dv=1.
\end{equation}

\end{lemma}

\begin{proof}

We transform the basis from $\{\tau_1,\tau_2,n\}$ to the standard bases $\{e_1,e_2,e_3\}$. For simplicity, we assume $T_w(x)=1$. The integration over $\mathcal{V}_\parallel$( defined in~\eqref{eqn: Define space} ), after the orthonormal transformation, becomes integration over $\mathbb{R}^2$. We have
\[\int_{\mathbb{R}^2} \frac{1}{r_\parallel(2-r_\parallel)}  \exp\Big(\frac{|v_\parallel-(1-r_\parallel)u_\parallel|^2}{r_\parallel(2-r_\parallel)} \Big)dv_\parallel,\]
which is obviously normalized.

Then we consider the integration over $\mathcal{V}_\perp$, which is $e_3<0$ after the transformation. We want to show
%We use the same orthonormal transformation from $\{n,\tau_1,\tau_2\}$ to the standard bases $\{e_1,e_2,e_3\}$,
%\[\mathcal{T}(x)n(x)=e_1,~~ \mathcal{T}(x)\tau_1(x)=e_2,~~\mathcal{T}(x)\tau_2(x)=e_3,~~\mathcal{T}^{-1}=\mathcal{T}^T\]
%by a change of variable $u=\mathcal{T}(x)v$, then we have
%\[v_n=n(x)\cdot v=n(x)\cdot \mathcal{T}^T(x)u=n(x)^T \mathcal{T}^T(x)u=[\mathcal{T}(x)n(x)]^T u=e_1\cdot u=u_1\]
%\[v_\tau=[\tau_1(x)\cdot v]\tau_1(x)+[\tau_2(x)\cdot v]\tau_2(x)=[\tau_1(x)\cdot \mathcal{T}^T(x)u]\tau_1(x)+[\tau_2(x)\cdot \mathcal{T}^T(x)u]\tau_2(x)\]
%\[=\{[\mathcal{T}\tau_1(x)]^Tu\}\tau_1(x)+\{[\mathcal{T}\tau_2(x)]^Tu\}\tau_2(x)=u_2 \tau_1(x)+u_3 \tau_2(x)=u_2\mathcal{T}^T(x)e_2+u_3\mathcal{T}^T(x)e_3\]
%\[u_\tau=[\tau_1\cdot u]\tau_1(x)+[\tau_2 \cdot u]\tau_2(x)=[\mathcal{T}^T(x)e_2\cdot u]\mathcal{T}^T(x)e_2+[\mathcal{T}^T(x)e_3\cdot u]\mathcal{T}^T(x)e_3\]
%We derive
%\[\int_{n(x)\cdot v<0}R(u\to v;x,t)dv=\int_{u_1<0}\frac{1}{r_\perp r_\parallel(2-r_\parallel)\pi/2}\frac{1}{(2T_w(x))^2}e^{-\frac{1}{2T_w(x)}}[\frac{|u_1|^2+(1-r_\perp)|u_\perp|^2}{r_\perp}+]\]
%\\\\\\\\

\begin{equation}\label{eqn: I0}
\frac{2}{r_\perp}\int_{-\infty}^0 -v_\perp e^{-\frac{|v_\perp|^2}{r_\perp}}e^{\frac{-(1-r_\perp)|u_\perp|^2}{r_\perp}}I_0\Big(\frac{2(1-r_\perp)^{1/2}v_\perp u_\perp}{r_\perp}\Big)dv_\perp=1.
\end{equation}
The Bessel function reads
\[J_0(y)=\frac{1}{\pi}\int_0^{\pi} e^{iy\cos\theta}d\theta=\sum_{k=0}^\infty \frac{1}{\pi}\int_0^\pi \frac{(iy\cos\theta)^k}{k!} d\theta=\sum_{k=0}^\infty \int_0^\pi  \frac{(iy\cos\theta)^{2k}}{(2k)!}  d\theta\]
\[\sum_{k=0}^\infty   \int_0^\pi  \frac{(-1)^k (y)^{2k} (\cos\theta)^{2k}}{(2k)!}d\theta=\sum_{k=0}^\infty (-1)^k \frac{(\frac{1}{4}y^2)^k}{(k!)^2},\]
where we use the Fubini's theorem and the fact that
\[\int_0^\pi \cos^{2k}\theta=\frac{\pi}{2^{2k}}\left(
                                                 \begin{array}{c}
                                                   2k \\
                                                   k \\
                                                 \end{array}
                                               \right).
\]
Hence
\begin{equation}\label{eqn: I0 sequence}
I_0(y)=\frac{1}{\pi}\int_0^\pi e^{i(-iy)\cos \theta}d\theta =J_0(-iy)=\sum_{k=0}^\infty \frac{(\frac{1}{4}y^2)^k}{(k!)^2},\quad I_0(y)=I_0(-y).
\end{equation}
By taking the change of variable $v_\perp\to -v_\perp$, the LHS of \eqref{eqn: I0} can be written as
\[\frac{2}{r_\perp}\int_{0}^\infty v_\perp e^{-\frac{|v_\perp|^2}{r_\perp}}e^{\frac{-(1-r_\perp)|u_\perp|^2}{r_\perp}}I_0\Big(\frac{2(1-r_\perp)^{1/2}v_\perp u_\perp}{r_\perp}\Big)dv_\perp.\]
Using~\eqref{eqn: I0 sequence} we rewrite the above term as
\begin{equation}\label{eqn: I_0 tough}
\sum_{k=0}^\infty \frac{2}{r_\perp}\int_0^\infty v_\perp e^{\frac{-|v_\perp|^2}{r_\perp}} e^{\frac{-(1-r_\perp)|u_\perp|^2}{r_\perp}}\frac{(1-r_\perp)^k v_\perp^{2k}u_\perp^{2k}}{(k!)^2r_\perp^{2k}}dv,
\end{equation}
where we use the Tonelli theorem. By rescaling $v_\perp=\sqrt{r_\perp}v_\perp$ we have
\[\frac{2}{r_\perp}\int_0^\infty v_\perp e^{\frac{-|v_\perp|^2}{r_\perp}} e^{\frac{-(1-r_\perp)|u_\perp|^2}{r_\perp}}\frac{(1-r_\perp)^k v_\perp^{2k}u_\perp^{2k}}{(k!)^2r_\perp^{2k}}dv\]
\[=2\int_0^\infty v_\perp e^{-|v_\perp|^2} e^{\frac{-(1-r_\perp)|u_\perp|^2}{r_\perp}}\frac{(1-r_\perp)^k v_\perp^{2k}u_\perp^{2k}}{(k!)^2r_\perp^{k}}dv\]
\begin{equation}\label{eqn: appendix for I0}
=2\int_0^\infty v_\perp^{2k+1}e^{-|v_\perp|^2}dv    e^{\frac{-(1-r_\perp)|u_\perp|^2}{r_\perp}}\frac{(1-r_\perp)^k u_\perp^{2k}}{(k!)^2r_\perp^{k}}
\end{equation}
\[=2\frac{k!}{2}e^{\frac{-(1-r_\perp)|u_\perp|^2}{r_\perp}}\frac{(1-r_\perp)^k u_\perp^{2k}}{(k!)^2r_\perp^{k}}=e^{\frac{-(1-r_\perp)|u_\perp|^2}{r_\perp}}\frac{(1-r_\perp)^k u_\perp^{2k}}{k!r_\perp^{k}}.\]
Therefore, the LHS of~\eqref{eqn: I0} can be written as
\[e^{\frac{-(1-r_\perp)|u_\perp|^2}{r_\perp}}\sum_{k=0}^\infty \frac{(1-r_\perp)^k u_\perp^{2k}}{k!r_\perp^{k}} =e^{\frac{-(1-r_\perp)|u_\perp|^2}{r_\perp}}e^{\frac{(1-r_\perp)|u_\perp|^2}{r_\perp}}=1.\]

\end{proof}

\begin{lemma}\label{Lemma: abc}
For any $a>0,b>0,\e>0$ with $a+\e<b$,
\begin{equation}\label{eqn: coe abc}
\frac{b}{\pi}\int_{\mathbb{R}^2} e^{\e|v|^2}  e^{a|v|^2}e^{-b|v-w|^2}dv=\frac{b}{b-a-\e}e^{\frac{(a+\e)b}{b-a-\e}|w|^2}.
\end{equation}
And when $\delta\ll 1$,
\begin{eqnarray}
     \frac{b}{\pi}\int_{|v-\frac{b}{b-a-\e}w|>\delta^{-1}} e^{\e|v|^2}  e^{a|v|^2}e^{-b|v-w|^2}dv  &\leq  &e^{-(b-a-\e)\delta^{-2}} \frac{b}{b-a-\e} e^{\frac{(a+\e)b}{b-a-\e}|w|^2} \label{eqn: coe abc smaller} \\
   & \leq & \delta \frac{b}{b-a-\e}e^{\frac{(a+\e)b}{b-a-\e}|w|^2} \label{eqn: coe abc small}.
\end{eqnarray}

\end{lemma}

\begin{proof}
\[\frac{b}{\pi}\int_{\mathbb{R}^2} e^{\e|v|^2}  e^{a|v|^2}e^{-b|v-w|^2}dv = \frac{b}{\pi}\int_{\mathbb{R}^2} e^{(a+\e-b)|v|^2} e^{2bv\cdot w} e^{-b|w|^2}dv   \]
\[=\frac{b}{\pi}\int_{\mathbb{R}^2}   e^{(a+\e-b)|v+\frac{b}{a+\e-b}w|^2} e^{\frac{-b^2}{a+\e-b}|w|^2} e^{-b|w|^2}dv   \]
\[=\frac{b}{\pi}\int_{\mathbb{R}^2} e^{(a+\e-b)|v|^2}dv e^{\frac{(a+\e)b}{b-a-\e}|w|^2}=\frac{b}{b-a-\e}e^{\frac{(a+\e)b}{b-a-\e}|w|^2},\]
where we apply change of variable $v+\frac{b}{a+\e-b}w\to v$ in the first step of the last line, then we obtain~\eqref{eqn: coe abc}.

Following the same derivation
\[\frac{b}{\pi}\int_{|v-\frac{b}{b-a-\e}w|>\delta^{-1}} e^{\e|v|^2}  e^{a|v|^2}e^{-b|v-w|^2}dv=\frac{b}{\pi}\int_{|v-\frac{b}{b-a-\e}w|>\delta^{-1}} e^{(a+\e-b)|v-\frac{b}{b-a-\e}w|^2}dv e^{\frac{(a+\e)b}{b-a-\e}|w|^2}\]
\[\leq  e^{-(b-a-\e)\delta^{-2}} \frac{b}{b-a-\e} e^{\frac{(a+\e)b}{b-a-\e}|w|^2} \leq \delta     \frac{b}{b-a-\e} e^{\frac{(a+\e)b}{b-a-\e}|w|^2},\]
thus we obtain~\eqref{eqn: coe abc small}.

\end{proof}

\begin{lemma}\label{Lemma: perp abc}
For any $a>0,b>0,\e>0$ with $a+\e<b$,
\begin{equation}\label{eqn: coe abc perp}
2b\int_{\mathbb{R}^+}v e^{\e v^2}e^{av^2} e^{-bv^2}e^{-bw^2}I_0(2bv w)dv=\frac{b}{b-a-\e}e^{\frac{(a+\e)b}{b-a-\e}w^2}.
\end{equation}
And when $\delta\ll 1$,
\begin{equation}\label{eqn: coe abc perp small}
2b\int_{0< v<\delta}v e^{\e v^2}e^{av^2} e^{-bv^2}e^{-bw^2}I_0(2bv w)dv\leq \delta\frac{b}{b-a-\e}e^{\frac{(a+\e)b}{b-a-\e}w^2}.
\end{equation}

\end{lemma}

\begin{proof}
\[2b\int_{\mathbb{R}^+}v e^{\e v^2}e^{av^2} e^{-bv^2}e^{-bw^2}I_0(2bv w)dv\]
\[=2b\int_{\mathbb{R}^+} v e^{(a+\e-b)v^2}I_0(2bv w) e^{\frac{b^2}{a+\e-b}w^2} e^{\frac{b^2}{b-a-\e}w^2} dv e^{-bw^2}\]
\[=2(b-a-\e)\int_{\mathbb{R}^+}v e^{(a+\e-b)v^2}I_0(2bv w)e^{\frac{(bw)^2}{a+\e-b}}dv \frac{b}{b-a-\e}e^{\frac{(a+\e)b}{b-a-\e}w^2}\]
\[=\frac{b}{b-a-\e}e^{\frac{(a+\e)b}{b-a-\e}w^2},\]
where we use~\eqref{eqn: I0} in Lemma~\ref{Lemma: Prob measure} in the last line, then we obtain~\eqref{eqn: coe abc perp}.

Following the same derivation we have
\[2b\int_{0< v< \delta}v e^{\e v^2}e^{av^2} e^{-bv^2}e^{-bw^2}I_0(2bv w)dv\]
\[=2(b-a-\e)\int_{0<v<\delta} v e^{(a+\e-b)v^2}I_0(2bv w)e^{\frac{(bw)^2}{a+\e-b}}dv \frac{b}{b-a-\e}e^{\frac{(a+\e)b}{b-a-\e}w^2}.\]
Using the definition of $I_0$ we have
\[I_0(y)=\frac{1}{\pi}\int_{0}^{\pi} e^{y\cos\phi}d\phi\leq e^{y}.\]
Thus when $a-b+\e<0$,
\[2(b-a-\e)\int_{0<v<\delta} v e^{(a+\e-b)v^2}I_0(2bv w)e^{\frac{(bw)^2}{a+\e-b}}dv\]
\[\leq 2(b-a-\e)\int_{0<v<\delta} v e^{(a-b+\e)v^2}e^{2v b w}e^{\frac{(bw)^2}{a-b+\e}}=2(b-a-\e)\int_{0<v<\delta} v e^{(a-b+\e)(v+\frac{bw}{a-b+\e})^2}dv\]
\[\leq 2(b-a-\e)\int_{0<v<\delta}vdv<\delta,\]
where we use $\delta\ll 1$ in the last step, then we obtain~\eqref{eqn: coe abc perp small}. Then we derive~\eqref{eqn: coe perp small 2}.

\end{proof}

\begin{lemma}\label{Lemma: integrate normal small}
For any $m,n>0$, when $\delta\ll 1$, we have
\begin{equation}\label{eqn: smallness for i0}
2m^2\int_{\frac{n}{m}u_\perp+\delta^{-1}}^\infty     v_\perp e^{-m^2v_\perp^2}I_0(2mnv_\perp u_\perp)e^{-n^2u_\perp^2}dv_\perp \lesssim e^{-\frac{m^2}{4\delta^{2}}}.
\end{equation}
In consequence, for any $a>0,b>0,\e>0$ with $a+\e<b$,
\begin{eqnarray}
 2b\int_{\frac{b}{b-a-\e}w+\delta^{-1}}^\infty v e^{\e v^2}e^{av^2} e^{-bv^2}e^{-bw^2}I_0(2bv w)dv  &\leq  & e^{\frac{-(b-a-\e)}{4\delta^2}}\frac{b}{b-a-\e}e^{\frac{(a+\e)b}{b-a-\e}w^2} \label{eqn: coe perp smaller 2}\\
   &\leq  & \delta\frac{b}{b-a-\e}e^{\frac{(a+\e)b}{b-a-\e}w^2}. \label{eqn: coe perp small 2}
\end{eqnarray}

\end{lemma}

\begin{proof}
We discuss two cases. The first case is $v_\perp>2\frac{n}{m}u_\perp$. We bound $I_0$ as
\[I_0(2mnv_\perp u_\perp)\leq \frac{1}{\pi}\int_0^\pi \exp \Big( 2mnv_\perp u_\perp\Big) d\theta=\exp \Big(2mnv_\perp u_\perp\Big).\]
The LHS of~\eqref{eqn: smallness for i0} is bounded by
\[2m^2\int_{\max\{2\frac{n}{m}u_\perp,\frac{n}{m}u_\perp+\delta^{-1}\}}^\infty    ve^{-m^2(v_\perp-\frac{n}{m}u_\perp)^2}dv.\]
Using $v_\perp>2\frac{n}{m}u_\perp$ we have
\begin{equation*}
(v_\perp-\frac{n}{m}u_\perp)^2\geq (\frac{v_\perp}{2}+\frac{v_\perp}{2}-\frac{n}{m}u_\perp)^2 \geq \frac{v_\perp^2}{4}.
\end{equation*}
Thus we can further bound LHS of~\eqref{eqn: smallness for i0} by
\begin{equation*}
2m^2\int_{\max\{2\frac{n}{m}u_\perp,\frac{n}{m}u_\perp+\delta^{-1}\}}^\infty      v_\perp e^{-\frac{m^2 v_\perp^2}{4}} dv_\perp \lesssim e^{-\frac{m^2}{4\delta^2}}.
\end{equation*}

The second case is $\leq v_\perp\leq 2\frac{n}{m}u_\perp$. Since $\frac{n}{m}u_\perp+\delta^{-1}<v_\perp$, without loss of generality, we can assume $u_\perp>\delta^{-1}$. We compare the Taylor series of $v_\perp I_0(2mnv_\perp u_\perp)$ and $\exp \Big(2mnv_\perp u_\perp \Big)$. We have
\begin{equation}\label{eqn: vI_0 taylor}
v_\perp I_0(2mnv_\perp u_\perp )=\sum_{k=0}^\infty \frac{m^{2k}n^{2k}v_\perp^{2k+1}u_\perp^{2k}}{(k!)^2},
\end{equation}
and
\begin{equation}\label{eqn: exp taylor}
\exp \Big(2mnv_\perp u_\perp \Big)=\sum_{k=0}^\infty \frac{2^k m^k n^k v_\perp^k u_\perp^k}{k!}.
\end{equation}
We choose $k_1$ such that when $k>k_1$, we can apply the Sterling formula such that
\[\frac{1}{2}\leq |\frac{k!}{k^ke^{-k}\sqrt{2\pi k}}|\leq 2.\]
Then we observe the quotient of the $k$-th term of~\eqref{eqn: vI_0 taylor} and the $2k+1$-th term of~\eqref{eqn: exp taylor},
\[\frac{m^{2k}n^{2k}v_\perp^{2k+1}u_\perp^{2k}}{(k!)^2}/\Big(\frac{2^{2k+1} m^{2k+1}n^{2k+1}v_\perp^{2k+1} u_\perp^{2k+1}}{(2k+1)!} \Big)\]
\[\leq \frac{4}{k^{2k}e^{-2k}2\pi k}/\Big(\frac{2^{2k+1} mn u_\perp}{(2k+1)^{2k+1}e^{-(2k+1)}\sqrt{2\pi (2k+1)}} \Big)\]
\[= \frac{4e}{2\pi mn}\Big(\frac{k+1/2}{k} \Big)^{2k+1}  \frac{\sqrt{2\pi (2k+1)}}{u_\perp}     \]
\[= \frac{4e}{2\pi mn}\Big(\frac{2k+1}{2k} \Big)^{2k+1} \frac{\sqrt{2\pi (2k+1)}}{u_\perp}\leq \frac{4e^2}{\sqrt{\pi} mn} \frac{\sqrt{k}}{u_\perp}.\]
Thus we can take $k_u=u_\perp^2$ such that when $k\leq k_u$,
\begin{equation}\label{k<ku}
\sum_{k=k_1}^{k_u} \frac{m^{2k}n^{2k}v_\perp^{2k+1}u_\perp^{2k}}{(k!)^2}\leq \frac{4e^2}{\sqrt{\pi}mn}\sum_{k=k_1}^{k_u}    \frac{2^{2k+1} m^{2k+1}n^{2k+1}v_\perp^{2k+1} u_\perp^{2k+1}}{(2k+1)!}.
\end{equation}
Similarly we observe the quotient of the $k$-th term of~\eqref{eqn: vI_0 taylor} and the $2k$-th term of~\eqref{eqn: exp taylor},
\[\frac{m^{2k}n^{2k}v_\perp^{2k+1}u_\perp^{2k}}{(k!)^2}/\Big(\frac{2^{2k} m^{2k}n^{2k}v_\perp^{2k} u_\perp^{2k}}{(2k)!} \Big)
\]
\[\leq \frac{4v_\perp}{k^{2k}e^{-2k}2\pi k}/\Big(\frac{2^{2k}}{(2k)^{2k}e^{-2k}\sqrt{4\pi k}} \Big)=\frac{4v_\perp}{\sqrt{\pi} \sqrt{k}}.\]
When $k>k_u=u_\perp^2$, by $u_\perp>\delta^{-1}$ and $v_\perp<2\frac{n}{m}u_\perp$ we have
\[\frac{4v_\perp}{\sqrt{\pi} \sqrt{k}}\leq \frac{4v_\perp}{\sqrt{\pi}u_\perp}\leq \frac{8n}{m\sqrt{\pi}}.\]
Thus we have
\begin{equation}\label{eqn: k>ku}
\sum_{k=k_u}^\infty \frac{m^{2k}n^{2k}v_\perp^{2k+1}u_\perp^{2k}}{(k!)^2}\leq \frac{8n}{m\sqrt{\pi}}\sum_{k=k_u}^\infty \frac{2^{2k} m^{2k}n^{2k}v_\perp^{2k} u_\perp^{2k}}{(2k)!}.
\end{equation}

Collecting~\eqref{eqn: k>ku}~\eqref{k<ku}, when $v_\perp<2\frac{n}{m}u_\perp$, we obtain
\begin{equation}\label{I0 <exp}
v_\perp I_0(2mnv_\perp u_\perp )\lesssim \exp \Big(\frac{2(1-r_\perp)^{1/2} v_\perp u_\perp}{r_\perp} \Big).
\end{equation}
By~\eqref{I0 <exp}, we have
\[\int_{\frac{n}{m}u_\perp+\delta^{-1}}^{2\frac{n}{m}u_\perp}v_\perp I_0(2mnv_\perp u_\perp)) e^{-m^2v_\perp^2}e^{n^2 v_\perp^2}dv\]
\begin{equation}\label{eqn: middle}
\lesssim \int_{\frac{n}{m}u_\perp+\delta^{-1}}^{2\frac{n}{m}u_\perp}    e^{-m^2(v_\perp-\frac{n}{m}u_\perp)^2} dv\leq e^{-m^2\delta^{-2}}.
\end{equation}
Collecting~\eqref{eqn: exp taylor} and~\eqref{eqn: middle} we prove~\eqref{eqn: smallness for i0}.

Then following the same derivation as~\eqref{eqn: coe abc perp},
\[2b\int_{\frac{b}{b-a-\e}w+\delta^{-1}}^\infty v e^{\e v^2}e^{av^2} e^{-bv^2}e^{-bw^2}I_0(2bv w)dv\]
\[=2(b-a-\e)\int_{\frac{b}{b-a-\e}w+\delta^{-1}}^\infty v e^{(a+\e-b)v^2}I_0(2bv w)e^{\frac{(bw)^2}{a+\e-b}}dv \frac{b}{b-a-\e}e^{\frac{(a+\e)b}{b-a-\e}w^2}\]
\[\leq e^{\frac{-(b-a-\e)}{4\delta^2}}\frac{b}{b-a-\e}e^{\frac{(a+\e)b}{b-a-\e}w^2} \leq \delta \frac{b}{b-a-\e}e^{\frac{(a+\e)b}{b-a-\e}w^2},\]
where we apply~\eqref{eqn: smallness for i0} in the first step in the third line and take $\delta\ll 1$ in the last step of the third line.

\end{proof}

\textbf{Acknowledgements.} The author thanks his advisors Chanwoo Kim and Qin Li for helpful discussion.

\bibliographystyle{amsplain}
\bibliography{CLB}

\providecommand{\bysame}{\leavevmode\hbox to3em{\hrulefill}\thinspace}
\providecommand{\MR}{\relax\ifhmode\unskip\space\fi MR }
\providecommand{\MRhref}[2]{%
  \href{http://www.ams.org/mathscinet-getitem?mr=#1}{#2}
}
\providecommand{\href}[2]{#2}
\begin{thebibliography}{10}

\bibitem{CKL}
Yunbai Cao, Chanwoo Kim, and Donghyun Lee, \emph{{Global strong solutions of
  the Vlasov-Poisson-Boltzmann system in bounded domains}}, arXiv preprint
  arXiv:1710.03851 (2017).

\bibitem{CC}
Carlo Cercignani, \emph{The boltzmann equation}, The Boltzmann equation and its
  applications, Springer, 1988, pp.~40--103.

\bibitem{CIP}
Carlo Cercignani, Reinhard Illner, and Mario Pulvirenti, \emph{The mathematical
  theory of dilute gases}, vol. 106, Springer Science \& Business Media, 2013.

\bibitem{CL}
Carlo Cercignani and Maria Lampis, \emph{{Kinetic models for gas-surface
  interactions}}, transport theory and statistical physics \textbf{1} (1971),
  no.~2, 101--114.

\bibitem{C}
TG~Cowling, \emph{{On the Cercignani-Lampis formula for gas-surface
  interactions}}, Journal of Physics D: Applied Physics \textbf{7} (1974),
  no.~6, 781.

\bibitem{Duan}
Renjun Duan, Feimin Huang, Yong Wang, and Zhu Zhang, \emph{{Effects of soft
  interaction and non-isothermal boundary upon long-time dynamics of rarefied
  gas}}, arXiv preprint arXiv:1807.05700 (2018).

\bibitem{EGKM}
R~Esposito, Y~Guo, C~Kim, and R~Marra, \emph{{Non-isothermal boundary in the
  Boltzmann theory and Fourier law}}, Communications in Mathematical Physics
  \textbf{323} (2013), no.~1, 177--239.

\bibitem{EGKM2}
Raffaele Esposito, Yan Guo, Chanwoo Kim, and Rossana Marra, \emph{{Stationary
  solutions to the Boltzmann equation in the hydrodynamic limit}}, Annals of
  PDE \textbf{4} (2018), no.~1, 1.

\bibitem{CS}
RDM Garcia and CE~Siewert, \emph{{The linearized Boltzmann equation with
  Cercignani--Lampis boundary conditions: Basic flow problems in a plane
  channel}}, European Journal of Mechanics-B/Fluids \textbf{28} (2009), no.~3,
  387--396.

\bibitem{Gar}
\bysame, \emph{{Viscous-slip, thermal-slip, and temperature-jump coefficients
  based on the linearized Boltzmann equation (and five kinetic models) with the
  Cercignani--Lampis boundary condition}}, European Journal of
  Mechanics-B/Fluids \textbf{29} (2010), no.~3, 181--191.

\bibitem{G2}
Yan Guo, \emph{{Bounded solutions for the Boltzmann equation}}, Quarterly of
  Applied Mathematics \textbf{68} (2010), no.~1, 143--148.

\bibitem{G}
\bysame, \emph{{Decay and continuity of the Boltzmann equation in bounded
  domains}}, Archive for rational mechanics and analysis \textbf{197} (2010),
  no.~3, 713--809.

\bibitem{GKTT2}
Yan Guo, Chanwoo Kim, Daniela Tonon, and Ariane Trescases, \emph{{BV-regularity
  of the Boltzmann equation in non-convex domains}}, Archive for Rational
  Mechanics and Analysis \textbf{220} (2016), no.~3, 1045--1093.

\bibitem{GKTT}
\bysame, \emph{{Regularity of the Boltzmann equation in convex domains}},
  Inventiones mathematicae \textbf{207} (2017), no.~1, 115--290.

\bibitem{K}
Chanwoo Kim, \emph{{Formation and propagation of discontinuity for Boltzmann
  equation in non-convex domains}}, Communications in mathematical physics
  \textbf{308} (2011), no.~3, 641--701.

\bibitem{KB}
RF~Knackfuss and LB~Barichello, \emph{{Surface effects in rarefied gas
  dynamics: an analysis based on the Cercignani--Lampis boundary condition}},
  European Journal of Mechanics-B/Fluids \textbf{25} (2006), no.~1, 113--129.

\bibitem{KB2}
Rosenei~Felippe Knackfuss and Liliane~Basso Barichello, \emph{{On the
  temperature-jump problem in rarefied gas dynamics: the effect of the
  Cercignani--Lampis boundary condition}}, SIAM Journal on Applied Mathematics
  \textbf{66} (2006), no.~6, 2149--2186.

\bibitem{L2}
RG~Lord, \emph{{Some extensions to the Cercignani--Lampis gas--surface
  scattering kernel}}, Physics of Fluids A: Fluid Dynamics \textbf{3} (1991),
  no.~4, 706--710.

\bibitem{L}
\bysame, \emph{{Some further extensions of the Cercignani--Lampis gas--surface
  interaction model}}, Physics of Fluids \textbf{7} (1995), no.~5, 1159--1161.

\bibitem{LS}
Silvia Lorenzani, \emph{{Higher order slip according to the linearized
  Boltzmann equation with general boundary conditions}}, Philosophical
  Transactions of the Royal Society A: Mathematical, Physical and Engineering
  Sciences \textbf{369} (2011), no.~1944, 2228--2236.

\bibitem{M}
St{\'e}phane Mischler, \emph{{Kinetic equations with Maxwell boundary
  conditions}}, Annales Scientifiques de l'Ecole Normale Superieure, vol.~43,
  2010, pp.~719--760.

\bibitem{SF1}
Felix Sharipov, \emph{{Application of the Cercignani--Lampis scattering kernel
  to calculations of rarefied gas flows. I. Plane flow between two parallel
  plates}}, European Journal of Mechanics-B/Fluids \textbf{21} (2002), no.~1,
  113--123.

\bibitem{SF2}
\bysame, \emph{{Application of the Cercignani--Lampis scattering kernel to
  calculations of rarefied gas flows. II. Slip and jump coefficients}},
  European Journal of Mechanics-B/Fluids \textbf{22} (2003), no.~2, 133--143.

\bibitem{SF3}
\bysame, \emph{{Application of the Cercignani--Lampis scattering kernel to
  calculations of rarefied gas flows. III. Poiseuille flow and thermal creep
  through a long tube}}, European Journal of Mechanics-B/Fluids \textbf{22}
  (2003), no.~2, 145--154.

\bibitem{CES}
CE~Siewert, \emph{{Generalized boundary conditions for the S-model kinetic
  equations basic to flow in a plane channel}}, Journal of Quantitative
  Spectroscopy and Radiative Transfer \textbf{72} (2002), no.~1, 75--88.

\bibitem{SI}
\bysame, \emph{{Viscous-slip, thermal-slip, and temperature-jump coefficients
  as defined by the linearized Boltzmann equation and the Cercignani--Lampis
  boundary condition}}, Physics of Fluids \textbf{15} (2003), no.~6,
  1696--1701.

\bibitem{WR}
MS~Woronowicz and DFG Rault, \emph{{Cercignani-lampis-lord gas surface
  interaction model-comparisons between theory and simulation}}, Journal of
  Spacecraft and Rockets \textbf{31} (1994), no.~3, 532--534.

\end{thebibliography}

\end{document}